\documentclass{article}

\usepackage{geometry}
\usepackage{hyperref}
\usepackage{indentfirst}

\usepackage{graphicx,wrapfig,lipsum}
\usepackage{color}
\usepackage{transparent}
\usepackage{subcaption}

\usepackage{amsthm}
\usepackage{amssymb}
\usepackage{amsfonts}
\usepackage{amsmath}
\usepackage{mathtools}
\usepackage{enumitem}

\usepackage{multirow}

\usepackage{adjustbox}
\usepackage{url} 
\usepackage[T1]{fontenc}

\usepackage{float}

\usepackage{tikz} 
\usepackage{tikz-cd}
\usetikzlibrary{positioning}
\usetikzlibrary{arrows}

\usepackage{diagbox}

\usepackage{titlesec}
\titleformat{\subsection}[runin]
  {\normalfont\bfseries}
  {\thesubsection}
  {0.5em}
  {}
  [.]

\usepackage{titlesec}
\titleformat{\subsubsection}[runin]
  {\normalfont\bfseries}
  {\thesubsubsection}
  {0.5em}
  {}
  [.]

\allowdisplaybreaks
\usepackage{mathrsfs}

\usepackage{stmaryrd}
\usepackage{old-arrows}

\usepackage{leftindex}

\usepackage{mathdots}

\DeclareFontFamily{U}{mathx}{}
\DeclareFontShape{U}{mathx}{m}{n}{<-> mathx10}{}
\DeclareSymbolFont{mathx}{U}{mathx}{m}{n}
\DeclareMathAccent{\widehat}{0}{mathx}{"70}
\DeclareMathAccent{\widecheck}{0}{mathx}{"71}

\newcommand{\C}{\mathbb{C}}
\newcommand{\R}{\mathbb{R}}
\newcommand{\Z}{\mathbb{Z}}
\newcommand{\Proj}{\mathbb{P}}
\newcommand{\Ball}{\mathbb{B}}
\newcommand{\Torus}{\mathbb{T}^2}
\newcommand{\Hyperbolic}{\mathbb{H}}

\newcommand{\CP}{\mathbb{CP}}

\newcommand{\RNum}[1]{\lowercase\expandafter{\romannumeral #1\relax}}
\newcommand{\URNum}[1]{\uppercase\expandafter{\romannumeral #1\relax}}

\DeclareMathOperator{\arccot}{arccot}

\theoremstyle{definition}
\newtheorem*{defx*}{Definition}
\newtheorem{definition}{Definition}[section]

\theoremstyle{plain}
\newtheorem{thmx}{Theorem}

\newtheorem{colx}[thmx]{Corollary}
\newtheorem{prox}[thmx]{Proposition}

\newtheorem{theorem}[definition]{Theorem}
\newtheorem{lemma}[definition]{Lemma}
\newtheorem{corollary}[definition]{Corollary}

\newtheorem{proposition}[definition]{Proposition}

\theoremstyle{remark}
\newtheorem{remark}[definition]{Remark}

\DeclareMathOperator{\Real}{Re}
\DeclareMathOperator{\Image}{Im}
\renewcommand{\Re}{\Real}
\renewcommand{\Im}{\Image}
\DeclareMathOperator{\trace}{tr}

\DeclareMathOperator{\SO}{SO}
\DeclareMathOperator{\GL}{GL}
\DeclareMathOperator{\Unitary}{U}

\DeclareMathOperator{\PU}{PU}
\DeclareMathOperator{\SU}{SU}

\DeclareMathOperator{\Isom}{Isom}

\DeclareMathOperator{\sgn}{sgn}

\DeclareMathOperator{\Cyg}{Cyg}

\newcommand{\MM}{\mathcal{M}}
\newcommand{\SSS}{\mathcal{S}}

\newcommand{\hc}{\Hyperbolic_\C^2}
\newcommand{\isphere}{I}

\DeclareMathOperator{\ball}{ball}
\DeclareMathOperator{\Cayley}{Cayley}

\newcommand{\EE}{\tilde{E}}

\usepackage{calc}

\usepackage{accents}
\newcommand{\dbtilde}[1]{\accentset{\approx}{#1}}

\newcommand{\EEE}{\dbtilde{E}}

\usepackage{amsmath}

\usepackage{stmaryrd}
\SetSymbolFont{stmry}{bold}{U}{stmry}{m}{n}

\renewcommand{\bf}{\boldsymbol}

\newcommand{\fp}{p}

\usepackage{listings}

\usepackage{algorithm}
 \usepackage{algpseudocode}
 \usepackage{algorithmicx}
 \algdef{SE}[DOWHILE]{Do}{doWhile}{\algorithmicdo}[1]{\algorithmicwhile\ #1}%

\algrenewcommand\algorithmicrequire{\textbf{Input:}}
\algrenewcommand\algorithmicensure{\textbf{Output:}}



\NewDocumentCommand\vvec{mg}{\overrightarrow{#1}}

\usepackage[backend=biber,style=alphabetic,sorting=nyt,maxbibnames=10,maxalphanames=10,url=true,doi=false,giveninits=true,backref=true]{biblatex}

\renewbibmacro{in:}{}
\DeclareFieldFormat[article]{journaltitle}{\mkbibemph{#1}}

\DefineBibliographyStrings{english}{%
  backrefpage = {},%
  backrefpages = {}%
}

\renewbibmacro*{finentry}{\iflistundef{pageref}{}{}\finentry}
\renewbibmacro*{pageref}{%
  \iflistundef{pageref}
    {}
    {
      \setunit{\adddot\addspace}\printtext{%
        \ifnumgreater{\value{pageref}}{1}
          {\bibstring{backrefpages}\ppspace}
          {\bibstring{backrefpage}\ppspace}%
        \printlist[pageref][-\value{listtotal}]{pageref}\adddot
      }
    }
}

\usepackage{xurl}

\title{Complex Hyperbolic Elliptics\\Preserving Lagrangian Planes}

\author{
 Mengmeng Xu \\
  \small{\emph{School of Mathematical Sciences}} \\
  \small{\emph{Huaqiao University}} \\
  \small{\emph{362021 Quanzhou, China}} \\
  \small{\emph{email:}} \tt{mengmengxu@hqu.edu.cn}
\and 
 Yibo Zhang \\
  \small{\emph{School of Mathematics}} \\
  \small{\emph{Sun Yat-sen University}} \\
  \small{\emph{510275 Guangzhou, China}} \\
  \small{\emph{email:}} \tt{zhangyibo12342000@gmail.com}
}

\date{}

\begin{document}


\maketitle

\begin{abstract}
For a regular, real elliptic isometry \( f \) of finite order in \(\mathbf{H}_{\mathbb{C}}^2\), we determine the explicit cellular structure of the Ford domain of \(\langle f \rangle\) with respect to the extended Cygan metric, provided the centre of the Ford domain lies on the ideal boundary of a Lagrangian plane preserved by \( f \).
As an application, we give a sufficient condition for complex hyperbolic \((n,\infty,\infty)\)-triangle groups to be discrete and faithful that is easy to verify.
\end{abstract}

%
%
\section{Introduction}

An \emph{elliptic} isometry of a Riemannian symmetric space $X$ fixes a point in its interior.
For a non-torsion elliptic isometry, the action on the boundary $\partial X$ yields circular orbits (i.e., having circular closure) for generic points, while degenerate orbits can also exist, which are often key to the full dynamical picture.
In the case where every boundary orbit is a circle, one may ask whether certain distinguished circles among them can characterise the geometry of the elliptic isometry.

In this papar, we study the case when $X = \hc$ is the \emph{complex hyperbolic space} (cf. \cite{Goldman1999,parker-2003-NCHG,Parker-Platis-CHQF}).
This space is the unique simply connected K\"ahler manifold with constant holomorphic section curvature $-1$, 
equipped with the \emph{Bergman metric}, which induces the \emph{Bergman distance}.
The complex hyperbolic space $\hc$ has only two kinds of totally geodesic submanifolds $L\subset \hc$ of real dimension two: \emph{complex geodesics} (isometric to $\Hyperbolic_\C^1$) and \emph{Lagrangian planes} (isometric to $\Hyperbolic_\R^2)$.
Their ideal boundaries $\partial_\infty L \subset \partial \hc$ are called $\C$-circles and $\R$-circles, respectively.
These circles can also be distinguished by the Cartan's angular invariant $\mathbb{A}(p_1,p_2,p_3)$ of three distinct boundary points (cf. \cite[Section 7.1]{Goldman1999}): the angle is $0$ if the points lie on an $\R$-circle and $\pm \pi/2$ if they lie on a $\C$-circle.

The group of holomorphic isometries of $\hc$ is identified with the \emph{projective unitary group} $\PU(2,1)$, defined as the quotient of the group of unitary matrices preserving a Hermitian form of signature $(2,1)$ by its centre.
In the unit ball model of $\hc$, every elliptic isometry in $\PU(2,1)$ is conjugate to one represented by a diagonal matrix of the form $\text{diag}\{e^{i\alpha}, e^{i\beta},1\}$, where $\alpha, \beta \in [0,2\pi)$.
An elliptic isometry in $\PU(2,1)$ is called \emph{regular elliptic} if it has three distinct eigenvalues.
Such an isometry has no fixed point on the boundary $\partial \hc$
and consequently, assuming it is non-torsion, every boundary orbit is a circle.
An elliptic isometry is called \emph{real elliptic} with angle $\theta$ if it is conjugate to the element represented by $\text{diag}\{e^{i\theta}, e^{-i\theta},1\}$.

Real elliptic isometries are characterised as follows: a regular elliptic isometry $f \in \PU(2,1)$ is real elliptic if and only if it preserves a Lagrangian plane $L$.
Any such $L$ must contain the fixed point $\fp_f$ of $f$.
The family $\mathcal{F}$ of all Lagrangian planes satisfying $f(L) = L$ is parameterised by an angle $\theta \in [0, 2\pi)$, i.e., $\mathcal{F} = \{L_{\theta} \mid \theta \in [0,2\pi) \}$.
The ideal boundaries $\partial_\infty L_\theta$ are disjoint $\R$-circles whose union is a torus $\Torus_f \subset \partial \hc$, called the \emph{fixed torus} of $f$.
For a detailed parameterisation when $f$ is represented by $\text{diag}\{e^{i \theta}, e^{-i \theta}, 1\}$, see Definition \ref{definition::standard_fixed_torus} and Figure \ref{figure::torus}.
The fixed torus is also the ideal boundary of an \emph{extor} in the sense of Goldman \cite[Chapter 8]{Goldman1999} given by the pair of positive eigenvectors.
Different regular, real elliptic isometries may share the same fixed torus.
In particular, $\Torus_f = \Torus_{f^k}$ for any positive integer $k$, unless $f^k$ is the identity or fails to be regular.

\subsection{Ford domains of a real elliptic finite cyclic group}  

We study a one-parameter group $G = \{f_{\theta} \mid \theta \in [0, 2\pi) \} \cong \SO(2)$ of real elliptic isometries sharing the same fixed torus.
Consider the associated \emph{isometric spheres} $I(f_{\theta})$ for $\theta \in (0,2\pi)$ about a marked boundary point $q_\infty \in \partial \hc$.
Intuitively, the isometric sphere $I(f_{\theta})$ is the set of points $p \in \overline{\hc}$
that are equidistant from $q_\infty$ and $f_{\theta}^{-1}(q_\infty)$ (see Definition \ref{definition:isometric_sphere}).
The set $I(f_{\theta})$ is homeomorphic to a $3$-ball and contains the fixed point $\fp_f$.
To make this precise,
we identify the boundary $\partial \hc$ with the one-point compactification of the Heisenberg group $\mathcal{N}$ and extend the Cygan metric $d_{\Cyg}$ on $\mathcal{N}$ to $\overline{\hc} \setminus \{q_\infty\}$ (cf. \cite[Section 4.2]{Parker-1992-Shimizu}).
The isometric sphere \(I(f_\theta)\) is also defined as the set of points whose Cygan distance to \(f_\theta^{-1}(q_\infty)\) equals the Cygan distance from \(\fp_f\) to \(f_\theta^{-1}(q_\infty)\).

Let \(\Torus_f\) be the common fixed torus of generic \(f_\theta \in G\).
Assume that \(q_\infty \in \Torus_f\). Then there exists an isometry \(Q\) preserving \(q_\infty\) that sends the isometric sphere \(I(f_\theta)\) to the Cygan sphere centred at \([-\cot(\theta/2),0] \in \mathcal{N}\) with radius \(1/\sin(\theta/2)\) (see Proposition~\ref{proposition::spheres_EEE}). We now describe the intersection pattern of the isometric spheres \(I(f_\theta)\) explicitly.

\begin{prox}
\label{prox::intersection}
Let $G = \{f_{\theta} \mid \theta \in [0,2\pi) \} \le \PU(2,1)$ be a one-parameter subgroup of real elliptic isometries, isomorphic to $\SO(2)$, in which all regular elliptic isometries share the fixed point $\fp_f$ and fixed torus $\Torus_f$.
Suppose that $q_\infty \in \Torus_f$ and
let $I(f_\theta)$ be the isometric sphere of $f_\theta$ about $q_\infty$.
Then, for any $0 < \theta_1 < \theta_2 < \theta_3 < \theta_4 < 2\pi$,
the following patterns hold:
\begin{enumerate}[label=(\arabic*)]
\item The pairwise intersection $I(f_{\theta_1}) \cap I(f_{\theta_2})$ is homeomorphic to a $2$-ball (i.e., a disk).

\item The triple intersection $I(f_{\theta_1}) \cap I(f_{\theta_2}) \cap I(f_{\theta_3})$ is homeomorphic to a crossing.

\item The quadruple intersection  $I(f_{\theta_1}) \cap I(f_{\theta_2}) \cap I(f_{\theta_3}) \cap I(f_{\theta_4})$
consists solely of the point $\fp_f$.
\end{enumerate}
Moreover, for $0 < \theta_1 < \theta_2 < 2\pi$, the punctured disk $I(f_{\theta_1}) \cap I(f_{\theta_2}) \setminus \{\fp_f\}$, after the removal of $12$ disjoint curves connecting $\fp_f$ to the ideal boundary, is foliated by the crossings $I(\theta_1)\cap I(\theta_2)\cap I(\theta
_3)$, where $\theta_3$ ranges over $(0,2\pi) \setminus \{\theta_1,\theta_2\}$.
\end{prox}

\begin{remark}
Here, a \emph{crossing} is defined as the set $\{(x,0),(0,y) \mid |x| \le 1, |y| \le 1\} \subset \R^2$.
\end{remark}

\begin{remark}
The assertion (1) in Proposition~\ref{prox::intersection} is well known. Indeed, the isometric spheres \(I(f_\theta)\) are closures of coequidistant metric bisectors, so their pairwise intersection is connected and its unique connected component is a disk (cf. \cite[Lemmas 9.1.5, 9.1.6 and Theorem 9.2.6]{Goldman1999}), called a \emph{Giraud disk}.
\end{remark}

\begin{remark}
Every isometric sphere \(I(f_\theta)\) extends to an extor in \(\C \Proj^2\). The proof of the assertion (2) in Proposition~\ref{prox::intersection} and Figure~\ref{figure::crossings} suggest that the intersection of such a triple of extors is a figure eight, while a corresponding extension of the foliation should also occur (see Figure \ref{figure::foliations}). We do not pursue this result in this paper.
\end{remark}

For a subgroup $G \le \PU(2,1)$,
the isometric spheres $I(g)$, $g \in G$, form the $3$-dimensional faces of the \emph{Ford domain} of $G$ centred at the marked boundary point $q_\infty \in \partial \hc$.
This domain, denoted by $D_G$, arises as the limit of Dirichlet domains of $G$ when the centre point approaches $q_\infty$ (see Section \ref{section::ford_domain}; cf. \cite[Section 9.3]{Goldman1999}).
From Proposition \ref{prox::intersection}, we deduce the following.

\definecolor{facecentral}{RGB}{232,232,232}
\definecolor{f8fill}{RGB}{219,234,254}
\definecolor{primefill}{RGB}{255,237,213}
\definecolor{unprimefill}{RGB}{220,252,231}
\definecolor{cutred}{RGB}{190,35,35}

\begin{figure}[ht]
\begin{center}
\begin{tikzpicture}[x=0.95cm,y=0.95cm,line join=round,line cap=round]
  \tikzset{
    edge label/.style={font=\tiny,inner sep=0.40pt,fill opacity=0.91,text opacity=1,rounded corners=0.35pt},
    face label/.style={font=\scriptsize\bfseries,inner sep=0pt},
    big face label/.style={font=\Large\bfseries,inner sep=0pt}
  }

  \path[fill=f8fill] (2.9248,0.6676) -- (2.9248,-0.6676) -- (3.5041,-1.8705) -- (4.5479,-2.7029) -- (5.8496,-3.0000) -- (7.1512,-2.7029) -- (8.1951,-1.8705) -- (8.7744,-0.6676) -- (8.7744,0.6676) -- (8.1951,1.8705) -- (7.1512,2.7029) -- (5.8496,3.0000) -- (4.5479,2.7029) -- (3.5041,1.8705) -- cycle; 
  \path[fill=facecentral] (2.9248,-0.6676) -- (2.9248,0.6676) -- (2.3455,1.8705) -- (1.3017,2.7029) -- (0.0000,3.0000) -- (-1.3017,2.7029) -- (-2.3455,1.8705) -- (-2.9248,0.6676) -- (-2.9248,-0.6676) -- (-2.3455,-1.8705) -- (-1.3017,-2.7029) -- (-0.0000,-3.0000) -- (1.3017,-2.7029) -- (2.3455,-1.8705) -- cycle; 
  \path[fill=primefill] (2.3455,1.8705) -- (2.9248,0.6676) -- (3.5041,1.8705) -- cycle; 
  \path[fill=unprimefill] (2.9248,-0.6676) -- (2.3455,-1.8705) -- (3.5041,-1.8705) -- cycle; 
  \path[fill=primefill] (2.3455,1.8705) -- (3.5041,1.8705) -- (3.5041,3.0290) -- (2.3455,3.0290) -- cycle; 
  \path[fill=primefill] (2.3455,3.0290) -- (3.5041,3.0290) -- (3.5041,4.1876) -- (2.3455,4.1876) -- cycle; 
  \path[fill=primefill] (2.3455,4.1876) -- (3.5041,4.1876) -- (3.5041,5.3462) -- (2.3455,5.3462) -- cycle; 
  \path[fill=primefill] (2.3455,5.3462) -- (3.5041,5.3462) -- (3.5041,6.5048) -- (2.3455,6.5048) -- cycle; 
  \path[fill=primefill] (2.3455,6.5048) -- (3.5041,6.5048) -- (2.9248,7.5081) -- cycle; 
  \path[fill=unprimefill] (3.5041,-1.8705) -- (2.3455,-1.8705) -- (2.3455,-3.0290) -- (3.5041,-3.0290) -- cycle; 
  \path[fill=unprimefill] (3.5041,-3.0290) -- (2.3455,-3.0290) -- (2.3455,-4.1876) -- (3.5041,-4.1876) -- cycle; 
  \path[fill=unprimefill] (3.5041,-4.1876) -- (2.3455,-4.1876) -- (2.3455,-5.3462) -- (3.5041,-5.3462) -- cycle; 
  \path[fill=unprimefill] (3.5041,-5.3462) -- (2.3455,-5.3462) -- (2.3455,-6.5048) -- (3.5041,-6.5048) -- cycle; 
  \path[fill=unprimefill] (3.5041,-6.5048) -- (2.3455,-6.5048) -- (2.9248,-7.5081) -- cycle; 

  \draw[black,line width=0.76pt] (2.9248,0.6676) -- (2.9248,-0.6676); 
  \draw[black,line width=0.76pt] (2.9248,-0.6676) -- (3.5041,-1.8705); 
  \draw[black,line width=0.76pt] (3.5041,1.8705) -- (2.9248,0.6676); 
  \draw[black,line width=0.76pt] (2.9248,0.6676) -- (2.3455,1.8705); 
  \draw[black,line width=0.76pt] (2.3455,-1.8705) -- (2.9248,-0.6676); 
  \draw[black,line width=0.76pt] (3.5041,1.8705) -- (2.3455,1.8705); 
  \draw[black,line width=0.76pt] (2.3455,-1.8705) -- (3.5041,-1.8705); 
  \draw[black,line width=0.76pt] (3.5041,3.0290) -- (2.3455,3.0290); 
  \draw[black,line width=0.76pt] (3.5041,4.1876) -- (2.3455,4.1876); 
  \draw[black,line width=0.76pt] (3.5041,5.3462) -- (2.3455,5.3462); 
  \draw[black,line width=0.76pt] (3.5041,6.5048) -- (2.3455,6.5048); 
  \draw[black,line width=0.76pt] (2.3455,-3.0290) -- (3.5041,-3.0290); 
  \draw[black,line width=0.76pt] (2.3455,-4.1876) -- (3.5041,-4.1876); 
  \draw[black,line width=0.76pt] (2.3455,-5.3462) -- (3.5041,-5.3462); 
  \draw[black,line width=0.76pt] (2.3455,-6.5048) -- (3.5041,-6.5048); 

  \draw[black,line width=0.70pt] (3.5041,-1.8705) -- (4.5479,-2.7029); 
  \draw[black,line width=0.70pt] (4.5479,-2.7029) -- (5.8496,-3.0000); 
  \draw[black,line width=0.70pt] (5.8496,-3.0000) -- (7.1512,-2.7029); 
  \draw[black,line width=0.70pt] (7.1512,-2.7029) -- (8.1951,-1.8705); 
  \draw[black,line width=0.70pt] (8.1951,-1.8705) -- (8.7744,-0.6676); 
  \draw[black,line width=0.70pt] (8.7744,-0.6676) -- (8.7744,0.6676); 
  \draw[black,line width=0.70pt] (8.7744,0.6676) -- (8.1951,1.8705); 
  \draw[black,line width=0.70pt] (8.1951,1.8705) -- (7.1512,2.7029); 
  \draw[black,line width=0.70pt] (7.1512,2.7029) -- (5.8496,3.0000); 
  \draw[black,line width=0.70pt] (5.8496,3.0000) -- (4.5479,2.7029); 
  \draw[black,line width=0.70pt] (4.5479,2.7029) -- (3.5041,1.8705); 
  \draw[black,line width=0.70pt] (2.3455,1.8705) -- (1.3017,2.7029); 
  \draw[black,line width=0.70pt] (1.3017,2.7029) -- (0.0000,3.0000); 
  \draw[black,line width=0.70pt] (0.0000,3.0000) -- (-1.3017,2.7029); 
  \draw[black,line width=0.70pt] (-1.3017,2.7029) -- (-2.3455,1.8705); 
  \draw[black,line width=0.70pt] (-2.3455,1.8705) -- (-2.9248,0.6676); 
  \draw[black,line width=0.70pt] (-2.9248,0.6676) -- (-2.9248,-0.6676); 
  \draw[black,line width=0.70pt] (-2.9248,-0.6676) -- (-2.3455,-1.8705); 
  \draw[black,line width=0.70pt] (-2.3455,-1.8705) -- (-1.3017,-2.7029); 
  \draw[black,line width=0.70pt] (-1.3017,-2.7029) -- (-0.0000,-3.0000); 
  \draw[black,line width=0.70pt] (-0.0000,-3.0000) -- (1.3017,-2.7029); 
  \draw[black,line width=0.70pt] (1.3017,-2.7029) -- (2.3455,-1.8705); 
  \draw[black,line width=0.70pt] (3.5041,1.8705) -- (3.5041,3.0290); 
  \draw[black,line width=0.70pt] (2.3455,3.0290) -- (2.3455,1.8705); 
  \draw[black,line width=0.70pt] (3.5041,3.0290) -- (3.5041,4.1876); 
  \draw[black,line width=0.70pt] (2.3455,4.1876) -- (2.3455,3.0290); 
  \draw[black,line width=0.70pt] (3.5041,4.1876) -- (3.5041,5.3462); 
  \draw[black,line width=0.70pt] (2.3455,5.3462) -- (2.3455,4.1876); 
  \draw[black,line width=0.70pt] (3.5041,5.3462) -- (3.5041,6.5048); 
  \draw[black,line width=0.70pt] (2.3455,6.5048) -- (2.3455,5.3462); 
  \draw[black,line width=0.70pt] (3.5041,6.5048) -- (2.9248,7.5081); 
  \draw[black,line width=0.70pt] (2.9248,7.5081) -- (2.3455,6.5048); 
  \draw[black,line width=0.70pt] (2.3455,-1.8705) -- (2.3455,-3.0290); 
  \draw[black,line width=0.70pt] (3.5041,-3.0290) -- (3.5041,-1.8705); 
  \draw[black,line width=0.70pt] (2.3455,-3.0290) -- (2.3455,-4.1876); 
  \draw[black,line width=0.70pt] (3.5041,-4.1876) -- (3.5041,-3.0290); 
  \draw[black,line width=0.70pt] (2.3455,-4.1876) -- (2.3455,-5.3462); 
  \draw[black,line width=0.70pt] (3.5041,-5.3462) -- (3.5041,-4.1876); 
  \draw[black,line width=0.70pt] (2.3455,-5.3462) -- (2.3455,-6.5048); 
  \draw[black,line width=0.70pt] (3.5041,-6.5048) -- (3.5041,-5.3462); 
  \draw[black,line width=0.70pt] (2.3455,-6.5048) -- (2.9248,-7.5081); 
  \draw[black,line width=0.70pt] (2.9248,-7.5081) -- (3.5041,-6.5048); 

  \draw[cutred,dash pattern=on 3.0pt off 1.8pt,line width=1.0pt] (8.7744,-0.6676) -- (8.7744,0.6676);
  \draw[cutred,dash pattern=on 3.0pt off 1.8pt,line width=1.0pt] (-2.9248,0.6676) -- (-2.9248,-0.6676);

  \node[big face label] at (5.8496,0.0000) {$f_8$};
  \node[big face label] at (-0.0000,0.0000) {$f_1$};
  \node[face label] at (2.9248,1.4695) {$f'_2$};
  \node[face label] at (2.9248,-1.4695) {$f_7$};
  \node[face label] at (2.9248,2.4498) {$f'_3$};
  \node[face label] at (2.9248,3.6083) {$f'_4$};
  \node[face label] at (2.9248,4.7669) {$f'_5$};
  \node[face label] at (2.9248,5.9255) {$f'_6$};
  \node[face label] at (2.9248,6.9392) {$f'_7$};
  \node[face label] at (2.9248,-2.4498) {$f_6$};
  \node[face label] at (2.9248,-3.6083) {$f_5$};
  \node[face label] at (2.9248,-4.7669) {$f_4$};
  \node[face label] at (2.9248,-5.9255) {$f_3$};
  \node[face label] at (2.9248,-6.9392) {$f_2$};

  \node[edge label,rotate=-90.00] at (2.8098,0.0000) {$e_{1,8}$};
  \node[edge label,rotate=-64.29] at (3.3946,-1.1822) {$e_{8,7}$};
  \node[edge label,rotate=-38.57] at (3.9138,-2.4274) {$e_{8,6}$};
  \node[edge label,rotate=-12.86] at (5.1587,-3.0269) {$e_{8,5}$};
  \node[edge label,rotate=12.86] at (6.5404,-3.0269) {$e_{8,4}$};
  \node[edge label,rotate=38.57] at (7.7854,-2.4274) {$e_{8,3}$};
  \node[edge label,rotate=64.29] at (8.6469,-1.3471) {$e_{8,2}$};
  \node[edge label,rotate=-90.00] at (8.9544,0.0000) {$e'_{1,8}$};
  \node[edge label,rotate=-64.29] at (8.6469,1.3471) {$e'_{8,7}$};
  \node[edge label,rotate=-38.57] at (7.7854,2.4274) {$e'_{8,6}$};
  \node[edge label,rotate=-12.86] at (6.5404,3.0269) {$e'_{8,5}$};
  \node[edge label,rotate=12.86] at (5.1587,3.0269) {$e'_{8,4}$};
  \node[edge label,rotate=38.57] at (3.9138,2.4274) {$e'_{8,3}$};
  \node[edge label,rotate=64.29] at (3.3946,1.1822) {$e'_{8,2}$};
  \node[edge label,rotate=-64.29] at (2.4549,1.1822) {$e'_{1,2}$};
  \node[edge label,rotate=-38.57] at (1.9202,2.4079) {$e'_{1,3}$};
  \node[edge label,rotate=-12.86] at (0.6853,3.0026) {$e'_{1,4}$};
  \node[edge label,rotate=12.86] at (-0.6853,3.0026) {$e'_{1,5}$};
  \node[edge label,rotate=38.57] at (-1.9202,2.4079) {$e'_{1,6}$};
  \node[edge label,rotate=64.29] at (-2.7748,1.3363) {$e'_{1,7}$};
  \node[edge label,rotate=-90.00] at (-3.0798,0.0000) {$e'_{1,8}$};
  \node[edge label,rotate=-64.29] at (-2.7748,-1.3363) {$e_{1,2}$};
  \node[edge label,rotate=-38.57] at (-1.9202,-2.4079) {$e_{1,3}$};
  \node[edge label,rotate=-12.86] at (-0.6853,-3.0026) {$e_{1,4}$};
  \node[edge label,rotate=12.86] at (0.6853,-3.0026) {$e_{1,5}$};
  \node[edge label,rotate=38.57] at (1.9202,-2.4079) {$e_{1,6}$};
  \node[edge label,rotate=64.29] at (2.4549,-1.1822) {$e_{1,7}$};
  \node[edge label,rotate=0.00] at (2.9248,1.9855) {$e'_{2,3}$};
  \node[edge label,rotate=-0.00] at (2.9248,-1.9855) {$e_{6,7}$};
  \node[edge label,rotate=-90.00] at (3.6591,2.6498) {$e'_{8,3}$};
  \node[edge label,rotate=0.00] at (2.9248,3.1440) {$e'_{3,4}$};
  \node[edge label,rotate=-90.00] at (2.1905,2.6498) {$e'_{1,3}$};
  \node[edge label,rotate=-90.00] at (3.6591,3.6083) {$e'_{8,4}$};
  \node[edge label,rotate=0.00] at (2.9248,4.3026) {$e'_{4,5}$};
  \node[edge label,rotate=-90.00] at (2.1905,3.6083) {$e'_{1,4}$};
  \node[edge label,rotate=-90.00] at (3.6591,4.7669) {$e'_{8,5}$};
  \node[edge label,rotate=0.00] at (2.9248,5.4612) {$e'_{5,6}$};
  \node[edge label,rotate=-90.00] at (2.1905,4.7669) {$e'_{1,5}$};
  \node[edge label,rotate=-90.00] at (3.6591,5.9255) {$e'_{8,6}$};
  \node[edge label,rotate=0.00] at (2.9248,6.6198) {$e'_{6,7}$};
  \node[edge label,rotate=-90.00] at (2.1905,5.9255) {$e'_{1,6}$};
  \node[edge label,rotate=-60.00] at (3.3487,7.0840) {$e'_{8,7}$};
  \node[edge label,rotate=60.00] at (2.5009,7.0840) {$e'_{1,7}$};
  \node[edge label,rotate=90.00] at (2.1905,-2.6498) {$e_{1,6}$};
  \node[edge label,rotate=-0.00] at (2.9248,-3.1440) {$e_{5,6}$};
  \node[edge label,rotate=90.00] at (3.6591,-2.6498) {$e_{8,6}$};
  \node[edge label,rotate=90.00] at (2.1905,-3.6083) {$e_{1,5}$};
  \node[edge label,rotate=-0.00] at (2.9248,-4.3026) {$e_{4,5}$};
  \node[edge label,rotate=90.00] at (3.6591,-3.6083) {$e_{8,5}$};
  \node[edge label,rotate=90.00] at (2.1905,-4.7669) {$e_{1,4}$};
  \node[edge label,rotate=-0.00] at (2.9248,-5.4612) {$e_{3,4}$};
  \node[edge label,rotate=90.00] at (3.6591,-4.7669) {$e_{8,4}$};
  \node[edge label,rotate=90.00] at (2.1905,-5.9255) {$e_{1,3}$};
  \node[edge label,rotate=-0.00] at (2.9248,-6.6198) {$e_{2,3}$};
  \node[edge label,rotate=90.00] at (3.6591,-5.9255) {$e_{8,3}$};
  \node[edge label,rotate=-60.00] at (2.5009,-7.0840) {$e_{1,2}$};
  \node[edge label,rotate=60.00] at (3.3487,-7.0840) {$e_{8,2}$};
\end{tikzpicture}
\end{center}
\caption{Unfolded cellular structure on $\partial \partial_{\infty} D_n$ for $n=9$.}
\label{Figure::ford_domain}
\end{figure}

\begin{thmx}
\label{thmx::ford}
Let $f \in \PU(2,1)$ be a regular, real elliptic isometry of order $n \ge 3$.  
Suppose that $q_\infty \in \Torus_f$.  
Then the Ford domain $D_n = D_{\langle f \rangle}$ has the following cellular structure, depending only on $n$.
\begin{enumerate}[label=(\arabic*)]
\item The boundary of the ideal boundary $\partial\partial_{\infty} D_n$ is a $2$-dimensional cellular complex homeomorphic to the $2$-sphere, obtained by gluing $2(n-2)$ polygons.

\item The closure of $\partial D_n$ in $\overline{\hc}$ is a $3$-dimensional cellular complex homeomorphic to the closed $3$-ball, obtained by taking the cone of $\partial\partial_{\infty} D_n$ whose cone point is the fixed point $\fp_f$.
\end{enumerate}
\end{thmx}

\begin{remark}
The explicit description for the cellular structure of $\partial \partial_{\infty} D$ is given in Theorem \ref{theorem::ford_domain}.
When $n = 9$, this cellular structure is shown in Figure \ref{Figure::ford_domain}.
\end{remark}

\begin{remark}
Theorem \ref{thmx::ford} is optimal: if we omit the hypothesis that $q_\infty \in \Torus_f$, the Ford domain $D_{\langle f\rangle}$ can have a different cellular structure.
Specifically, as shown in Subsection \ref{subsection::Ford_centred_away},
the triple intersection from Proposition \ref{prox::intersection}
is not necessarily a crossing when $q_\infty \notin \Torus_f$.
\end{remark}

\subsection{Complex hyperbolic triangle groups}

Ford domains are widely used in complex hyperbolic geometry, particularly in the study of complex hyperbolic triangle groups.
We give a brief review below, but refer the reader to Richard Schwartz's survey \cite{Schwartz2002} for further motivation and detail.

Recall that the $(p,q,r)$-triangle group has the presentation
\[
T_{p,q,r} = \langle \sigma_1, \sigma_2, \sigma_3 \ \mid \ \sigma_1^2 = \sigma_2^2 = \sigma_3^2 = (\sigma_2 \sigma_3)^p = (\sigma_3 \sigma_1)^q = (\sigma_1 \sigma_2)^r = 1\rangle,
\]
where $p,q,r \in \Z_{\ge 3} \sqcup \{\infty\}$ and $p \le q \le r$.
If one of $p,q,r$ is $\infty$, 
the corresponding relation is omitted.
A \emph{complex hyperbolic $(p,q,r)$-triangle group} is the image of a representation $\rho: T_{p,q,r} \rightarrow \PU(2,1)$ 
that sends each $\sigma_i$ to a complex reflection $I_i$ about a complex geodesic $C_i$ (see Definition \ref{definition::complex_reflection}).
For each relation $(\sigma_i \sigma_j)^n = 1$, we require the complex geodesics $C_i$ and $C_j$ to meet at an angle $\pi/n$.
We say the group $\Delta(p,q,r) \coloneqq \rho(T_{p,q,r})$ is \emph{discrete and faithful} if the representation $\rho$ is.

We study complex hyperbolic $(n,\infty,\infty)$-triangle groups for $n \ge 3$.
In this setting, the isometry $B \coloneqq I_2 I_3$ is known to be a real elliptic isometry of order $n$ (see \cite[Proposition 4]{Paupert-Will-2017-involution}).
This allows us to apply Theorem \ref{thmx::ford}
to explicitly construct a Ford domain of the index-$2$ subgroup $\langle A\coloneqq I_1 I_2, B \rangle \le \Delta(n,\infty,\infty)$.
This construction, justified by the Poincar\'e polyhedron theorem, is centred at the boundary fixed point $\fp_A \in \partial \hc$ of $A$,
which we identify with $q_\infty$.
A key observation is that $\fp_A \in \Torus_B$.
Therefore, the point $\fp_A$
lies in the fixed torus of every isometry of the form $A^{k}B^j A^{-k}$, where $k \in \Z$ and $j \in \{1,\cdots,n-1\}$.
This leads to the following sufficient condition for discreteness and faithfulness.

\begin{colx}
\label{colx::triangle_group}
Let $\Delta(n,\infty,\infty)$ be a complex hyperbolic $(n,\infty,\infty)$-triangle group with $n \ge 3$. Then $\Delta(n,\infty,\infty)$ is discrete and faithful provided that in each of $O(n^3)$ specified pairs the two Cygan spheres are disjoint, and in one particular pair they are tangent at a single ideal point.
\end{colx}

\subsection*{Outline}

Section~\ref{section::preliminary} introduces the complex hyperbolic space $\hc$, the Cygan metric and the fixed torus $\Torus_f$ of a regular, real elliptic isometry.
Section~\ref{section::sphere} studies the family of isometric spheres $\{I(f_\theta) \mid \theta \in (0, 2\pi)\}$ and proves Proposition~\ref{prox::intersection}.
Section~\ref{section::ford_domain} describes the cellular structure of the Ford domain $D_n$ for $\langle f_{2\pi/n}\rangle$, i.e., Theorem~\ref{thmx::ford}.
Section~\ref{section::triangle} investigates complex hyperbolic $(n,\infty,\infty)$-triangle groups and applies the Poincaré polyhedron theorem to obtain Corollary~\ref{colx::triangle_group}.

\subsection*{Acknowledgments}

This work was completed while the second author was visiting Huaqiao University, and he is grateful for the hospitality of the first author.
We would also like to thank Yueping Jiang, Louis Funar, Greg McShane, Martin Deraux, Baohua Xie, Jieyan Wang and Pierre Will for valuable discussions related to some of the topics addressed in this paper.
We thank an anonymous referee for constructive suggestions concerning real elliptic isometries and their fixed tori, which greatly improved the paper.

\tableofcontents

\section{Geometric preliminaries}
\label{section::preliminary}

\subsection{Complex hyperbolic plane}

The \emph{complex hyperbolic plane} $\hc$ is the unique simply connected K\"{a}hler manifold with constant holomorphic sectional curvature $-1$.
This subsection introduces the unit ball and Siegel domain models of $\hc$, its holomorphic isometries, and their classification.

\subsubsection*{Projective models}

Let $\C^{2,1}_H$ denote $\C^3$ equipped with a non-degenerate, indefinite Hermitian form $\langle {\bf{z}}, {\bf{w}} \rangle_H ={\bf{w}^{*}}H{\bf{z}}$ of signature $(2,1)$,
where $H$ is a Hermitian matrix.
We partition $\C^{2,1}_H \setminus \{0\}$ into the disjoint union of the following cones:
\begin{align*}
V_H^- &=
\left\{ \bf{p} \in \C^{2,1}_H
\middle| \langle \bf{p}, \bf{p} \rangle_H < 0 \right\}, \\
V_H^0 &=
\left\{ \bf{p} \in \C^{2,1}_H \setminus \{0\}
\middle| \langle \bf{p}, \bf{p} \rangle_H = 0 \right\}, \\
V_H^+ &=
\left\{ \bf{p} \in \C^{2,1}_H 
\middle| \langle \bf{p}, \bf{p} \rangle_H > 0 \right\}.
\end{align*}
These are called the \emph{negative}, \emph{null} and \emph{positive cones}, respectively.

Let $\Proj: \C^3 \rightarrow \CP^2$ denote the projectivization map.
We adopt the convention that a point $\bf{p} \in \C^3$ is a lift of $p \in \CP^2$
and we use $(z_1,z_2,z_3)^t$ for a vector in $\C^3$ but use $[z_1,z_2,z_3]^t$ for the corresponding equivalence class in $\CP^2$.

The \emph{projective model} of the complex hyperbolic space $\hc$ is defined as $\Proj(V_H^-)$, with boundary $\Proj(V_H^0)$.
The closure $\overline{\hc} \coloneqq \hc \sqcup \partial \hc$ is thus described by $\Proj(V_{H}^- \sqcup V_H^+)$ in this model.
For any subset $U \subset \overline{\hc}$, let $\overline{U}$ denote its closure in $\overline{\hc}$ and define its \emph{ideal boundary} as $\partial_\infty U = \overline{U} \cap \partial \hc$.

Two standard Hermitian forms on $\C^3$ of signature $(2,1)$ are given by the matrices:
\begin{equation*}
H_1 = \left(
\begin{array}{ccc}
1 & 0 & 0 \\
0 & 1 & 0 \\
0 & 0 & -1 \\
\end{array}
\right),
\quad
H_2 = \left(
\begin{array}{ccc}
0 & 0 & 1 \\
0 & 1 & 0 \\
1 & 0 & 0 \\
\end{array}
\right).
\end{equation*}

The points in $\Proj(V_{H_1}^-)$ are precisely those of the form $[w_1,w_2,1]^t$ with $|w_1|^2 + |w_2|^2 < 1$, forming the \emph{unit ball model} of $\hc$.
Its boundary $\Proj(V_{H_1}^0)$ consists of points $[w_1,w_2,1]^t$ with $|w_1|^2 + |w_2|^2 = 1$.
For every $p = [w_1,w_2,1]^t \in \overline{\hc}$ in the unit ball model, the vector $\bf{p} = (w_1,w_2,1)^t$ is its \emph{standard lift}.

The points in $\Proj(V_{H_2}^-)$ are precisely those of the form $[z_1,z_2,1]^t$ with $2 \Re(z_1) + |z_2|^2 < 0$, forming the \emph{Siegel domain model} of $\hc$.
Its boundary $\Proj(V_{H_2}^0)$ consists of points $[z_1,z_2,1]^t$ with $2 \Re(z_1) + |z_2|^2 = 0$ and one exceptional point given by $q_\infty = [1,0,0]^t$.
We again say the \emph{standard lift} of $p = [z_1,z_2,1]^t \in \overline{\hc}$ in the Siegel domain is the vector $\bf{p} = (z_1, z_2, 1)^t$, but the standard lift of $q_\infty$ is $\bf{q}_\infty = (1, 0, 0)^t$.

One can pass between the unit ball model $\Proj(V_{H_1}^-)$ and the Siegal domain model $\Proj(V_{H_2}^-)$ using the following Cayley transform:
\begin{equation*}
\Cayley = \left(
\begin{array}{ccc}
1 / \sqrt{2} & 0 & 1 / \sqrt{2} \\
0 & 1 & 0 \\
1 / \sqrt{2} & 0 & -1 / \sqrt{2} \\
\end{array}
\right)
\end{equation*}
Note that $\Cayley^{-1} = \Cayley$.


%
%
\subsubsection*{Bergman distance}

The complex hyperbolic space $\hc$ is endowed with the \emph{Bergman metric}, whose associated distance $d_{H_i}(\cdot,\cdot)$ is defined by
\(
\cosh^2\left(\dfrac{d_{H_i}(u,v)}{2}\right)=
\dfrac{\langle {\bf{u}}, {\bf{v}} \rangle_{H_i} \langle {\bf{v}}, {\bf{u}} \rangle_{H_i}}{\langle {\bf{u}}, {\bf{u}} \rangle_{H_i} \langle {\bf{v}}, {\bf{v}} \rangle_{H_i}}
\),
where $\bf{u}$ and $\bf{v} \in \C_{H_i}^{2,1}$ are lifts of the points $u$ and $v$.


%
%
\subsubsection*{Holomorphic isometries}

For $i=1,2$, define the unitary group
\[
\Unitary(H_i) = \left\{ A \in \GL(3,\mathbb{C})
\ \middle| \
A^{*}H_i A = H_i \right\}.
\]
The full group of holomorphic isometries of $\hc$, identified with $\Isom(\hc)$,
is the projective unitary group $\PU(H_i) \coloneqq \Unitary(H_i) / \Unitary(1)$.
We work with the special unitary group $\SU(H_i)$, consisting of matrices in $\Unitary(H_i)$ with determinant $1$.
Consequently, we have $\PU(H_i) = \SU(H_i) / \{ I, \omega I, \omega^2 I \}$,
where $\omega = (-1 + i\sqrt{3}) / 2$ is a primitive cube root of unity.
We adopt the notation where parentheses $(\cdots)$ denote a matrix in $\Unitary(H_i)$ or $\SU(H_i)$, while square brackets $[\cdots]$ denote its projective equivalence class in $\PU(H_i)$.

Elements of $\PU(H_i)$ are classified into three types based on their fixed points:
a \emph{loxodromic} isometry fixes exactly two points on the boundary $\partial \hc$;
a \emph{parabolic} isometry fixes exactly one point on $\partial \hc$;
an \emph{elliptic} isometry fixes at least one point in the interior $\hc$.
This classification can be further refined.
A parabolic element is \emph{unipotent} if it is represented by a unipotent matrix in $\SU(H_i)$;
otherwise it is \emph{screw-parabolic}.
An elliptic element is \emph{regular} if all its eigenvalues are distinct;
otherwise it is called \emph{special}.

In the unit ball model, every elliptic isometry is conjugate to
$E_{\alpha,\beta}^{\ball} = \text{diag}\{e^{i \alpha}, e^{i \beta}, 1\}$ for some $\alpha,\beta \in [0, 2\pi]$.
The corresponding element in the Siegel domain model, obtained via the Cayley transform, is $E_{\alpha,\beta}$, as the following:
\begin{equation*}
E_{\alpha,\beta} = \dfrac{1}{2}
\left[
\begin{array}{ccc}
e^{i (2\alpha-\beta) / 3} + e^{-i (\alpha+\beta) / 3}
& 0 &
e^{i (2\alpha-\beta) / 3} - e^{-i (\alpha+\beta) / 3}
\\
0 & 2 e^{i (2\beta - \alpha) / 3} & 0 \\
e^{i (2\alpha-\beta) / 3} - e^{-i (\alpha+\beta) / 3}
& 0 &
e^{i (2\alpha-\beta) / 3} + e^{-i (\alpha+\beta) / 3}
\\
\end{array}
\right].
\end{equation*}
The fixed point of $E_{\alpha,\beta}$ is $\fp_E = [-1,0,1]^t$,
which is independent of the parameters $(\alpha,\beta)$.
Such an isometry is regular if and only if $0 < \alpha \neq \beta < 2\pi$.
It is called \emph{real elliptic} if $\beta = 2\pi - \alpha$.
For convenience, we use 
$E^{\text{ball}}_{\alpha,-\alpha}$
and
$E_{\alpha,-\alpha}$
to denote the real elliptic isometries $E^{\text{ball}}_{\alpha,2\pi-\alpha}$ and $E_{\alpha,2\pi-\alpha}$, respectively.

\begin{proposition}[Theorem 6.2.4 in \cite{Goldman1999}]
\label{proposition::trace-polynomial}
Using $f(z) = |z|^4 - 8 \Real(z^3) + 18 |z|^2 - 27$, we get
\begin{enumerate}[label=(\roman*)]
\item $A \in \SU(H_i)$ represents a loxodromic isometry if and only if $f(\trace A) > 0$;

\item $A \in \SU(H_i)$ represents a regular elliptic isometry if and only if $f(\trace A) < 0$.
\end{enumerate}
\end{proposition}


%
%
\subsubsection*{Totally geodesic subspaces}

There are five types of totally geodesic submanifolds  
in $\hc$: points, real geodesics, \emph{complex geodesics} (isometric to $\Hyperbolic_\C^1$), \emph{Lagrangian planes} (isometric to $\Hyperbolic_\R^2$) and $\hc$ itself.
Notably, there is no totally geodesic hypersurfaces in $\hc$.
The ideal boundary of a complex geodesic is called a $\C$-circle, while the ideal boundary of a Lagrangian plane is called an $\R$-circle.

Let $C$ be a complex geodesic.
Its \emph{polar vector}, denoted by $n_C$, is the unique positive vector
with respect to the Hermitian form
orthogonal to $C$, 
determined up to positive scalar multiplication.
The following defines a special elliptic isometry in $\PU(H_i)$ associated with $C$.

\begin{definition}
\label{definition::complex_reflection}
The \emph{complex reflection} (of order $2$) in a complex geodesic $C$ is the holomorphic isometry defined by
$R_C(p) = -\bf{p} + 2 \dfrac{\langle\bf{p},n_C\rangle_{H_i}}{\langle n_C,n_C\rangle_{H_i}} n_C$,
where $\bf{p}$ is any lift of $p \in \overline{\hc}$.
The complex geodesic $C$ is called the \emph{mirror} of $R_C$.
\end{definition}

Given a triple of distinct points $p_1, p_2, p_3 \in \partial \hc$,
the \emph{Cartan’s angular invariant} is defined as
\begin{equation*}
\mathbb{A}(p_1, p_2, p_3) \coloneqq
\arg \left(-
\langle \bf{p}_1,\bf{p}_2 \rangle_{H_i}
\langle \bf{p}_2,\bf{p}_3 \rangle_{H_i}
\langle \bf{p}_3,\bf{p}_1 \rangle_{H_i}
\right),
\end{equation*}
where $\bf{p}_i$ are any lifts of $p_i$ and the argument is taken in $(-\pi, \pi]$.
This invariant always takes a value in $[-\pi/2,\pi/2]$
and it characterises the configuration of the points as follows:
$p_1$, $p_2$, $p_3$ lie on a complex geodesic if and only if $\mathbb{A}(p_1, p_2, p_3) = \pm \pi / 2$;
$p_1$, $p_2$, $p_3$ lie on a Lagrangian plane if and only if $\mathbb{A}(p_1, p_2, p_3) = 0$.

\subsection{Boundary}
\label{subsection::boundary}

We identify $\partial \hc$ with the one-point compactification of the Heisenberg group.

\subsubsection*{Heisenberg group}

Let $\mathcal{N}$ be the \emph{Heisenberg group}, a nilpotent Lie group with underlying manifold $\C \times \R$ and coordinates $[z,t]$,
equipped with the group law
\[
[z,t] \cdot [z',t'] = [z+z',t+t'+2\Im (z \overline{z'})].
\]
The closure $\overline{\hc}$ is identified with ${\mathcal{N}}\times{\mathbb{R}_{\geq0}}\sqcup \{q_{\infty}\}$.
A point $q=[z,t,u] \in {\mathcal{N}}\times{\mathbb{R}_{\geq0}}$ corresponds to a point $q \in \overline{\hc}$ in the Siegel domain model with standard lift
\(
{\bf{q}} = \left((-|z|^2-u+it)/2, z, 1\right)^t.
\)
We refer to $[z,t,u]$ as the \emph{horospherical coordinates} of $\overline {\hc}$. 
When $u=0$, the point $[z,t,0]$ lies on $\partial\hc$,
allowing us to identify
$\partial\hc$ with $\mathcal{N}\sqcup \{q_{\infty}\}$. 
For every $u_0 >0$, the level set $\{u=u_0\}$ is called a \emph{horosphere} based at $q_{\infty}$, the super-level set $\{u\ge u_0\}$ is called a \emph{horoball} based at $q_{\infty}$.

The stabilizer subgroup of $q_\infty$ in $\PU(H_2)$ is generated by the following matrices:
\begin{equation*}
T_{[z,t]} =
\begin{bmatrix}
1 & -\overline{z}& \frac{-|z|^{2}+it}{2} \\ 0 & 1 & z \\ 0 & 0 & 1 \end{bmatrix}, \ 
R_{\theta} =
\begin{bmatrix}
1 & 0& 0 \\ 0 & e^{i\theta} & 0 \\ 0 & 0 & 1 \end{bmatrix}
\ \text{and} \
D_{\lambda} =
\begin{bmatrix}
\lambda & 0 & 0 \\
0 & 1 & 0 \\
0 & 0 & 1/\lambda 
\end{bmatrix},
\end{equation*}
where $[z,t]\in\mathcal{N}$, $\theta \in [0,2\pi)$ and $\lambda \in \mathbb{R}\setminus\{0\}$.
These are called the
\emph{Heisenberg translation},
\emph{Heisenberg rotation}
and \emph{Heisenberg dilatation}, respectively.
Their induced actions on $\mathcal{N} \times \R_{\ge 0}$ are given by:
\begin{align*}
T_{[z,t]}: &
[\zeta, \nu, u] \mapsto
[z+\zeta, t+\nu+2 \Im(z \bar{\zeta}), u], 
\\
R_{\theta}: &
[\zeta,\nu,u] \mapsto [e^{i\theta}\zeta, \nu, u],
\\
D_{\lambda}: &
[\zeta,\nu,u] \mapsto
[\lambda \zeta,\lambda^2 \nu,\lambda^2 u].
\end{align*}

\subsubsection*{Cygan metric}

The \emph{Cygan metric} on the Heisenberg group $\mathcal{N}$ is defined for points $\partial \hc \ni p = [z,t] \in \mathcal{N}$ and $\partial \hc \ni q = [w,s] \in \mathcal{N}$ by
\begin{equation*}
d_{\Cyg}(p, q) =
\left| 2 \langle \bf{p}, \bf{q} \rangle_{H_2} \right|^{1/2}
=
\left| |z-w|^2 - i(t - s + 2\Im(z\bar{w})) \right|^{1/2},
\end{equation*}
where $\bf{p}$ and $\bf{q}$ are standard lifts of $p$ and $q$, respectively.
This is the restriction to $\mathcal{N}$ of the \emph{extended Cygan metric}, defined for $p=[z,t,u]$ and $q=[w,s,v] \in \overline{\hc} \setminus \{q_\infty\} = \mathcal{N} \times \R_{\geq 0}$ by
\begin{equation*}
d_{\Cyg}(p, q) = 
\left| |z-w|^2 + |u-v| - i(t-s+2\Im(z\bar{w})) \right|^{1/2}.
\end{equation*}
The extended metric satisfies $d_{\Cyg}(p,q) = |2\langle \bf{p}, \bf{q} \rangle_{H_2}|^{1/2}$ if (at least) one of $p$ and $q$ lies in $\partial \hc$ (i.e., $uv = 0$).
Note that the (extended) Cygan metric is not geodesic.
A \emph{Cygan sphere} in $\overline{\hc}$ with centre $[z_0,t_0] \in \mathcal{N}$ and radius $r>0$, is defined as
\begin{equation*}
S_{[z_0,t_0]}(r) =
\left\{ [z,t,u] \in \overline{\hc} \setminus \{q_\infty\} 
\ \middle| \
d_{\Cyg}([z,t,u], [z_0,t_0,0]) = r \right\}.  
\end{equation*}
Note that $S_{[z_0,t_0]}(r)$ is the image of $S_{[0,0]}(r)$ under the Heisenberg translation $T_{[z_0,t_0]}$.
An alternative description of Cygan spheres are given by the geographic coordinates.

\begin{definition}
\label{definition::geographic}
The \emph{geographic coordinates} $(\alpha,\beta,\omega)$ of a point 
$q = q(\alpha,\beta,\omega) \in S_{[0,0]}(r)$
are defined via the standard lift
\( \bf{q} = \bf{q}(\alpha,\beta,\omega)
= (-r^2 e^{-i\alpha}/2, rwe^{i(-\alpha/2+\beta)}, 1)^t \),
where $\alpha\in [-\pi/2,\pi/2]$,
$\beta\in [0, \pi)$
and $w\in [-\sqrt{\cos(\alpha)}, +\sqrt{\cos(\alpha)}]$.
In particular, the ideal boundary of $S_{[0,0]}(r)$ consists of points with $\omega = \pm\sqrt{\cos(\alpha)}$.
\end{definition}

\subsubsection*{Pullback the Cygan metric to the ball model}

We introduce the Cygan metric defined with respect to an arbitrary boundary point $q^{\ball} \in \Proj(V_{H_1}^0)$ in the unit ball model.

Let $q = [\kappa_1, \kappa_2, 1]^t \in \Proj(V_{H_2}^0)$ be the boundary point in the Siegel domain model corresponding to
$q^{\ball} = \left[\dfrac{\kappa_1 + 1}{\sqrt{2}}, \kappa_2, \dfrac{\kappa_1 - 1}{\sqrt{2}}\right]^t$.
Applying the Heisenberg translation $T_{-q} = T_{[-\kappa_2, -2 \Im(\kappa_1)]}$ that sends $q$ to $[0,0,1]^t$,
so that $H_2 \cdot T_{-q}(q) = q_{\infty}$.
Now, consider a point $\omega = [\omega_1,\omega_2, 1]^t \in \Proj(V_{H_1}^- \sqcup V_{H_1}^0)$ in the unit ball model.
Its image under the Cayley transform is
\(\zeta =\Cayley(\omega) = [\omega_1+1, \sqrt{2} \cdot \omega_2, \omega_1-1]^t \in \Proj(V_{H_2}^+ \sqcup V_{H_2}^0)\).
Applying the transformation $H_2 \cdot T_{-q}$ yields
\begin{equation*}
\widetilde{\zeta} \coloneqq H_2 \cdot T_{-q}(\zeta) =
\left[
\begin{array}{c}
\dfrac{\omega_1 - 1}{\omega_1 + 1 + \sqrt{2} \cdot \overline{\kappa_2} \cdot \omega_2 + \overline{\kappa_1} \cdot (\omega_1 - 1)} \\
\dfrac{\sqrt{2} \cdot \omega_2 -  \kappa_2 \cdot (\omega_1 - 1)}{\omega_1 + 1 + \sqrt{2} \cdot \overline{\kappa_2} \cdot \omega_2 + \overline{\kappa_1} \cdot (\omega_1 - 1)} \\
1 \\
\end{array}
\right].
\end{equation*}

\begin{definition}
The \emph{pullback Cygan metric} with respect to $q^{\ball} \in \Proj(V_{H_1}^0)$
is defined as
\(d_{q^{\ball}} \coloneqq (H_2 \cdot T_{-q} \cdot \Cayley)^* d_{\Cyg} \),
so that $d_{q^{\ball}}(\omega,\omega') = d_{\Cyg}(\widetilde{\zeta}, \widetilde{\zeta'})$
for any \( \omega, \omega' \in \Proj(V_{H_1}^- \sqcup V_{H_1}^0) \setminus \{q^{\ball}\} \).
\end{definition}

An explicit expression for this metric is given below.

\begin{proposition}
\label{proposition::dqball}
For \( \omega, \omega' \in \Proj(V_{H_1}^- \sqcup V_{H_1}^0) \setminus \{q^{\ball}\} \),
if at least one of $\omega, \omega'$ lies in $\partial \hc$, then
\[
d_{q^{\ball}}(\omega,\omega') = 
\left| 
\dfrac{4 \langle
\boldsymbol{\omega},
\boldsymbol{\omega'}
\rangle_{H_1}}{
\left(\omega_1 + 1 + \sqrt{2} \overline{\kappa_2}  \omega_2 + \overline{\kappa_1}  (\omega_1 - 1)\right) \cdot \left(\overline{\omega'_1} + 1 + \sqrt{2} \kappa_2 \overline{\omega'_2} + \kappa_1  (\overline{\omega'_1} - 1)\right)
}
\right|^{1/2}.
\]
\end{proposition}

\begin{proof}
When at least one of $\omega$, $\omega'$ lie on the boundary,
the Cygan distance is also given by the hermitian product, i.e.,
$d_{q^{\ball}}(\omega,\omega')
=
d_{\Cyg}(\widetilde{\zeta},\widetilde{\zeta'})
=
\left| 2 \langle
\boldsymbol{\widetilde{\zeta}},
\boldsymbol{\widetilde{\zeta'}}
\rangle_{H_2} \right|^{1/2}$,
where
\begin{align*}
\langle
\boldsymbol{\widetilde{\zeta}},
\boldsymbol{\widetilde{\zeta'}}
\rangle_{H_2} 
=&~
\dfrac{\omega_1 - 1}{\omega_1 + 1 + \sqrt{2} \overline{\kappa_2}  \omega_2 + \overline{\kappa_1}  (\omega_1 - 1)}
+
\dfrac{\overline{\omega'_1} - 1}{\overline{\omega'_1} + 1 + \sqrt{2} \kappa_2 \overline{\omega'_2} + \kappa_1  (\overline{\omega'_1} - 1)}
\\
&+
\dfrac{\sqrt{2}  \omega_2 -  \kappa_2  (\omega_1 - 1)}{\omega_1 + 1 + \sqrt{2}  \overline{\kappa_2}  \omega_2 + \overline{\kappa_1}  (\omega_1 - 1)}
\cdot
\dfrac{\sqrt{2} \cdot \overline{\omega'_2} -  \overline{\kappa_2}  (\overline{\omega'_1} - 1)}{\overline{\omega'_1} + 1 + \sqrt{2}  \kappa_2  \overline{\omega'_2} + \kappa_1  (\overline{\omega'_1} - 1)}.
\end{align*}
Denote the numerator and the denominator of 
$\langle
\boldsymbol{\widetilde{\zeta}},
\boldsymbol{\widetilde{\zeta'}}
\rangle_{H_2}$
by $\mathcal{N}$ and $\mathcal{D}$, respectively.
We get
\[
\mathcal{D} = \left(\omega_1 + 1 + \sqrt{2} \overline{\kappa_2}  \omega_2 + \overline{\kappa_1}  (\omega_1 - 1)\right) \cdot \left(\overline{\omega'_1} + 1 + \sqrt{2} \kappa_2 \overline{\omega'_2} + \kappa_1  (\overline{\omega'_1} - 1)\right)
\]
and
\begin{align*}
\mathcal{N} =&~
\omega_1 \overline{\omega'_1} + \omega_1 + \sqrt{2} \kappa_2 \overline{\omega'_2} \omega_1
+
\kappa_1 \omega_1 \left(
\overline{\omega'_1} -1
\right)
-
\left(
\overline{\omega'_1} + 1 + \sqrt{2} \kappa_2 \overline{\omega'_2}
+
\kappa_1 \left( \overline{\omega'_1} - 1 \right)
\right) \\
& +
\overline{\omega'_1} \omega_1 + \overline{\omega'_1} + \sqrt{2} \overline{\kappa_2} \omega_2 \overline{\omega'_1}
+
\overline{\kappa_1} \overline{\omega'_1} \left(
\omega_1 -1
\right)
-
\left(
\omega_1 + 1 + \sqrt{2} \overline{\kappa_2} \omega_2
+
\overline{\kappa_1} \left( \omega_1 - 1 \right)
\right) \\
& +
2 \omega_2 \overline{\omega'_2}
- \sqrt{2} \kappa_2 (\omega_1 - 1) \overline{\omega'_2}
+ |\kappa_2|^2 (\omega_1 - 1)(\overline{\omega'_1} - 1)
- \sqrt{2} \overline{\kappa_2} (\overline{\omega'_1} - 1) \omega_2 \\
=&~ 
2 \cdot \langle \omega, \omega' \rangle_{H_1},
\end{align*}
as desired.
\end{proof}

\subsection{Isometric spheres}

Suppose that $g \in \PU(H_2)$ is represented by $(g_{i,j})_{i,j=1}^3 \in \SU(H_2)$ and that $g(q_\infty) \neq q_\infty$.

\begin{definition} \label{definition:isometric_sphere}
The \emph{isometric sphere} of $g$,
denoted by $\isphere(g)$,
is the set
\[
\isphere(g) \coloneqq
\left\{
p \in \overline{\hc}
\ \middle| \
|\langle {\bf{p}}, {\bf{q}}_{\infty} \rangle_{H_2} |
=
|\langle {\bf{p}}, g^{-1}({\bf{q}}_{\infty}) \rangle_{H_2} |
\right\},
\]
where $\bf{p}$ is the standard lift of $p$ and $\bf{q}_\infty$ is the standard lift of $q_\infty$.
\end{definition}

\begin{remark}
For any $g \in \PU(H_2)$,
the modulus $|\langle {\bf{p}}, g({\bf{q}}_{\infty}) \rangle_{H_2} |$ is well-defined.
Indeed, different choice of the representative $\bf{g} \in \SU(H_2)$ of $g$ leaves the modulus $|\langle {\bf{p}}, \bf{g}({\bf{q}}_{\infty}) \rangle_{H_2} |$ unchanged.
\end{remark}

The isometric sphere $I(g)$ is a $3$-ball.
It coincides with the Cygan sphere centred at
$\boldsymbol{c}_g \coloneqq g^{-1}(q_{\infty}) = [\overline{g_{3,2}} / \overline{g_{3,1}}, 2\cdot \Im(\overline{g_{3,3}} / \overline{g_{3,1}})] \in\mathcal{N}$
with radius $\boldsymbol{r}_g \coloneqq \sqrt{2/|g_{3,1}|}$.
The \emph{interior} and the \emph{exterior} of $\isphere(g)$ are defined, respectively, as:
\begin{align*}
\isphere_{-}(g) \coloneqq
\left\{ p \in \overline{\hc}
\ \middle| \
|\langle {\bf{p}}, {\bf{q}}_{\infty} \rangle_{H_2} |
>
|\langle {\bf{p}}, g^{-1}({\bf{q}}_{\infty}) \rangle_{H_2}|
\right\}, \\
\isphere_{+}(g) \coloneqq
\left\{ p \in \overline{\hc}
\ \middle| \
|\langle {\bf{p}}, {\bf{q}}_{\infty} \rangle_{H_2} |
<
|\langle {\bf{p}}, g^{-1}({\bf{q}}_{\infty}) \rangle_{H_2}|
\right\}.
\end{align*}

If $g \in \PU(H_2)$ does not fix $q_{\infty}$ and has a fixed point $\fp_g \neq q_{\infty}$ in $\overline{\hc}$ for which the corresponding eigenvalue has unit modulus (eg., $g$ is elliptic or unipotent),
then $\fp_g \in I(g)$.
If $g \in \PU(H_2)$ fixes $q_{\infty}$ and the corresponding eigenvalue has unit modulus,
then for any $g' \in \PU(H_2)$ that does not fix $q_{\infty}$, we have $I(g') = I(gg')$.

\subsection{Fixed torus of regular, real elliptic isometry}
\label{subsection::elliptic}

Suppose that $g \in \PU(H_2)$ is a regular elliptic isometry. Such an isometry is real elliptic if and only if it preserves a Lagrangian plane. In this subsection, we introduce the union of all $\R$-circles preserved by a regular, real elliptic isometry. We start with the following definitions.

\begin{definition}
\label{definition::standard_fixed_torus}
In the Siegel domain model of $\hc$, the \emph{standard fixed torus} $\Torus$ is defined as
\[
\Torus \coloneqq \left\{
[z_1, z_2, 1]^t \in \partial \hc
~\middle|~
|z_1 + 1|^2 = 2 |z_2|^2
\right\}.
\]
\end{definition}

\begin{definition}
For a point $q = [z_1, z_2, 1]^t \in \partial \hc$, the associated circle $C_q$ through $q$ is defined as
\[
C_q
\coloneqq \left\{
C_q(\theta) \coloneqq E_{\theta,-\theta}(q)
= \left[
\begin{array}{c}
\dfrac{z_1 + 1}{2} \cdot e^{i \theta} + \dfrac{z_1 - 1}{2} \\
z_2 \cdot e^{-i \theta} \\
\dfrac{z_1 + 1}{2} \cdot e^{i \theta} - \dfrac{z_1 - 1}{2}
\end{array}
\right]
~\middle|~
0 \le \theta \le 2 \pi \mod{2 \pi}
\right\},
\]
where $q = C_q(0)$.
\end{definition}

\begin{figure}[htbp]
 \centering
 \begin{subfigure}{0.49\textwidth}
  \centering
  \includegraphics[width=\textwidth]{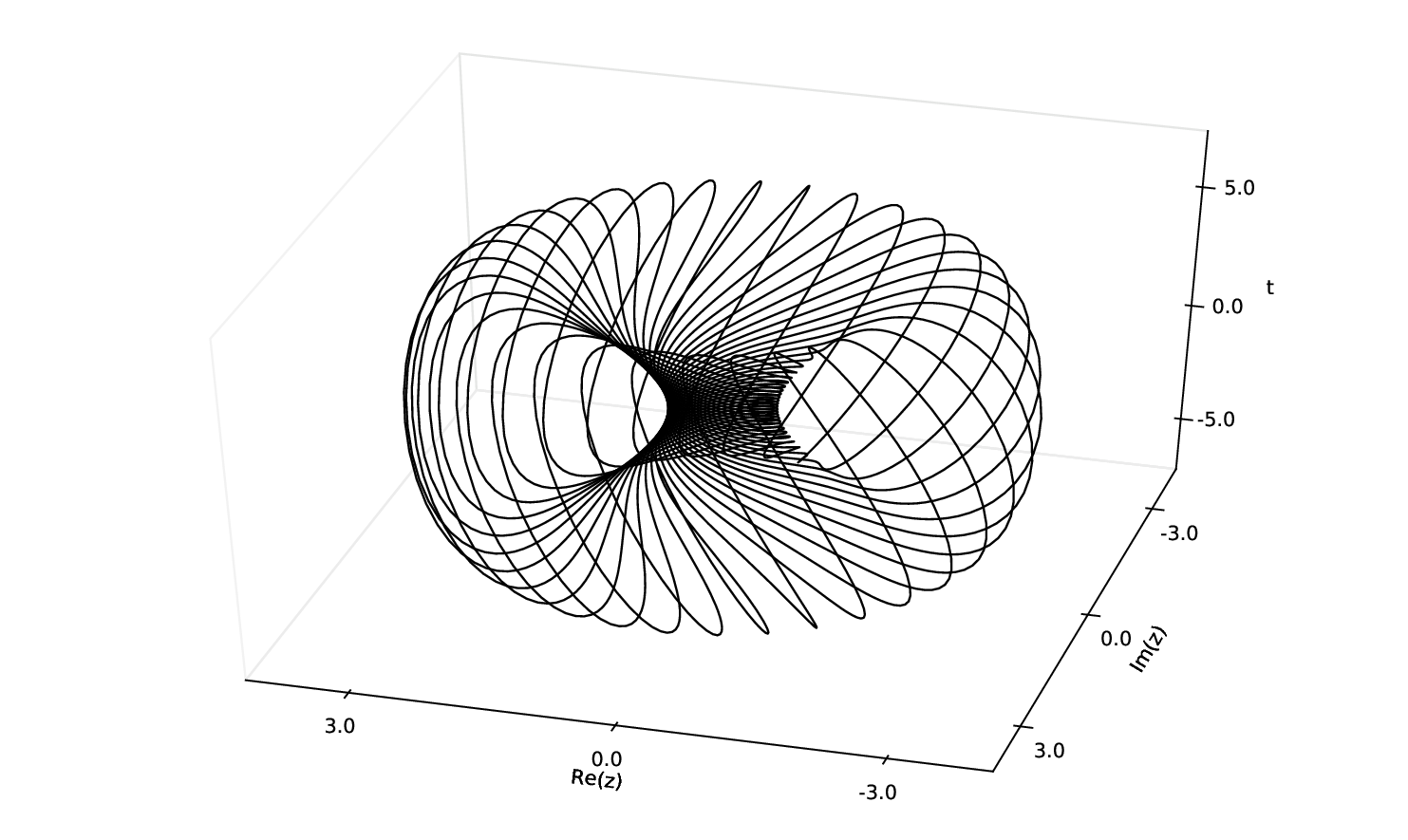}
 \end{subfigure}
 \begin{subfigure}{0.49\textwidth}
  \centering
  \includegraphics[width=\textwidth]{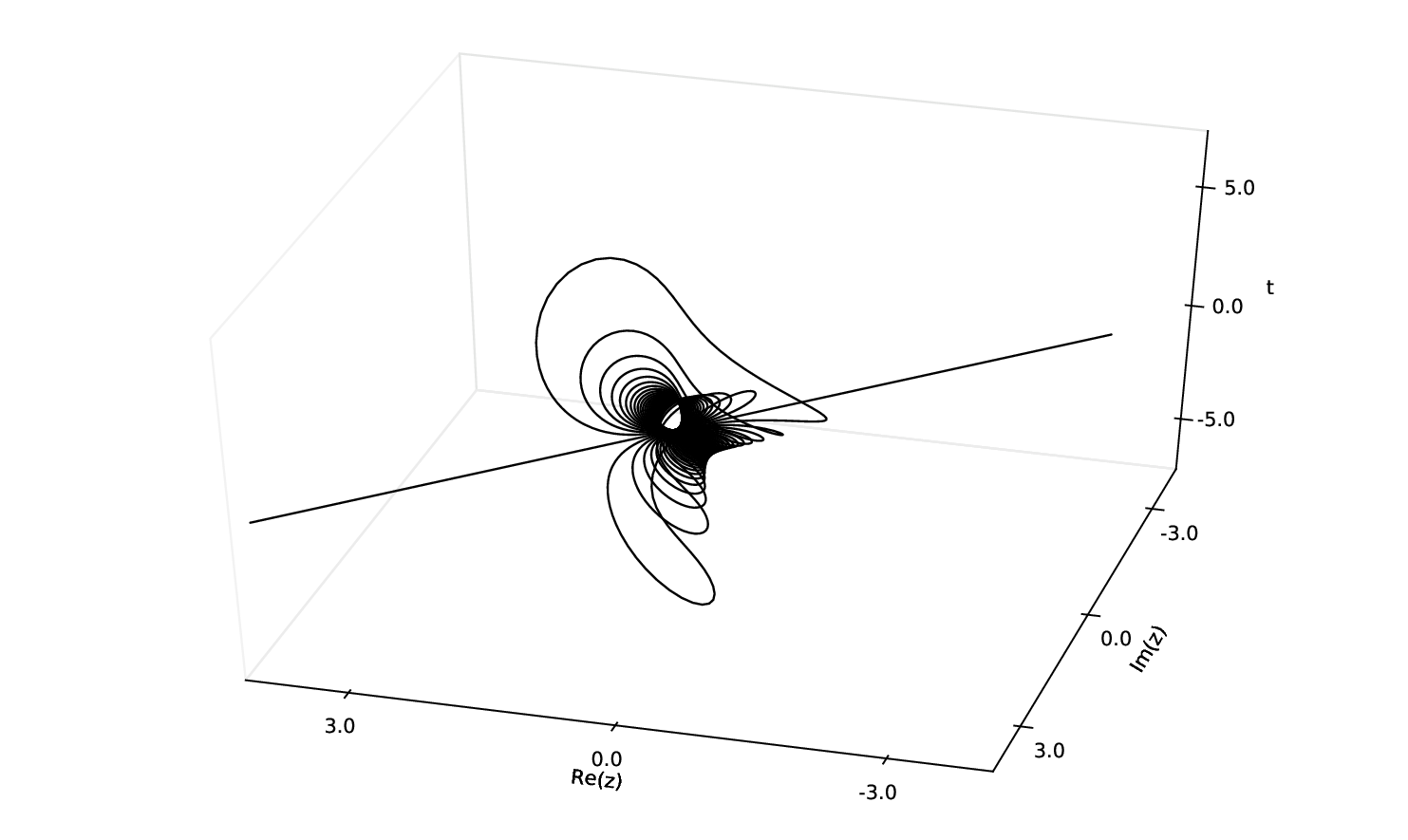}
 \end{subfigure}
 \caption{The standard torus $\Torus$ is foliated by the circles $\{C_q \mid q \in \Torus\}$, which can be explicitly parameterised as $\Torus = \bigsqcup_{\phi \in [0, 2\pi)} C_{q(\phi)}$ with $q(\phi)\coloneqq [-(3+2\sqrt{2}),(2+\sqrt{2})e^{i \phi},1]^t$. The left side of this figure shows the leaves $C_{q(\phi)}$ for $\phi \in \{k \pi / 16 \mid k = 0, 1, \ldots, 31\}$. Now, let $T$ be the Heisenberg translation sending $q_0 = [-1-2i, \sqrt{2}, 1]^t \in \Torus$ to $[0,0] \in \mathcal{N}$. The right side of this figure displays the images of the leaves $C_{q(\phi)}$ under $H_2 \cdot T$. In particular, the leaf $C_{q(0)}$ containing $q_0$ is mapped to the one-point compactification of an affine line.}
 \label{figure::torus}
\end{figure}

The torus $\Torus$ is foliated by the family $\{C_q \mid q \in \Torus\}$. Although this family is parametrized by $\Torus$, it is essentially a one-parameter family: distinct points $q \neq q'$ may satisfy $C_q = C_{q'}$. Each circle $C_q$ is also preserved by $E_{\theta,-\theta}$ for all $\theta$, and consequently $E_{\theta,-\theta}(\Torus) = \Torus$. Figure~\ref{figure::torus} illustrates this foliation of $\Torus$ together with its image under an isometry sending $q_0 = [-1-2i, \sqrt{2}, 1]^t \in \Torus$ to $q_\infty$. As the following propositions show, the torus $\Torus$ is precisely the union of all $\R$-circles preserved by $E_{\theta,-\theta}$.

\begin{proposition}
\label{proposition::torus_1}
For every $q \in \Torus$, the circle $C_q$ is an $\R$-circle.
\end{proposition}

\begin{proposition}
\label{proposition::torus_2}
Every $\R$-circle preserved by $E_{\theta,-\theta}$ for some $\theta \in (0, 2\pi) \setminus \{\pi\}$ coincides with a leaf $C_q$.
\end{proposition}

\begin{remark}
\label{remark::not_all_R_circles}
Proposition \ref{proposition::torus_1} claims that
every circle $C_q$ with $q \in \Torus$ is an $\R$-circle.
The condition $q \in \Torus$ cannot be removed.
Indeed, every real elliptic isometry $E_{\theta,-\theta}$ has two distinct positive eigenvectors,
implying the existence of two distinct complex geodesics $L_1$ and $L_2$ with $L_1 \cap L_2 = \{\fp_E\}$ that are preserved by $E_{\theta,-\theta}$.
If $E_{\theta,-\theta}$ is non-torsion, it follows that there are exactly two distinct circles $C_{q_1}$ and $C_{q_2}$ that are $\C$-circles.
By revisiting the description of $\C$-circles in \cite[40]{parker-2003-NCHG}, 
one can verify that they are the infinite $\C$-circle $C_{[0,0]}$ and the finite $\C$-circle $C_{[\sqrt{2},0]}$, respectively.
\end{remark}

Since every regular, real elliptic isometry $f$ is conjugate to some $E_{\theta,-\theta}$, the $\R$-circles preserved by $f$ are pairwise disjoint and their union is a torus. We call this union the \emph{fixed torus} of $f$ and denote it by $\Torus_f$. In particular, $\Torus_{E_{\theta,-\theta}} = \Torus$ for all $\theta \in (0,2\pi) \setminus \{\pi\}$.

The proofs of Propositions~\ref{proposition::torus_1} and~\ref{proposition::torus_2} are given in the remainder of this subsection.

\begin{proof}[Proof of Proposition \ref{proposition::torus_1}]
For $q = [\kappa_1, \kappa_2, 1]^t \in \Torus$, applying $H_2 \cdot T_{-q}$ to the $C_q$ yields the following.

\begin{lemma}
\label{lemma::torus_to_line}
Suppose that $\widetilde{\bf{c}_{\theta}} = [\widetilde{z_{\theta}}, \widetilde{t_{\theta}}] \coloneqq \left( H_2 \cdot T_{-q} \right) \left(C_q(\theta)\right)$, for $\theta \in (0, 2\pi)$. Then
\[
\widetilde{z}_{\theta_2} - \widetilde{z}_{\theta_1} = \mathcal{C}_z \cdot
\left(
\cot{\frac{\theta_2}{2}} - \cot{\frac{\theta_1}{2}}
\right)
\quad
\text{and}
\quad
\widetilde{t}_{\theta_2} - \widetilde{t}_{\theta_1} = \mathcal{C}_t \cdot
\left(
\cot{\frac{\theta_2}{2}} - \cot{\frac{\theta_1}{2}}
\right),
\]
where $\mathcal{C}_z = \dfrac{\kappa_2 (\kappa_1 + 3) i}{4 |\kappa_2|^2}$
and $\mathcal{C}_t = -\dfrac{\Re(\kappa_1) + 1}{2 |\kappa_2|^2}$.
Hence $(H_2 \cdot T_{-q}) (C_q) \setminus q_{\infty}$ is an affine line.
\end{lemma}

\begin{proof}[Proof of Lemma \ref{lemma::torus_to_line}]
Notice that $H_2 \cdot T_{-q} \cdot q = H_2 \cdot T_{-q} \cdot C_q(0) = q_\infty$
and
\begin{equation*}
H_2 \cdot T_{-q} \cdot C_q(\theta)
= \left[
\begin{array}{c}
\frac{1}{2}\left(
(\kappa_1 + 1) e^{i \theta} + (1 - \kappa_1) 
\right) \\
-i \cdot \kappa_2 (3 + \kappa_1) \cdot \sin{\frac{\theta}{2}} \cos{\frac{\theta}{2}} + \kappa_2 (\kappa_1 - 1) \cdot \sin^2{\frac{\theta}{2}} \\
-4 |\kappa_2|^2 \cdot \sin^2{\frac{\theta}{2}}
\end{array}
\right].
\end{equation*}
Therefore, this lemma comes from
\begin{align*}
\widetilde{z}_{\theta_2} - \widetilde{z}_{\theta_1} &=
\frac{[H_2 \cdot T_{-q} \cdot C_q(\theta_2)]_2}{[H_2 \cdot T_{-q} \cdot C_q(\theta_2)]_3}
-
\frac{[H_2 \cdot T_{-q} \cdot C_q(\theta_1)]_2}{[H_2 \cdot T_{-q} \cdot C_q(\theta_1)]_3}
= \mathcal{C}_z \cdot
\left(
\cot{\frac{\theta_2}{2}} - \cot{\frac{\theta_1}{2}}
\right), \\
\widetilde{t}_{\theta_2} - \widetilde{t}_{\theta_1} &= 
2 \Im \left(
\frac{[H_2 \cdot T_{-q} \cdot C_q(\theta_2)]_1}{[H_2 \cdot T_{-q} \cdot C_q(\theta_2)]_3}
\right)
-
2 \Im \left(
\frac{[H_2 \cdot T_{-q} \cdot C_q(\theta_1)]_1}{[H_2 \cdot T_{-q} \cdot C_q(\theta_1)]_3}
\right)
= \mathcal{C}_t \cdot
\left(
\cot{\frac{\theta_2}{2}} - \cot{\frac{\theta_1}{2}}
\right),
\end{align*}
as desired.
\end{proof}

We now return to the proof of Proposition~\ref{proposition::torus_1}.
An $\R$-circle passing through $q_\infty$ is the one-point compactification of the affine line
\[
\left\{
[x \cdot e^{i \theta_0} + x_0 + iy_0, v_0 + 2x \cdot y_0 \cos{\theta_0} - 2x \cdot x_0 \sin{\theta_0}]
~\middle|~
x \in \R
\right\} \subset \mathcal{N}
\]
for some $\theta_0 \in [0,2\pi)$ and $(x_0+i y_0, v_0) \in \mathcal{N}$ (see \cite[41]{parker-2003-NCHG}; cf. \cite[Section 4.4.9]{Goldman1999}).
Lemma \ref{lemma::torus_to_line} has shown that the curve $H_2 \cdot T_{-q} \cdot C_q \setminus q_\infty$ is indeed an affine line.
Instead of using $\theta \in (0,2\pi)$, we use $x \coloneqq -|\mathcal{C}_z| \cdot \cot{(\theta/2)} \in \R$ to parameterise this affine line.

We claim that the affine line parameterised by $x$ is an $\R$-circle passing through $q_\infty$ given by $(\theta_0, x_0, y_0, v_0)$.
As $x_0 + i y_0 = \widetilde{z}_{\pi}$ and $v_0 = \widetilde{t}_{\pi}$ and
$H_2 \cdot T_{-q} \cdot C_q(\pi) = \left[-\kappa_1, \kappa_2(\kappa_1 - 1), -4|\kappa_2|^2\right]^t$,
we get
$x_0 = \Re\left( \dfrac{\kappa_2 (\kappa_1 - 1)}{-4|\kappa_2|^2} \right)$
and
$y_0 = \Im\left( \dfrac{\kappa_2 (\kappa_1 - 1)}{-4|\kappa_2|^2} \right)$.
As $\widetilde{z}_{2 \cdot \arccot{\left(-x / |\mathcal{C}_z| \right)}} - \widetilde{z}_{\pi} = x \cdot \dfrac{-\mathcal{C}_z}{|\mathcal{C}_z|}$ for all $x \in \R$,
we get $e^{i \theta_0} = \dfrac{-\mathcal{C}_z}{|\mathcal{C}_z|}$.
It remains to verify that
\[
x\cdot\left(2 y_0 \cdot \cos{\theta_0} - 2 x_0 \cdot \sin{\theta_0}\right)
= \widetilde{t}_{2 \cdot \arccot{\left(-x / |\mathcal{C}_z| \right)}} - \widetilde{t}_{\pi} = x \cdot \dfrac{-\mathcal{C}_t}{|\mathcal{C}_z|}
\]
for all $x \in \R$,
which is equivalent to show that $2 y_0 \cdot \cos{\theta_0} - 2 x_0 \cdot \sin{\theta_0} = -\dfrac{\mathcal{C}_t}{|\mathcal{C}_z|}$.
More precisely, we need
$\Im\left( \kappa_2(\kappa_1 - 1) \right) \cdot \Re\left( \kappa_2(3 + \kappa_1) i \right) -
\Re\left( \kappa_2(\kappa_1 - 1) \right) \cdot \Im\left( \kappa_2(3 + \kappa_1) i \right)
=
4 |\kappa_2|^2 \cdot \left( \Re(\kappa_1) + 1 \right)$, where the left hand side of the desired equation is equal to
\begin{align*}
& - \Im(\kappa_1 \kappa_2 + \kappa_2)^2 + 4 \cdot \Im(\kappa_2)^2 - \Re(\kappa_1 \kappa_2 + \kappa_2)^2 + 4 \cdot \Re(\kappa_2)^2 \\
= & 4 |\kappa_2|^2 - |\kappa_1 \kappa_2 + \kappa_2|^2 \\
= & 4|\kappa_2|^2 - |\kappa_2|^2 \cdot |\kappa_1 + 1|^2 \\
= & 4 |\kappa_2|^2 - 2 |\kappa_2|^4
= 4 |\kappa_2|^2 \cdot \left( 1 - \frac{|\kappa_2|^2}{2} \right)
= 4|\kappa_2|^2 \cdot\left( \Re(\kappa_1) + 1 \right).
\end{align*}
Thus, the circle $C_q$ is an $\R$-circle.
\end{proof}

\begin{proof}[Proof of Proposition \ref{proposition::torus_2}]
Let $q = [z_1, z_2, 1]^t \in \partial \hc$ satisfy $|z_2|^2 + 2 \Re(z_1) = 0$. Suppose that $q$ lies on an $\R$-circle $\partial L$ preserved by $E_{\theta,-\theta}$. The fixed point $\fp_E$ of $E_{\theta,-\theta}$ must lie on the Lagrangian plane $L$, so Cartan's angular invariant gives
\(
\mathbb{A}(\fp_E, q, E_{\theta,-\theta}(q)) = 0
\).
Hence $|z_1+1|^2 + 2|z_2|^2 = 0$.
\end{proof}

\section{Intersection of isometric spheres}
\label{section::sphere}

From now on, we fix $q \in \Torus$ represented by
$\bf{q} = (\kappa_1, \kappa_2, 1)^t$, where $\kappa_1$ and $\kappa_2$ satisfy $|\kappa_2|^2 + 2 \Re(\kappa_1) = 0$ and $|\kappa_1 + 1|^2 = 2 |\kappa_2|^2$.
Consider the $\R$-circle $C_q = \{C_q(\theta) = E_{\theta,-\theta} (q)\}$ through $q$.
Set
\[
\EE_{\theta,-\theta,q} \coloneqq (H_2 T_{-q}) \cdot E_{\theta,-\theta} \cdot (H_2 T_{-q})^{-1}.
\]
The isometric sphere $I(\EE_{-\theta,\theta,q})$ is the Cygan sphere $S_{\bf{\widetilde{c}}_{\theta}}(\bf{\widetilde{r}}_{\theta})$ in $\overline{\hc}$ with $\boldsymbol{\widetilde{c}}_\theta \coloneqq \EE_{\theta,-\theta,q}(q_\infty) = H_2 \cdot T_{-q} \cdot C_q(\theta)$
and
\[
\boldsymbol{\widetilde{r}}_\theta \coloneqq
\sqrt{2} \cdot |(\EE_{-\theta,\theta,q})_{3,1}|^{-1/2}
=
\sqrt{2} \cdot |(\EE_{-\theta,\theta,q} \cdot q_\infty)_{3}|^{-1/2}
=
\frac{1}{\sqrt{2} |\kappa_2| \cdot \sin{(\theta / 2)}}.
\]

The following well-known proposition allows us to fix the choice of \(q \in \Torus\).

\begin{proposition}
\label{proposition::spheres_EEE}
For every \(q \in \Torus\), there exists an isometry \(Q_q \in \PU(H_2)\) that sends \(q\) to \(q_\infty\) and, for every \(\theta\), maps the isometric sphere \(I(\EE_{-\theta,\theta,q})\) to the Cygan sphere of radius \(1/\sin(\theta/2)\) centred at \([-\cot(\theta/2),0] \in \mathcal{N}\). Moreover, for any distinct \(q, q' \in \Torus\), there exists an isometry that maps \(I(\EE_{-\theta,\theta,q})\) to \(I(\EE_{-\theta,\theta,q'})\) for every \(\theta\).
\end{proposition}

\begin{proof}
Fix $0 < \theta_1 < 2\pi$. For $q \in \Torus$, choose a Heisenberg translation $T'$ and a rotation $R'$ such that $(R' T' H_2 T_{-q}) \cdot (C_q \setminus \{q_\infty\})$ is the real axis of $\C \subset \mathcal{N}$, with the origin at $(R' T' H_2 T_{-q}) \cdot C_q(\theta_1)$ and moving rightward as $\theta$ increases. Set $\EEE_{\theta,-\theta,q} = (R' T' H_2 T_{-q}) \, E_{\theta,-\theta} \, (R' T' H_2 T_{-q})^{-1}$. Then $I(\EEE_{-\theta,\theta,q})$ is the Cygan sphere $S_{\tilde{\bf{c}}_\theta}(\tilde{r}_\theta)$ with centre $\tilde{\bf{c}}_\theta = (R' T' H_2 T_{-q}) \cdot C_q(\theta)$ and radius $\tilde{r}_\theta = \sqrt{2}\,|(\EEE_{-\theta,\theta,q})_{3,1}|^{-1/2}$.

By Lemma \ref{lemma::torus_to_line}, the centre of $I(\EE_{-\theta,\theta,q})$ satisfies $\widetilde{z}_\theta - \widetilde{z}_{\theta_1} = \mathcal{C}_z \bigl( \cot(\theta/2) - \cot(\theta_1/2) \bigr)$. Applying $R' T'$ sends the centre to $[\widetilde{z}_\theta',0]$ with $\widetilde{z}_\theta' - \widetilde{z}_{\theta_1}' = \pm |\widetilde{z}_\theta - \widetilde{z}_{\theta_1}|$ and the radius remains $\tilde{r}_\theta$.
Hence, the isometric sphere $I(\EEE_{-\theta,\theta,q})$ is the Cygan sphere $S_{\bf{\dbtilde{c}}_{\theta}}(\bf{\dbtilde{r}}_{\theta})$ with
\[
\bf{\dbtilde{c}}_\theta =
\left[ |\mathcal{C}_z| \cdot \left( \cot{\frac{\theta_1}{2}} - \cot{\frac{\theta}{2}} \right), 0 \right] \in \mathcal{N}
\quad
\text{and}
\quad
\bf{\dbtilde{r}}_\theta = 
\frac{1}{\sqrt{2} |\kappa_2| \cdot \sin{(\theta / 2)}}.
\]
Note that \(|\mathcal{C}_z| = \dfrac{1}{\sqrt{2}\,|\kappa_2|}\). Hence, applying the Heisenberg dilatation \(D_{1/|\mathcal{C}_z|}\) and setting \(\theta_1 = \pi\), we obtain a Cygan sphere centred at \([-\cot(\theta/2),0]\) with radius \(1/\sin(\theta/2)\).
\end{proof}


\subsection{Standard model}
\label{subsection::standard}

From now on, we fix $q = [\kappa_1,\kappa_2,1]$ with $(\kappa_1, \kappa_2) = (-1-2i, \sqrt{2})$.

To describe the isometric spheres $I(\EE_{\theta,-\theta,q})$ for $\theta \in [0,2\pi)$ explicitly,
we pull them back into the unit ball model along $H_2 \cdot T_{-q} \cdot \Cayley$,
as follows.
First of all, the $\R$-circle $C_q$ through $q$ in the unit ball model is given by
\begin{equation*}
\Cayley(C_q) = \left\{
\Cayley(C_q(\theta))
=
\left[
\begin{array}{c}
\frac{\kappa_1 + 1}{\kappa_1 - 1} \cdot e^{i \theta} \\
\frac{\sqrt{2} \kappa_2}{\kappa_1 - 1} \cdot e^{-i \theta} \\
1
\end{array}
\right]
=
\left[
\begin{array}{c}
(1 + i) \cdot e^{i \theta} \\
-(1 - i) \cdot e^{-i \theta}\\
2
\end{array}
\right]
~\middle|~
0 \le \theta < 2\pi
\right\}.
\end{equation*}
Define
\[
I(\theta) \coloneqq (H_2 \cdot T_{-q} \cdot \Cayley)^{-1} \cdot I(\EE_{-\theta, \theta, q}),
\]
which is the isometric sphere in the unit ball model and
is called the \emph{standard model of the intersection} in our consideration.
We use $I_-(\theta)$ to denote the pull-back of the interior of the isometric sphere $I(\EE_{-\theta,\theta,q})$
and
also use $I_+(\theta)$ to denote the pull-back of the exterior.

The pull-back of $q_\infty$ is
$q^{\ball} = (H_2 \cdot T_{-q} \cdot \Cayley)^{-1} (q_\infty) = \Cayley(q) = [1+i,-(1-i),2]^t$.

\begin{proposition}
\label{proposition::standard-spheres}
For every point $\omega = [\omega_1, \omega_2, 1]^t \in \overline{\hc}$ in the unit ball model,
we have
\[
\omega \in I(\theta) \sqcup I_-(\theta)
\quad
\text{if and only if}
\quad
\dfrac{
\left|
(1-i) \cdot \omega_1 \cdot e^{-i \theta} - (1+i) \cdot \omega_2 \cdot e^{i\theta} - 2
\right|
}{
\left|
(1-i) \cdot \omega_1 - (1+i) \cdot \omega_2 - 2
\right|
} \le 1.
\]
Furthermore, $\omega \in I(\theta)$ if and only if the equality holds.
\end{proposition}

\begin{remark}
The denominator $|(1-i) \cdot \omega_1 - (1+i) \cdot \omega_2 - 2|$ vanishes if and only if $\omega = q^{\ball}$.
\end{remark}

\begin{remark}
The point $[0,0,1]^t \in \hc$ lies on every standard isometric sphere $I(\theta)$.
\end{remark}

\begin{proof}[Proof of Proposition \ref{proposition::standard-spheres}]
Recall that the isometric sphere $I(\EE_{-\theta,\theta,q})$ is the Cygan sphere with centre
$\boldsymbol{\widetilde{c}}_\theta = H_2 \cdot T_{-q} \cdot C_q(\theta)$
and radius
\(
\boldsymbol{\widetilde{r}}_\theta
=
\dfrac{1}{\sqrt{2} |\kappa_2| \cdot \sin{(\theta/2)}}
=
\dfrac{1}{2 \cdot \sin{(\theta/2)}}
\).
Then, the centre of $I(\theta)$ is
\[
(H_2 \cdot T_{-q} \cdot \Cayley)^{-1}(\bf{\widetilde{c}}_{\theta})
=
\Cayley(C_q(\theta))
=
[(1+i) \cdot e^{i \theta}, -(1-i)\cdot e^{-i \theta}, 2]^t \eqqcolon \omega'.
\]
Thus, by Proposition \ref{proposition::dqball}, a point $\omega = [\omega_1, \omega_2, 1]^t$
belongs to $I(\theta) \sqcup I_-(\theta)$ if and only if
\[
\boldsymbol{\widetilde{r}}_\theta
\ge
d_{q^{\ball}}(\omega,\omega')
=
\left|
\dfrac{
4 \cdot
\left(
\dfrac{1-i}{2} e^{-i \theta}
\omega_1
- \dfrac{1+i}{2} e^{i \theta} \omega_2 - 1
\right)
}{
8 \big(
\left(-1+i\right) \omega_1 + \left(1+i\right) \omega_2 + 2
\big)
\cdot
\sin^2{\left( \theta/2 \right)}
}
\right|^{1/2},
\]
which is equivalent to the desired inequality.
\end{proof}

We conclude this subsection with the proof of Proposition \ref{prox::intersection}.

\begin{proof}[Proof of Proposition \ref{prox::intersection}]
By Proposition \ref{proposition::spheres_EEE}, there exists an isometry sending each \(I(f_{\theta})\) to \(I(\theta)\). Hence it suffices to study intersections of the sets \(I(\theta)\). With this reduction, assertions (1), (2), and (3) follow respectively from Theorem \ref{theorem::intersection_of_two}, Corollary \ref{corollary::intersection_of_three}, and Corollary \ref{corollary::intersection_of_four}, while the desired foliation follows from Corollary \ref{corollary::foliation_of_intersections}.
\end{proof}

%
\subsection{Intersection of two isometric spheres}
\label{subsection::2spheres}

Let $0 < \theta_1 < \theta_2 < 2\pi$ be arbitrary.
Suppose that $\omega = [\omega_1, \omega_2, 1]^t \in I(\theta_1) \cap I(\theta_2)$.
By Proposition \ref{proposition::standard-spheres}, we get
\[
\dfrac{
\left|
(1-i) \cdot \omega_1 \cdot e^{-i \theta_1} - (1+i) \cdot \omega_2 \cdot e^{i\theta_1} - 2
\right|
}{
\left|
(1-i) \cdot \omega_1 - (1+i) \cdot \omega_2 - 2
\right|
}
=
\dfrac{
\left|
(1-i) \cdot \omega_1 \cdot e^{-i \theta_2} - (1+i) \cdot \omega_2 \cdot e^{i\theta_2} - 2
\right|
}{
\left|
(1-i) \cdot \omega_1 - (1+i) \cdot \omega_2 - 2
\right|
}
= 1.
\]
To further illustrate the intersection $I(\theta_1) \cap(\theta_2)$, we define the coordinates $(\psi_1, \psi_2)$ by
\begin{align*}
(1-i) \cdot \omega_1 \cdot e^{-i \theta_1} - (1+i) \cdot \omega_2 \cdot e^{i\theta_1} - 2
=
(-1) \cdot \left(
(1-i) \cdot \omega_1 - (1+i) \cdot \omega_2 - 2
\right)
\cdot e^{i \psi_1}, \\
(1-i) \cdot \omega_1 \cdot e^{-i \theta_2} - (1+i) \cdot \omega_2 \cdot e^{i\theta_2} - 2
=
(-1) \cdot \left(
(1-i) \cdot \omega_1 - (1+i) \cdot \omega_2 - 2
\right)
\cdot e^{i \psi_2}.
\end{align*}
The coordinates $(\psi_1, \psi_2)$ also parameterise the union of the extors extending $I(\theta_1)$ and $I(\theta_2)$, which is a torus (cf. \cite[Section 8.3.4]{Goldman1999}). We now describe the intersection $I(\theta_1) \cap I(\theta_2)$ in $(\psi_1, \psi_2)$.

\begin{proposition}
\label{proposition::intersection-of-two-isometric-spheres-using-psi}
The intersection $I(\theta_1) \cap I(\theta_2)$, parameterised by $(\psi_1, \psi_2) \in [0, 2\pi)^2$, is described by the inequality
\begin{align*}
\widetilde{W}(\psi_1, \psi_2) \coloneqq &
\left(
\left(
\cos{\left( \theta_1 \right)} - 1
\right)^2
+
\left(
\cos{\left( \theta_2 \right)} - 1
\right)^2
+
\left(
\cos{\left( \theta_1 - \theta_2 \right)} - 1
\right)^2
\right) 
\\
&+ 
2
\left( 1 - \cos{\left(\theta_1\right)}
\right)
\left(\cos{\left(\theta_2\right)}
- 1
\right)
\cdot 
\cos{\left(\psi_1 - \psi_2\right)} \\
&+ 
2
\left( \cos{\left(\theta_2\right)}
-1
\right)
\left(\cos{\left(\theta_2 - \theta_1\right)}
- 1
\right)
\cdot 
\cos{\left(\psi_1\right)} \\
&+
2
\left( \cos{\left(\theta_1\right)}
-1
\right)
\left(\cos{\left(\theta_2 - \theta_1\right)}
- 1
\right)
\cdot 
\cos{\left(\psi_2\right)} 
\le 0.
\end{align*}
\end{proposition}

\begin{proof}
Set $u_j \coloneqq (1 - i)\left(e^{-i \theta_j} + e^{i \psi_j}\right)$
and
$v_j \coloneqq (1 + i)\left(e^{i \theta_j} + e^{i \psi_j}\right)$, for $j=1,2$.
Then, we have
\( \omega_1 u_1 - \omega_2 v_1 = 2\cdot\left( e^{i \psi_1} + 1
\right) \)
and
\( \omega_1 u_2 - \omega_2 v_2 = 2\cdot\left( e^{i \psi_2} + 1
\right) \).
We first have the following progressive claims:

\noindent
\textbf{Claim 1:}
If $u_1 v_2 = u_2 v_1$, then either
\begin{itemize}[label = -]
\item $\theta_1 = \pi$ and $\psi_1 = 0$, or
\item $\theta_2 = \pi$ and $\psi_2 = 0$, or
\item $\theta_2 - \theta_1 = \pi$ and $|\psi_2 - \psi_1| = \pi$, or
\item $\theta_1 \neq \pi$, $\theta_2 \neq \pi$, $\theta_2 - \theta_1 \neq \pi$, $\cos(\psi_1) = -\cos(\theta_1)$ and $\cos(\psi_2) = -\cos(\theta_2)$.
\end{itemize}

\noindent
\textbf{Claim 2:}
If $u_1 v_2 = u_2 v_1$, then there is no solution for $(\omega_1, \omega_2)$.

\begin{proof}[Proof of Claim 1]
The relation $u_1 v_2 = u_2 v_1$ holds if and only if
\[
\left(e^{i \theta_1}  + e^{i \psi_1}\right) \cdot \left(e^{- i \theta_2} + e^{i \psi_2}\right)
=
\left(e^{i \theta_2} + e^{i \psi_2}\right) \cdot \left(e^{- i \theta_1} + e^{i \psi_1}\right).
\]
It implies that $\sin(\theta_1 - \theta_2) + \sin(\theta_1) \cdot
e^{i \psi_2} = \sin(\theta_2) \cdot e^{i \psi_1}$.
If $\sin(\theta_2) = 0$, then $\theta_2 = \pi$ and $e^{i \psi_2} = 1$, so $\psi_2 = 0$.
If $\sin(\theta_1) = 0$, then $\theta_1 = \pi$ and $e^{i \psi_1} = 1$, so $\psi_1 = 0$.
If $\sin(\theta_1 - \theta_2) = 0$, then $\theta_2 - \theta_1 = \pi$ and $e^{i \psi_2} = - e^{i \psi_1}$, so $|\psi_2 - \psi_1| = \pi$.
Otherwise, since $|e^{i \psi_2}| = 1$, we get
$|e^{i \psi_1} \cdot \sin(\theta_2) - \sin(\theta_1 - \theta_2)| = |\sin(\theta_1)|$.
Rewriting yields
\begin{equation*}
\sin^2(\theta_1) = \sin^2(\theta_2) - 2 \sin(\theta_2) \sin(\theta_1 - \theta_2) \cos(\psi_1) + \sin^2(\theta_1 - \theta_2).
\end{equation*}
Given $\sin(\theta_2) \neq 0$
and $\sin(\theta_1 - \theta_2) \neq 0$, it follows that
\[
\cos(\psi_1) = \dfrac{\sin^2(\theta_2) + \sin^2(\theta_1 - \theta_2) - \sin^2(\theta_1)}{2 \sin(\theta_2) \sin(\theta_1 - \theta_2)}
= -\cos(\theta_1).
\]
Symmetrically, we also obtain $\cos(\psi_2) = -\cos(\theta_2)$.
\end{proof} 

\begin{proof}[Proof of Claim 2]
When $\theta_1 = \pi$ and $\psi_1 = 0$, we have $u_1 = 0$ and $v_1 = 0$,
but they imply the contradiction $0 = \omega_1 u_1 - \omega_2 v_1 = 2 (e^{i \psi_1} + 1) = 4$.
When $\theta_2 = \pi$ and $\psi_2 = 0$, we have $u_2 = 0$ and $v_2 = 0$, but they imply
the contradiction $0 = \omega_1 u_2 - \omega_2 v_2 = 2 (e^{i \psi_2} + 1) = 4$.
When $\theta_2 - \theta_1 = \pi$ and $|\psi_2 - \psi_1| = \pi$, we have $u_2 = -u_1$ and $v_2 = -v_1$.
Therefore, we get $2(e^{i \psi_2} + 1) = - 2(e^{i \psi_1} + 1)$, but it implies a contradiction $2 = -2$.

Now, suppose that $\theta_1 \neq \pi$, $\theta_2 \neq \pi$, $\theta_2 - \theta_1 \neq \pi$, $\cos(\psi_1) = -\cos(\theta_1)$ and $\cos(\psi_2) = -\cos(\theta_2)$.
In this case, for every $j \in \{1,2\}$, we have either $\theta_j + \psi_j \equiv \pi$ or $\theta_j - \psi_j \equiv \pi \mod 2\pi$.
Equivalently, for every $j$, either $u_j = 0$ or $v_j = 0$, for every $j$.
If $u_2 = v_1 = 0$, then $u_1 v_2 = 0$, i.e., one of $u_1$, $v_2$ vanishes.
If $u_1 = v_2 = 0$, then $u_2 v_1 = 0$, i.e., one of $u_2$, $v_1$ vanishes.
Therefore, either $u_1 = u_2 = 0$ or $v_1 = v_2 = 0$.

If $u_1 = u_2 = 0$, then $v_j = 2i \cdot (e^{i \theta_j} + e^{i \psi_j}) = 2i \cdot \sin(\theta_j)$.
Thus, the relations $\omega_1 u_j - \omega_2 v_j = 2 (e^{i \psi_j} + 1) \coloneqq c_j$ become $\omega_2 = - c_1 / v_1 = - c_2 / v_2$. 
Hence, we get $\dfrac{1 - e^{-i \theta_1}}{\sin(\theta_1)} = \dfrac{1 - e^{-i \theta_2}}{\sin(\theta_2)}$,
which implies the contradiction $\theta_1 = \theta_2$.
Similarly, if $v_1 = v_2 = 0$, then there is no solution for $(\omega_1, \omega_2)$.
\end{proof}

Now, we continue the proof of Proposition \ref{proposition::intersection-of-two-isometric-spheres-using-psi}.
Since $u_1 v_2 - u_2 v_1 \neq 0$, we write
\begin{align*}
\omega_1 &= 2 \dfrac{
\left( e^{i \psi_1} + 1 \right) v_2 - \left( e^{i \psi_2} + 1 \right) v_1
}{u_1 v_2 - u_2 v_1}
=
(1 + i)
\dfrac{
e^{i \psi_1} (e^{i \theta_2} - 1) + e^{i \psi_2} (1 - e^{i \theta_1}) + e^{i \theta_2} - e^{i \theta_1}
}{
-2i \sin(\theta_1) e^{i \psi_2}
+ 2i \sin(\theta_2) e^{i \psi_1}
+ 2i \sin(\theta_2 - \theta_1)
},
\\
\omega_2 &= 2 \dfrac{
\left( e^{i \psi_1} + 1 \right) u_2 - \left( e^{i \psi_2} + 1 \right) u_1
}{u_1 v_2 - u_2 v_1}
=
(1 - i)
\dfrac{
e^{i \psi_1} (e^{-i \theta_2} - 1) + e^{i \psi_2} (1 - e^{-i \theta_1}) + e^{-i \theta_2} - e^{-i \theta_1}
}{
-2i \sin(\theta_1) e^{i \psi_2}
+ 2i \sin(\theta_2) e^{i \psi_1}
+ 2i \sin(\theta_2 - \theta_1)
}.
\end{align*}
For simplicity, we denote the numerators and the common denominator on the right-hand side of the expressions above by $\mathcal{N}_1$, $\mathcal{N}_2$ and $\mathcal{D} \neq 0$, respectively.
Therefore, the intersection $I(\theta_1) \cap I(\theta_2)$, given by $|\omega_1|^2 + |\omega_2|^2 \le 1$, is also described by $|\mathcal{N}_1|^2 + |\mathcal{N}_2|^2 - |\mathcal{D}|^2 \le 0$.

We have the following calculations.
\begin{align*}
|\mathcal{N}_1|^2 + |\mathcal{N}_2|^2
=&~ 2 \cdot \left(
12 - 4 \cos(\theta_2) - 4 \cos(\theta_1) - 4 \cos(\theta_1 - \theta_2)
\right) \\
&+ 8 \cdot \left(
\cos(\theta_1) + \cos(\theta_2) - \cos(\theta_1 - \theta_2) - 1
\right)
\cdot \cos(\psi_1 - \psi_2) \\
&+ 8 \cdot \left(
1 - \cos(\theta_2) - \cos(\theta_1 - \theta_2) + \cos(\theta_1)
\right)
\cdot \cos(\psi_1) \\
&+ 8 \cdot \left(
1 - \cos(\theta_1) - \cos(\theta_1 - \theta_2) + \cos(\theta_2)
\right)
\cdot \cos(\psi_2). \\
|\mathcal{D}|^2
=&~
4 \sin^2(\theta_1) + 4 \sin^2(\theta_2) + 4 \sin^2(\theta_2 - \theta_1) \\
&- 8 \sin(\theta_1) \sin(\theta_2) \cdot \cos(\psi_1 - \psi_2) \\
&+ 8 \sin(\theta_2) \sin(\theta_2 - \theta_1) 
\cdot \cos(\psi_1) \\
&- 8 \sin(\theta_1) \sin(\theta_2 - \theta_1)
\cdot \cos(\psi_2).
\end{align*}
Define the function 
$\widetilde{W} \coloneqq (|\mathcal{N}_1|^2 + |\mathcal{N}_2|^2 - |\mathcal{D}|^2) / 4$.
By combining the above results, we obtain the desired expression. This completes the proof of Proposition \ref{proposition::intersection-of-two-isometric-spheres-using-psi}.
\end{proof}

The following lemma claims that it is sufficient to consider only $(\psi_1, \psi_2) \in (0, 2\pi)^2$.

\begin{lemma}
If $\psi_1 = 0$ or $\psi_2 = 0$,
then $\widetilde{W}(\psi_1, \psi_2) > 0$.
\end{lemma}

\begin{proof}
Assume that $\psi_1=0$.
We get
\begin{align*}
\widetilde{W}(0,\psi_2) =& 
\left(
\left(
\cos{\left( \theta_1 \right)} - 1
\right)^2
+
\left(
\cos{\left( \theta_2 \right)} - 1
\right)^2
+
\left(
\cos{\left( \theta_1 - \theta_2 \right)} - 1
\right)^2
\right) \\
&+ 
2
\left(
\cos{\left(\theta_1\right)}
-1
\right)
\left(\cos{\left(\theta_1 - \theta_2\right)}
- 
\cos{\left(\theta_2\right)}
\right)
\cdot 
\cos{\left(\psi_2\right)} \\
&+
2
\left( \cos{\left(\theta_2\right)}
-1
\right)
\left(\cos{\left(\theta_2 - \theta_1\right)}
- 1
\right) \\
\ge & 
\left(
\left(
\cos{\left( \theta_1 \right)} - 1
\right)^2
+
\left(
\cos{\left( \theta_2 \right)} - 1
\right)^2
+
\left(
\cos{\left( \theta_1 - \theta_2 \right)} - 1
\right)^2
\right) \\
&-
2
\left|
\cos{\left(\theta_1\right)}
-1
\right|
\cdot
\left|\cos{\left(\theta_1 - \theta_2\right)}
- 
\cos{\left(\theta_2\right)}
\right| \\
&+
2
\left( \cos{\left(\theta_2\right)}
-1
\right)
\left(\cos{\left(\theta_2 - \theta_1\right)}
- 1
\right) \\
> & 
\left(
\left(
\cos{\left( \theta_1 \right)} - 1
\right)^2
+
\left(
\cos{\left( \theta_2 \right)} - 1
\right)^2
+
\left(
\cos{\left( \theta_1 - \theta_2 \right)} - 1
\right)^2
\right) \\
&-
2
\left|
\cos{\left(\theta_1\right)}
-1
\right|
\cdot
\max\left\{
\left|\cos{\left(\theta_1 - \theta_2\right)}
- 
1
\right|
,
\left|
\cos{\left(\theta_2\right)}
- 1
\right|
\right\}
\\
&+
2
\left( \cos{\left(\theta_2\right)}
-1
\right)
\left(\cos{\left(\theta_2 - \theta_1\right)}
- 1
\right),
\end{align*}
where the strict inequality $``>"$ follows from the triangle inequality applied to
$\cos(\theta_1 - \theta_2)$, $\cos(\theta_2)$ and $1$,
noting that equality cannot hold since $\theta_2 \in (0,2\pi)$ and $\theta_1 \neq \theta_2$.
Therefore, we must have $\widetilde{W}(0,\psi_2) > 0$.
The case $\psi_2 = 0$ is similar.
\end{proof}

Using the substitutions $X \coloneqq \cot(\psi_1/2)$, $Y \coloneqq \cot(\psi_2/2)$, we rewrite Proposition \ref{proposition::intersection-of-two-isometric-spheres-using-psi} as follows.

\begin{proposition}
\label{proposition::intersection-of-two-isometric-spheres}
The intersection $I(\theta_1) \cap I(\theta_2)$, parameterised by $(\psi_1, \psi_2) \in (0, 2\pi)^2$, is described by the inequality
\(
W(X,Y) \coloneqq
c_{22} X^2Y^2 + c_{20} X^2 + c_{02} Y^2 + 2 c_{11} XY + c_{00} \le 0,
\)
where
\begin{align*}
c_{22} &\coloneqq 2 \sin^2{\left(\frac{\theta_1 - \theta_2}{2}\right)} \cdot
\left(
6 - \cos{\left(\theta_1 - \theta_2\right)}
- \cos{\left(\theta_1 + \theta_2\right)}
- 2 \cos{\left(\theta_1\right)}
- 2 \cos{\left(\theta_2\right)}
\right), \\
c_{20} &\coloneqq 2 \sin^2{\left(\frac{\theta_2}{2}\right)} \cdot
\left(
6 - 2\cos{\left(\theta_1 - \theta_2\right)}
- \cos{\left(2\theta_1 - \theta_2\right)}
- 2 \cos{\left(\theta_1\right)}
- \cos{\left(\theta_2\right)}
\right), \\
c_{02} &\coloneqq 2 \sin^2{\left(\frac{\theta_1}{2}\right)} \cdot
\left(
6 - 2\cos{\left(\theta_1 - \theta_2\right)}
- \cos{\left(\theta_1 - 2\theta_2\right)}
- \cos{\left(\theta_1\right)}
- 2 \cos{\left(\theta_2\right)}
\right), \\
c_{11} &\coloneqq -16 \sin^2{\left(\frac{\theta_1}{2}\right)} \sin^2{\left(\frac{\theta_2}{2}\right)}
\quad
\text{and}
\quad
c_{00} \coloneqq -16 \sin^2{\left(\frac{\theta_1}{2}\right)} \sin^2{\left(\frac{\theta_2}{2}\right)} \sin^2{\left(\frac{\theta_1 - \theta_2}{2}\right)}.
\end{align*}
\end{proposition}

\begin{proof}
This proposition is derived from Proposition \ref{proposition::intersection-of-two-isometric-spheres-using-psi} by making the substitutions
\begin{equation*}
\cos(\psi_1) = \dfrac{\cot^2\left(\psi_1/2\right)-1}{\cot^2\left(\psi_1/2\right)+1} = \dfrac{X^2 - 1}{X^2 + 1},
\quad
\cos(\psi_2) = \dfrac{\cot^2\left(\psi_2/2\right)-1}{\cot^2\left(\psi_2/2\right)+1} = \dfrac{Y^2 - 1}{Y^2 + 1}
\end{equation*}
and
\begin{align*}
\cos(\psi_1 - \psi_2)
=& \cos(\psi_1) \cos(\psi_2) + \sin(\psi_1) \sin(\psi_2) \\
=&
\dfrac{\cot^2\left(\psi_1/2\right)-1}{\cot^2\left(\psi_1/2\right)+1}
\cdot
\dfrac{\cot^2\left(\psi_2/2\right)-1}{\cot^2\left(\psi_2/2\right)+1}
+
\dfrac{2\cot\left(\psi_1/2\right)}{\cot^2\left(\psi_1/2\right)+1}
\cdot
\dfrac{2\cot\left(\psi_2/2\right)}{\cot^2\left(\psi_2/2\right)+1} \\
=&
\dfrac{X^2Y^2 - X^2 - Y^2 + 4XY + 1}{(X^2+1) (Y^2+1)},
\end{align*}
where we define $W(X,Y) \coloneqq \widetilde{W}(\psi_1, \psi_2) \cdot (X^2+1) (Y^2+1)$.
\end{proof}

\begin{lemma}
\label{lemma::parameters-W}
The coefficients of $W(X,Y)$ satisfy $c_{22} > 0$, $c_{20} > 0$, $c_{02} > 0$ and $c_{11}^2 - c_{20}c_{02} < 0$.
\end{lemma}

\begin{proof}
These inequalities follow from a direct computation.
\end{proof}

\begin{lemma}
\label{lemma::W_with_XY=1}
The coefficients of $W(X,Y)$ satisfy $c_{20}+c_{02}+2c_{11}=c_{22}+c_{00}$.
Moreover, if $|XY| = 1$ then $W(X,Y) > 0$.
\end{lemma}

\begin{proof}
Observe that
\(
c_{20} + c_{02} + 2 c_{11} = c_{22} + c_{00} = -8 \sin^2{\left(\dfrac{\theta_1 - \theta_2}{2}\right)}
\cdot
\left(
\cos{\left(\theta_1\right)}
\cos{\left(\theta_2\right)}
-1
\right)
\).
If $XY = 1$ then, by Lemma \ref{lemma::parameters-W} and the above observation, we get
\[
W(X,Y) \ge c_{22} + 2\sqrt{c_{20} c_{02}} + 2 c_{11} + c_{00}
\ge c_{22} + 2 |c_{11}| + 2 c_{11} + c_{00}
= c_{22} + c_{00} > 0.
\]
If $XY = -1$, then
$W(X,Y) \ge c_{20} + c_{02} + 2 \sqrt{c_{20} c_{02}}$ > 0.
\end{proof}

We therefore arrive at the following theorem.

\begin{theorem}
\label{theorem::intersection_of_two}
For $0 < \theta_1 < \theta_2 < 2\pi$,
the intersection $I(\theta_1) \cap I(\theta_2)$ is homeomorphic to a $2$-disk.
\end{theorem}

\begin{proof}
By Proposition \ref{proposition::intersection-of-two-isometric-spheres}, the intersection $I(\theta_1) \cap I(\theta_2)$ is given by
\[
W(X,Y) = 
c_{22}X^2Y^2 + (X,Y)
\left(
\begin{array}{cc}
c_{20} & c_{11} \\
c_{11} & c_{02}
\end{array}
\right)
\left(
\begin{array}{c}
X \\
Y
\end{array} 
\right) + c_{00} \le 0.
\]
Since both the quartic term $c_{22}X^2Y^2$ and the quadratic form, which is an ellipse by Lemma \ref{lemma::parameters-W}, are strictly increasing along rays from the origin, every non-empty level set of $W$ with respect to $(X,Y)$ is either the single point $\{(0,0)\}$ or is homeomorphic to a circle.
\end{proof}

\subsection{Intersection of three isometric spheres}
\label{subsection::3spheres}

Let $0 < \theta_1 < \theta_2 < 2\pi$ and $0 \le \theta_3 \le 2 \pi$ be arbitrary.

\begin{proposition}
\label{proposition::intersection-of-three-isometric-spheres}
For every $\theta_3 \not\in \{0, \theta_1, \theta_2, 2\pi\}$,
the intersection $I(\theta_1) \cap I(\theta_2) \cap \left( 
I(\theta_3) \sqcup I_-(\theta_3)\right) \subset I(\theta_1) \cap I(\theta_2)$,
parameterised by $(\psi_1, \psi_2) \in (0,2\pi)^2$,
is described by
$\widehat{Q}(X,Y) \coloneqq c_{22}' X^2 Y^2 + c_{20}' X^2 + c_{02}' Y^2 + 2 c_{11}' XY \le 0$
and $W(X,Y) \le 0$,
where $X \coloneqq \cot{\frac{\psi_1}{2}}$, $Y \coloneqq \cot{\frac{\psi_2}{2}}$ and
\begin{align*}
c_{22}' &=
64 \cdot \sin{\left( \frac{\theta_3}{2} \right)}
\sin^2{\left( \frac{\theta_1 - \theta_2}{2} \right)}
\sin{\left( \frac{\theta_1 - \theta_3}{2} \right)}
\sin{\left( \frac{\theta_2 - \theta_3}{2} \right)}
\sin{\left( \frac{\theta_3 - (\theta_1 + \theta_2)}{2} \right)}, \\
c_{20}' &=
64 \cdot \sin^2{\left( \frac{\theta_2}{2} \right)}
\sin{\left( \frac{\theta_3}{2} \right)}
\sin{\left( \frac{\theta_1 - \theta_3}{2} \right)}
\sin{\left( \frac{\theta_2 - \theta_3}{2} \right)}
\sin{\left( \frac{\theta_1 - \theta_2 + \theta_3}{2} \right)}, \\
c_{02}' &=
64 \cdot \sin^2{\left( \frac{\theta_1}{2} \right)}
\sin{\left( \frac{\theta_3}{2} \right)}
\sin{\left( \frac{\theta_1 - \theta_3}{2} \right)}
\sin{\left( \frac{\theta_2 - \theta_3}{2} \right)}
\sin{\left( \frac{\theta_2 - \theta_1 + \theta_3}{2} \right)}, \\
c_{11}' &=
-64 \cdot \sin{\left( \frac{\theta_1}{2} \right)}
\sin{\left( \frac{\theta_2}{2} \right)}
\sin^2{\left( \frac{\theta_3}{2} \right)}
\sin{\left( \frac{\theta_1 - \theta_3}{2} \right)}
\sin{\left( \frac{\theta_2 - \theta_3}{2} \right)}.
\end{align*}
In particular, the intersection $I(\theta_1) \cap I(\theta_2) \cap I(\theta_3)$
is given by
$\widehat{Q}(X,Y) = 0$ and $W(X,Y) \le 0$.
\end{proposition}

\begin{proof}
Suppose that $\omega \in I(\theta_1) \cap I(\theta_2)$.
As in the proof of Proposition \ref{proposition::intersection-of-two-isometric-spheres-using-psi}, given $(\psi_1, \psi_2)$, we get
\begin{align*}
\omega_1 &
=
(1 + i)
\dfrac{
e^{i \psi_1} (e^{i \theta_2} - 1) + e^{i \psi_2} (1 - e^{i \theta_1}) + e^{i \theta_2} - e^{i \theta_1}
}{
-2i \sin(\theta_1) e^{i \psi_2}
+ 2i \sin(\theta_2) e^{i \psi_1}
+ 2i \sin(\theta_2 - \theta_1)
}
\eqqcolon (1+i) \dfrac{\mathcal{N}_1}{\mathcal{D}},
\\
\omega_2 &
=
(1 - i)
\dfrac{
e^{i \psi_1} (e^{-i \theta_2} - 1) + e^{i \psi_2} (1 - e^{-i \theta_1}) + e^{-i \theta_2} - e^{-i \theta_1}
}{
-2i \sin(\theta_1) e^{i \psi_2}
+ 2i \sin(\theta_2) e^{i \psi_1}
+ 2i \sin(\theta_2 - \theta_1)
}
\eqqcolon (1+i) \dfrac{\mathcal{N}_2}{\mathcal{D}}.
\end{align*}
We further suppose that $\omega \in I(\theta_3) \sqcup I_-(\theta_3)$.
Therefore, by Proposition \ref{proposition::standard-spheres}, we get
\[
1 \ge 
\dfrac{
\left|
(1-i) \cdot \omega_1 \cdot e^{-i \theta_3} - (1+i) \cdot \omega_2 \cdot e^{i\theta_3} - 2
\right|
}{
\left|
(1-i) \cdot \omega_1 - (1+i) \cdot \omega_2 - 2
\right|
}
=
\dfrac{
\left|
\mathcal{N}_1 e^{-i \theta_3} - \mathcal{N}_2 e^{i\theta_3} - \mathcal{D}
\right|
}{
\left|
\mathcal{N}_1 - \mathcal{N}_2 - \mathcal{D}
\right|
}.
\]
Equivalently, we get
$0 \ge |\mathcal{N}_1 e^{-i\theta_3}-\mathcal{N}_2 e^{i\theta_3}-\mathcal{D}|^2 - |\mathcal{N}_1-\mathcal{N}_2-\mathcal{D}|^2$, where
\begin{align*}
\mathcal{N}_1 e^{-i\theta_3}-\mathcal{N}_2 e^{i\theta_3}-\mathcal{D}
=&~
2 i \cdot e^{i \psi_1}
\left(
\sin(\theta_2 - \theta_3) + \sin(\theta_3) - \sin(\theta_2) 
\right) \\
& + 2 i \cdot e^{i \psi_2}
\left(
\sin(\theta_3 - \theta_1) + \sin(\theta_1) - \sin(\theta_3)
\right) \\
& + 2 i \cdot
\left(
\sin(\theta_2 - \theta_3) + \sin(\theta_3 - \theta_1) + \sin(\theta_1 - \theta_2) 
\right), \\
\mathcal{N}_1-\mathcal{N}_2-\mathcal{D}
=&
\left.
\left(\mathcal{N}_1 e^{-i\theta_3}-\mathcal{N}_2 e^{i\theta_3}-\mathcal{D}
\right)
\right|_{\theta_3 = 0}
= 2i \cdot \left(
\sin(\theta_2) - \sin(\theta_1) + \sin(\theta_1 - \theta_2)
\right).
\end{align*}
We suppose that
$\left|\mathcal{N}_1 e^{-i\theta_3}-\mathcal{N}_2 e^{i\theta_3}-\mathcal{D}\right|^2 / 4 = \delta_0 + \delta_1 \cos(\psi_1) + \delta_2 \cos(\psi_2) + \delta_{1,2} \cos(\psi_1 - \psi_2)$
and suppose that
$\left|\mathcal{N}_1-\mathcal{N}_2-\mathcal{D}\right|^2 / 4 = \delta'_0$.
The coefficients $\delta_0$, $\delta_1$, $\delta_2$, $\delta_{1,2}$ and $\delta'_0$ are given by the following expressions:
\begin{align*}
\delta_0 =&~
\left(
\sin(\theta_2-\theta_3) + \sin(\theta_3) - \sin(\theta_2)
\right)^2
+
\left(
\sin(\theta_3-\theta_1) + \sin(\theta_1) - \sin(\theta_3)
\right)^2
\\
&
+
\left(
\sin(\theta_2 - \theta_3)
+ \sin(\theta_3 - \theta_1)
+ \sin(\theta_1 - \theta_2)
\right)^2,
\\
\delta_1 =&~
2 \cdot \left(
\sin(\theta_2-\theta_3)+\sin(\theta_3)-\sin(\theta_2)
\right)
\cdot
\left(
\sin(\theta_2-\theta_3)+\sin(\theta_3-\theta_1)+\sin(\theta_1-\theta_2)
\right) \\
=&~ 32 \cdot 
\sin{\left(\dfrac{\theta_2}{2}\right)}
\sin{\left(\dfrac{\theta_3}{2}\right)}
\sin{\left(\dfrac{\theta_1 - \theta_2}{2}\right)}
\sin{\left(\dfrac{\theta_1 - \theta_3}{2}\right)}
\sin^2{\left(\dfrac{\theta_2 - \theta_3}{2}\right)}, \\
\delta_2 =&~
2 \cdot \left(
\sin(\theta_3-\theta_1)+\sin(\theta_1)-\sin(\theta_3)
\right)
\cdot
\left(
\sin(\theta_2-\theta_3)+\sin(\theta_3-\theta_1)+\sin(\theta_1-\theta_2)
\right) \\
=&~ -32 \cdot
\sin{\left(\dfrac{\theta_1}{2}\right)}
\sin{\left(\dfrac{\theta_3}{2}\right)}
\sin{\left(\dfrac{\theta_1-\theta_2}{2}\right)}
\sin^2{\left(\dfrac{\theta_1-\theta_3}{2}\right)}
\sin{\left(\dfrac{\theta_2-\theta_3}{2}\right)}, \\
\delta_{1,2} =&~
2 \cdot \left(
\sin(\theta_2-\theta_3)+\sin(\theta_3)-\sin(\theta_2)
\right)
\cdot
\left(
\sin(\theta_3-\theta_1)+\sin(\theta_1)-\sin(\theta_3)
\right) \\
=&~ -32 \cdot
\sin{\left(\dfrac{\theta_1}{2}\right)}
\sin{\left(\dfrac{\theta_2}{2}\right)}
\sin^2{\left(\dfrac{\theta_3}{2}\right)}
\sin{\left(\dfrac{\theta_1-\theta_3}{2}\right)}
\sin{\left(\dfrac{\theta_2-\theta_3}{2}\right)}, \\
\delta'_0
=&~
\left(
\sin(\theta_2) - \sin(\theta_1) + \sin(\theta_1 - \theta_2)
\right)^2.
\end{align*}
We also observe that
\begin{align*}
0 =&~
\Big(
\big(
\sin(\theta_2-\theta_3)+\sin(\theta_3)-\sin(\theta_2)
\big)
+
\big(
\sin(\theta_3-\theta_1)+\sin(\theta_1)-\sin(\theta_3)
\big) \\
&-
\big(
\sin(\theta_2-\theta_3)+\sin(\theta_3-\theta_1)-\sin(\theta_1-\theta_2)
\big)
\Big)^2
-
\Big(
\sin(\theta_1-\theta_2)+\sin(\theta_2)-\sin(\theta_1)
\Big)^2 \\
=&~ \delta_0 - \delta_1 - \delta_2 + \delta_{1,2} - \delta'_0.
\end{align*}
Thus, we make the substitutions $X \coloneqq \cot(\psi_1/2)$ and $Y \coloneqq \cot(\psi_2/2)$. Define the function
\begin{align*}
\widehat{Q}(X,Y) =&~ \dfrac{(X^2+1)(Y^2+1)}{4}
\left(
|\mathcal{N}_1 e^{-i\theta_3}-\mathcal{N}_2 e^{i\theta_3}-\mathcal{D}|^2 - |\mathcal{N}_1-\mathcal{N}_2-\mathcal{D}|^2
\right) \\
=&~
(\delta_0 - \delta'_0) \cdot (X^2+1)(Y^2+1)
+ \delta_1 \cdot (X^2-1)(Y^2+1)
+ \delta_2 \cdot (X^2+1)(Y^2-1) \\
&+ \delta_{1,2} \cdot \left( \left(X^2-1\right) \left(Y^2-1\right)+4XY\right) \\
=&~ (\delta_0 - \delta'_0 + \delta_1 + \delta_2 + \delta_{1,2}) \cdot X^2Y^2
+ (\delta_0 - \delta'_0 + \delta_1 - \delta_2 - \delta_{1,d 2}) \cdot X^2 \\
&+ (\delta_0 - \delta'_0 - \delta_1 + \delta_2 - \delta_{1,2}) \cdot Y^2
+ 4\delta_{1,2} \cdot XY
+ (\delta_0 - \delta'_0 - \delta_1 - \delta_2 + \delta_{1,2}).
\end{align*}
A straightforward calculation then yields the desired inequality.
\end{proof}

\begin{definition}
For every $\theta_3 \in [0, 2\pi]$,
the function $Q_{\theta_1,\theta_2,\theta_3}$ is defined by
\[
Q_{\theta_1,\theta_2,\theta_3}(X,Y) = 
Q(X,Y) \coloneqq c_{22}'' X^2 Y^2 + c_{20}'' X^2 + c_{02}'' Y^2 + 2 c_{11}'' XY,
\]
where
\begin{align*}
&c_{22}'' = 
\sin^2{\left( \frac{\theta_1 - \theta_2}{2} \right)}
\sin{\left( \frac{\theta_3 - (\theta_1 + \theta_2)}{2} \right)},
\quad
&c_{11}'' =
- \sin{\left( \frac{\theta_1}{2} \right)}
\sin{\left( \frac{\theta_2}{2} \right)}
\sin{\left( \frac{\theta_3}{2} \right)}, \\
&c_{20}'' = 
\sin^2{\left( \frac{\theta_2}{2} \right)}
\sin{\left( \frac{\theta_1 - \theta_2 + \theta_3}{2} \right)},
\quad
&c_{02}'' = 
\sin^2{\left( \frac{\theta_1}{2} \right)}
\sin{\left( \frac{\theta_2 - \theta_1 + \theta_3}{2} \right)}.
\end{align*}
\end{definition}

\begin{corollary}
\label{corollary::Q(X,Y)}
For every $\theta_3 \not\in \{0, \theta_1, \theta_2, 2\pi\}$,
the intersection $I(\theta_1) \cap I(\theta_2) \cap I(\theta_3) \subset I(\theta_1) \cap I(\theta_2)$,
parameterised by $(\psi_1, \psi_2) \in (0,2\pi)^2$,
is described by
$Q(X,Y) = 0$
and $W(X,Y) \le 0$,
where $X \coloneqq \cot{(\psi_1 / 2)}$ and $Y \coloneqq \cot{(\psi_2 / 2)}$.
\end{corollary}

\begin{proof}
This corollary follows immediately from Proposition \ref{proposition::intersection-of-three-isometric-spheres}.
\end{proof}

\begin{lemma}
\label{lemma::parameters-Q}
For every $\theta_3 \in [0, 2\pi]$,
the coefficients of $Q(X,Y)$ satisfy $c_{11}''^2 - c_{20}'' c_{02}'' > 0$.
\end{lemma}

\begin{proof}
This follows from 
\(
c_{11}'' - c_{20}'' c_{02}'' = \sin^2{(\theta_1/2)}
\sin^2{(\theta_2/2)}
\sin^2{((\theta_1 - \theta_2)/2)}
\).
\end{proof}

When $c_{22}'' \neq 0$, we also define the coefficients
$a \coloneqq - c_{20}'' / c_{22}''$, $b \coloneqq - c_{11}'' / c_{22}''$ and $c \coloneqq - c_{02}'' / c_{22}''$.

\begin{lemma}
\label{lemma::a+2b+c=-1}
Suppose that $\theta_3 \in [0, 2\pi]$ satisfies $c_{22}'' \neq 0$.
Then
the coefficients $a, b, c$ satisfy the relations $b^2 > ac$ and $a+2b+c = -1$.
\end{lemma}

\begin{proof}
The lemma follows immediately from
$c_{20}'' + c_{02}'' + 2 c_{11}'' = c_{22}''$.
\end{proof}

\begin{lemma}
\label{lemma::fixed_point}
Suppose that $\theta_1 \neq \pi$ and $\theta_2 \neq \pi$.
Define the point
\[
p_\nabla = \left(X_\nabla, Y_\nabla\right)
\coloneqq \left(\tan{\left(\frac{\theta_1}{2}\right)}, \tan{\left(\frac{\theta_2}{2}\right)}\right).
\]
Then, for every $\theta_3 \in [0,2\pi]$,
the point $p_\nabla$ satisfies
$Q(X_\nabla, Y_\nabla) = 0$
and
$W(X_\nabla, Y_\nabla) > 0$.
\end{lemma}

\begin{proof}
This follows from a direct computation.
\end{proof}

Assume that $c_{22}'' \neq 0$.
We define two families of intervals and associated curves as follows.

\begin{definition}
Suppose that $a \neq 0$.
For each $\epsilon, \tau, \sigma \in \{\pm 1\}$,
let $J^x_{\epsilon,\tau,\sigma} \subset \R_{\ge 0}$
be the maximal interval containing $0$ such that 
for all non-zero $\chi \in J^x_{\epsilon,\tau,\sigma}$
we have $(\sigma \chi - 2b)^2 - 4ac > 0$ and $\widehat{x}_\tau(\sigma \chi) > 0$, where
\[
\widehat{x}_{\pm 1}(\sigma \chi) =
\widehat{x}_{\pm}(\sigma \chi) \coloneqq
\frac{\sigma \chi}{2a} \left(
\left(\sigma \chi - 2b \right)
\pm
\sigma \sqrt{\left( \sigma \chi - 2b \right)^2 - 4ac}
\right).
\]
\end{definition}

\begin{definition}
Suppose that $c \neq 0$.
For each $\epsilon, \tau, \sigma \in \{\pm 1\}$,
let $J^y_{\epsilon,\tau,\sigma} \subset \R_{\ge 0}$ be the maximal interval containing $0$ such that for all non-zero $\chi \in J^y_{\epsilon,\tau,\sigma}$ we have
$(\sigma \chi - 2b)^2 - 4ac > 0$ and $\widehat{y}_\tau(\sigma \chi) > 0$, where
\[
\widehat{y}_{\pm 1}(\sigma \chi) =
\widehat{y}_{\pm}(\sigma \chi) \coloneqq
\frac{\sigma \chi}{2c} \left(
\left( \sigma \chi - 2b \right)
\pm
\sigma
\sqrt{\left( \sigma \chi - 2b \right)^2 - 4ac}
\right).
\]
\end{definition}

\begin{definition}
The path $\gamma_{\epsilon,\tau,\sigma}: J^x_{\epsilon,\tau,\sigma} \rightarrow \R^2$ and
the path $\delta_{\epsilon,\tau,\sigma}: J^y_{\epsilon,\tau,\sigma} \rightarrow \R^2$
are given by
\[
\gamma_{\epsilon,\tau,\sigma}:
\chi \mapsto
\left( \epsilon \cdot \sqrt{\widehat{x}_\tau (\sigma \chi)}, \epsilon \cdot \frac{\sigma \chi}{\sqrt{\widehat{x}_\tau (\sigma \chi)}} \right)
\quad
\text{and}
\quad
\delta_{\epsilon,\tau,\sigma}:
\chi \mapsto
\left( \epsilon \cdot \frac{\sigma \chi}{\sqrt{\widehat{y}_\tau (\sigma \chi)}},
\epsilon \cdot \sqrt{\widehat{y}_\tau (\sigma \chi)}
\right).
\]
\end{definition}

We will show that the intersection $I(\theta_1) \cap I(\theta_2) \cap I(\theta_3)$ consists of four subarcs of these curves.

\begin{lemma}
\label{lemma::aneq0_but_c=0}
Suppose that $\theta_3 \in (0,2\pi) \setminus \{\theta_1,\theta_2\}$
satisfies $c_{22}'' \neq 0$, $a \neq 0$ but $c = 0$,
Then, there exists an open domain $\Omega \subset \R^2$ such that $\{W \le 0\} \subset \Omega$, $\{X=0\} \subset \Omega$ and the locus of $X = \dfrac{2b \cdot Y}{Y^2 - a}$ intersects the domain $\Omega$ in the concatenation of two curves in the form $\gamma_{\epsilon,\tau,\sigma}$.
\end{lemma}

\begin{lemma}
\label{lemma::cneq0_but_a=0}
Suppose that $\theta_3 \in (0,2\pi) \setminus \{\theta_1,\theta_2\}$
satisfies $c_{22}'' \neq 0$, $c \neq 0$ but $a = 0$.
Then, there exists an open domain $\Omega \subset \R^2$ such that $\{W\le 0\} \subset \Omega$, $\{Y=0\} \subset \omega$ and the locus of $Y = \dfrac{2b \cdot X}{X^2 - c}$ intersects the domain $\Omega$ in the concatenation of two curves in the form $\delta_{\epsilon,\tau,\sigma}$.
\end{lemma}

We will prove only Lemma \ref{lemma::aneq0_but_c=0} but omit the proof of Lemma \ref{lemma::cneq0_but_a=0}, which is quite similar.

\begin{proof}[Proof of Lemma \ref{lemma::aneq0_but_c=0}]
Recall $\widehat{x}_{\pm}(\sigma\chi)
=
\frac{\sigma\chi}{2a} \left(
\left( \sigma\chi - 2b \right)
\pm
|\sigma\chi - 2b|
\right)$,
one of which, say $\widehat{x}_\tau(\sigma\chi) = \dfrac{\sigma\chi}{a} (\sigma\chi - 2b)$
while the other is $0$,
where $\tau = \sgn(\sigma \chi - 2b)$.
We first claim that the locus of $X = \dfrac{2b \cdot Y}{Y^2 - a}$
is contained within the union of the curves $\gamma_{\epsilon,\tau,\sigma}$ with
$\epsilon, \tau, \sigma \in \{\pm 1\}$.
To see this,
consider a point $(X,Y)$
on this locus.
Set $\chi = |XY|$ and $\sigma\chi = XY$.
Then $\sigma\chi = \dfrac{2b \cdot Y}{Y^2 - a} Y$
and consequently, we get $Y^2 = \dfrac{\sigma\chi a}{\sigma\chi - 2b}$
and $X^2 = \dfrac{\chi^2}{Y^2} = \widehat{x}_\tau(\sigma\chi)$
with $\tau = \sgn(XY - 2b)$.
Therefore, the point $(X,Y)$ lies on $\gamma_{\sgn(X),\tau,\sigma}$.

We consider the two cases for the sign of $a$.

\textbf{Case 1:} $a < 0$.
In this case, the dominator $Y^2 - a$ is always positive.
We have $\epsilon = \sgn(X) = \sgn(bY)$, $\sigma = \sgn(XY) = \sgn(b)$ and $\tau = \sgn(\sigma\chi - 2b) = \sgn(ab)$.
Thus, the locus of $X = \dfrac{2b \cdot Y}{Y^2 - a}$
is the concatenation of the curves
$\gamma_{\sgn(b),\sgn(b),\sgn(ab)}$ and $\gamma_{-\sgn(b),\sgn(b),\sgn(ab)}$.
The open domain $\Omega$ is defined to be the whole space $\R^2$.

\textbf{Case 2:} $a > 0$.
From Lemma \ref{lemma::a+2b+c=-1} (i.e. $a+2b+c=-1$) and $c=0$, we have $a + 2b = -1$, which implies $2b < -1$.
The locus over $|Y| < \sqrt{a}$ is the concatenation of the two curves $\gamma_{-\sgn(bY),-\sgn(b),-\sgn(ab)}$.
It suffices to show that the locus over $|Y| > \sqrt{a}$ lies outside the disk $W \le 0$.
For $Y > \sqrt{a}$, the product $XY = \dfrac{2b \cdot Y^2}{Y^2 - a} = \cfrac{2b}{1 - a/Y^2} < 0$ is strictly increasing, tending to $2b$ as $Y\rightarrow +\infty$.
By Lemma \ref{lemma::W_with_XY=1}, we conclude that $W \le 0$.
symmetric argument applies to $Y < -\sqrt{a}$.
\end{proof}

We therefore arrive at the following results.

\begin{theorem}
\label{theorem::intersection_of_three}
For every $\theta_3 \in [0,2\pi]$,
there exists an open domain $\Omega \subset \R^2$ such that $\{W\le 0\} \subset \Omega$
and the locus of $Q=0$ in $\Omega$ is homeomorphic to a crossing, i.e.,
the subset
\[
\left\{
(x,0), (0,y) ~\middle|~ |x| \le 1, |y| \le 1
\right\} 
\subset
\left\{
(x,y) \in \R^2 ~\middle|~ x^2 + y^2 \le 1
\right\}.
\]
More precisely, if $\theta_3$ satisfies $c_{22}'' \neq 0$ and $abc \neq 0$, then the locus of $Q(X,Y) = 0$ in $\Omega$ is the union of the underlying curves given in the following table.
\begin{center}
\begin{tabular}{c|c|c|c|c}
\hline
$a$ & $b$ & $c$ & underlying curves & domain of $\chi$ \\
\hline
$>0$ &  & $<0$ & $\gamma_{\epsilon,+1,+1} \cup \gamma_{\epsilon,+1,-1}$ & $J^x_{\epsilon,+1,+1} = J^x_{\epsilon,+1,-1} = [0,\infty)$ \\
$<0$ &  & $>0$ & $\gamma_{\epsilon,-1,+1} \cup \gamma_{\epsilon,-1,-1}$ & $J^x_{\epsilon,-1,+1} = J^x_{\epsilon,-1,-1} = [0,\infty)$ \\
$>0$ & $<0$ & $>0$ & $\gamma_{\epsilon,+1,+1} \cup \gamma_{\epsilon,-1,+1}$ & $J^x_{\epsilon,+1,+1} = J^x_{\epsilon,-1,+1} = [0,\infty)$ \\
$>0$ & $>0$ & $>0$ & $\emptyset$ & \\
$<0$ & $>0$ & $<0$ & $\gamma_{\epsilon,+1,+1} \cup \gamma_{\epsilon,-1,+1}$ & $J^x_{\epsilon,+1,+1} = J^x_{\epsilon,-1,+1} = [0,\lambda_1)$ \\
$<0$ & $<0$ & $<0$ & $\gamma_{\epsilon,+1,-1} \cup \gamma_{\epsilon,-1,-1}$ & $J^x_{\epsilon,+1,-1} = J^x_{\epsilon,-1,-1} = [0,-\lambda_2)$ \\
\hline
\end{tabular}
\end{center}
where $\epsilon \in \{\pm 1\}$, $\lambda_1 \coloneqq 2b-2\sqrt{ac}$
and $\lambda_2 \coloneqq 2b + 2\sqrt{ac}$. 
\end{theorem}

\begin{proposition}
\label{proposition::Q=0_and_W=0}
For every $\theta_3 \in [0, 2\pi]$, the locus $Q=0$ intersects $W=0$ in four points.
\end{proposition}

\begin{corollary}
\label{corollary::intersection_of_three}
For every $\theta_3 \not\in \{\theta_1, \theta_2\}$,
the intersection $I(\theta_1) \cap I(\theta_2) \cap I(\theta_3) \subset I(\theta_1) \cap I(\theta_2)$
is homeomorphic to a crossing.
\end{corollary}

\begin{proof}
This corollary follows immediately from Theorem 
\ref{theorem::intersection_of_three}
and Proposition \ref{proposition::Q=0_and_W=0}.
\end{proof}

Using the parameters $(\psi_1, \psi_2) \in (0, 2\pi)^2$ for the circle $W=0$ and the locus $Q=0$, we depict various configurations of $(\theta_1, \theta_2, \theta_3)$ in Figure \ref{figure::crossings}.
These figures also demonstrate the intersection of ${W \le 0}$ and ${Q = 0}$, which is a crossing as claimed in Corollary \ref{corollary::intersection_of_three}.

\begin{figure}[htb]
 \centering
 \begin{subfigure}[b]{0.3\textwidth}
  \centering
  \includegraphics[width=\textwidth]{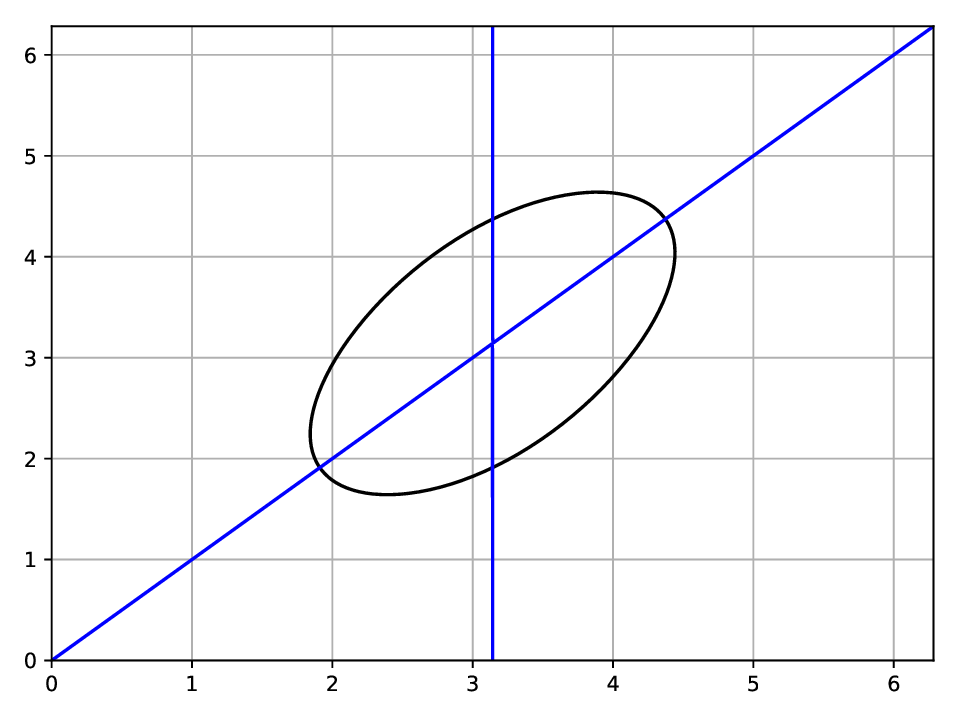}
  \caption{$(\theta_1,\theta_2,\theta_3) = \left(\dfrac{\pi}{2}, \pi, \dfrac{3\pi}{2}\right)$}
 \end{subfigure}
 \hfill
 \begin{subfigure}[b]{0.3\textwidth}
  \centering
  \includegraphics[width=\textwidth]{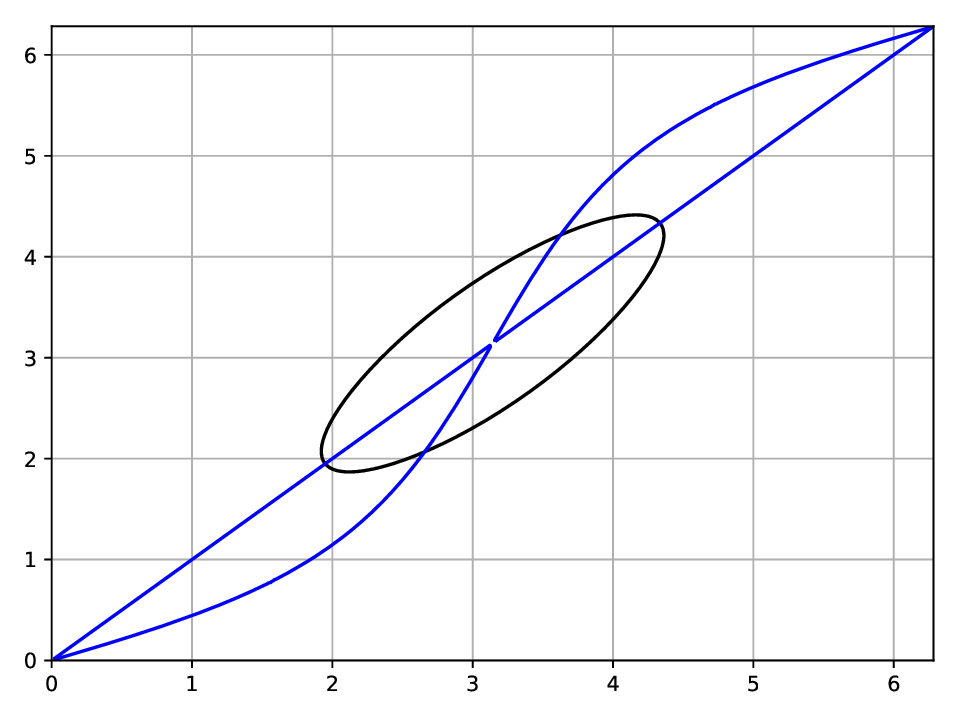}
  \caption{$(\theta_1,\theta_2,\theta_3) = \left(\dfrac{\pi}{2}, \dfrac{3\pi}{4}, \dfrac{5\pi}{4}\right)$}
 \end{subfigure}
 \hfill
 \begin{subfigure}[b]{0.3\textwidth}
  \centering
  \includegraphics[width=\textwidth]{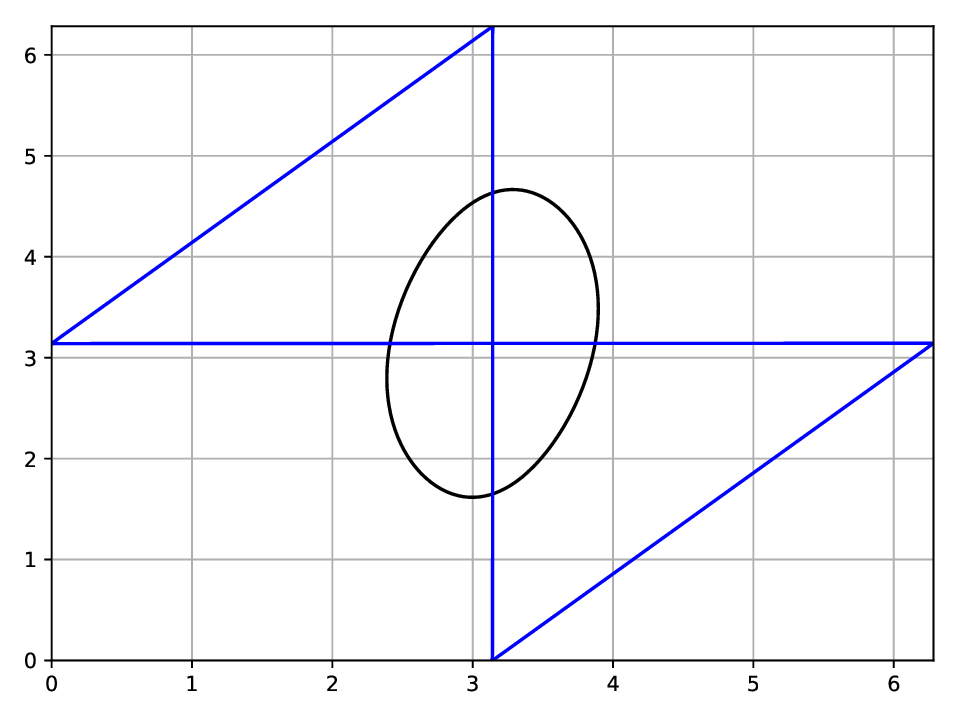}
  \caption{$(\theta_1,\theta_2,\theta_3) = \left(\dfrac{\pi}{4}, \dfrac{5\pi}{4}, \pi\right)$}
 \end{subfigure}
 \\
 \vspace{0.5cm}
 \begin{subfigure}[b]{0.3\textwidth}
  \centering
  \includegraphics[width=\textwidth]{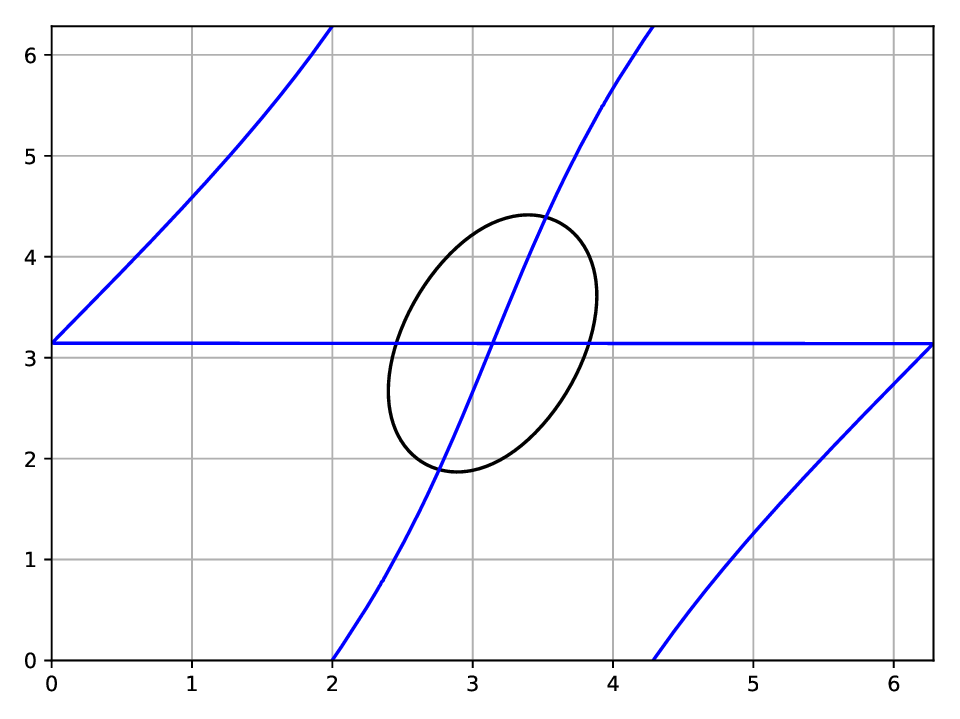}
  \caption{$(\theta_1,\theta_2,\theta_3) = \left(\dfrac{\pi}{4}, \dfrac{3\pi}{4}, \dfrac{\pi}{2}\right)$}
 \end{subfigure}
 \hfill
 \begin{subfigure}[b]{0.3\textwidth}
  \centering
  \includegraphics[width=\textwidth]{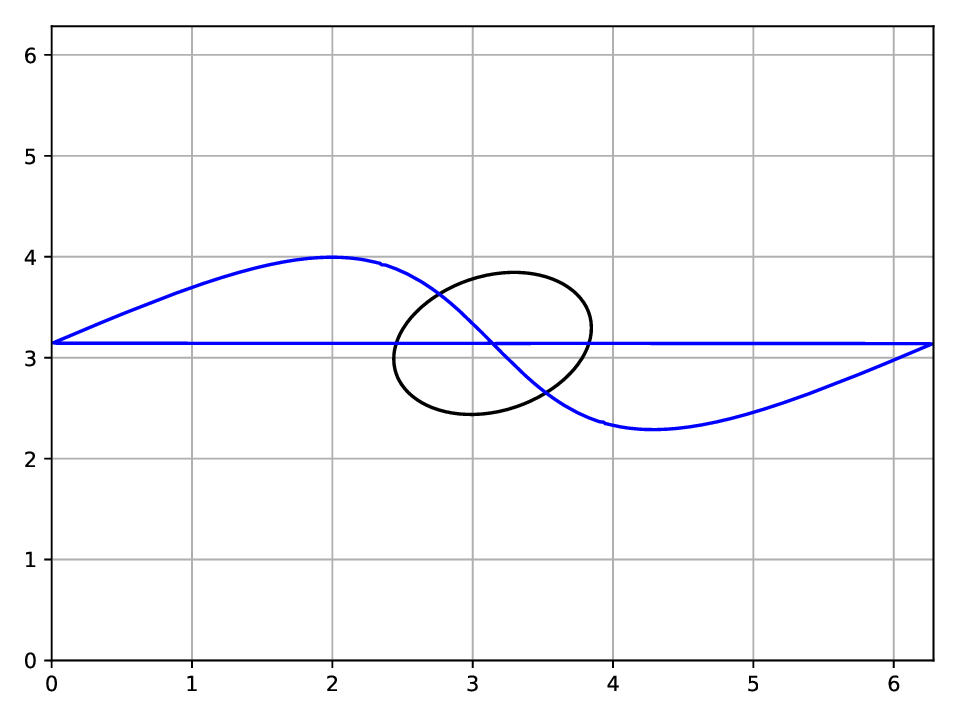}
  \caption{$(\theta_1,\theta_2,\theta_3) = \left(\dfrac{\pi}{4}, \dfrac{7\pi}{4}, \dfrac{3\pi}{2}\right)$}
 \end{subfigure}
 \hfill
 \begin{subfigure}[b]{0.3\textwidth}
  \centering
  \includegraphics[width=\textwidth]{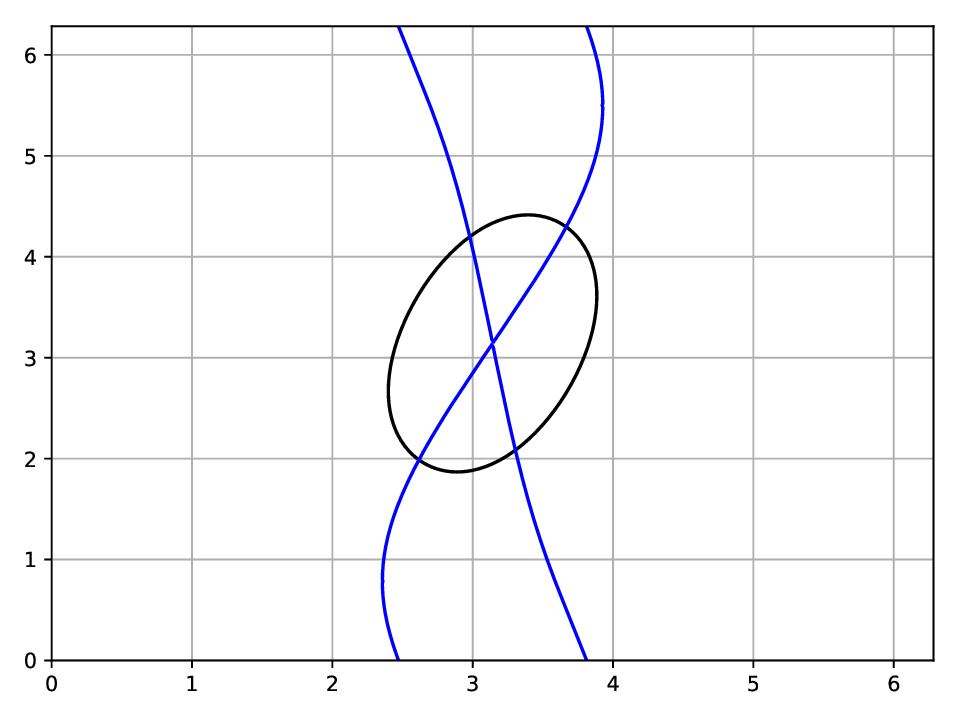}
  \caption{$(\theta_1,\theta_2,\theta_3) = \left(\dfrac{\pi}{4}, \dfrac{3\pi}{4}, \dfrac{7\pi}{4}\right)$}
 \end{subfigure}
 \\
 \vspace{0.5cm}
 \begin{subfigure}[b]{0.3\textwidth}
  \centering
  \includegraphics[width=\textwidth]{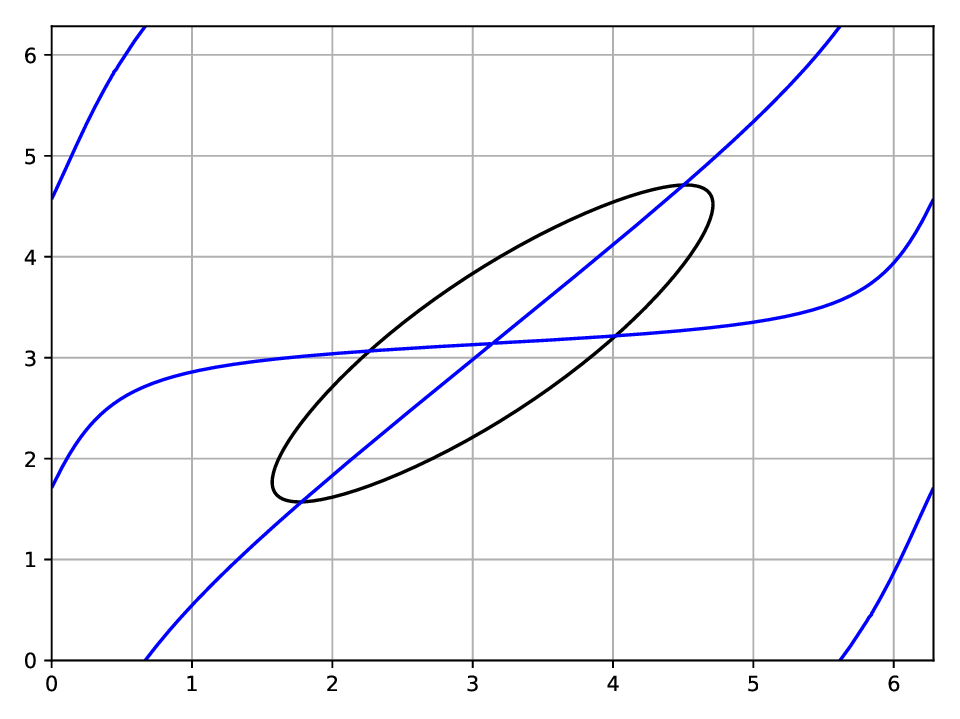}
  \caption{$(\theta_1,\theta_2,\theta_3) = \left(\dfrac{6\pi}{7}, \dfrac{8\pi}{7}, \dfrac{\pi}{3}\right)$}
 \end{subfigure}
 \hfill
 \begin{subfigure}[b]{0.3\textwidth}
  \centering
  \includegraphics[width=\textwidth]{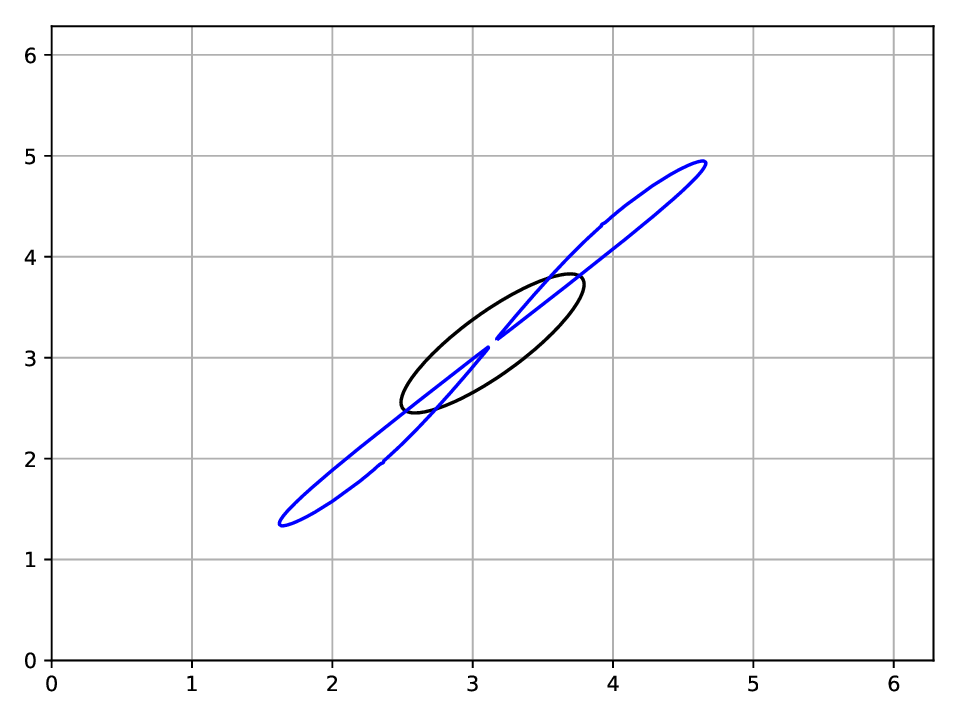}
  \caption{$(\theta_1,\theta_2,\theta_3) = \left(\dfrac{\pi}{4}, \dfrac{3\pi}{8}, \dfrac{3\pi}{4}\right)$}
 \end{subfigure}
 \hfill
 \begin{subfigure}[b]{0.3\textwidth}
  \centering
  \includegraphics[width=\textwidth]{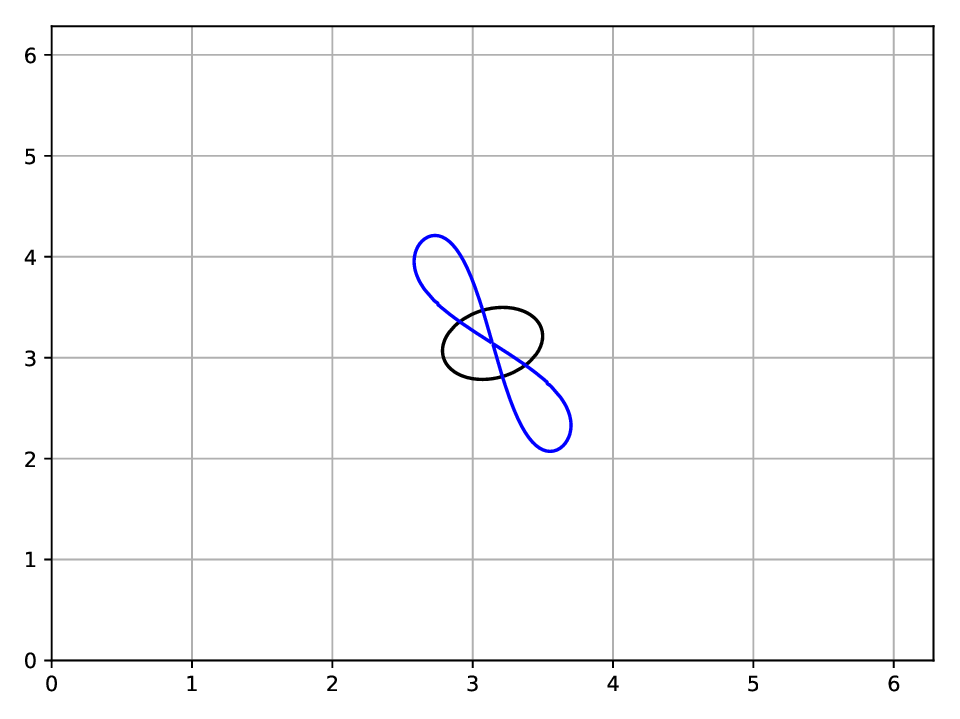}
  \caption{$(\theta_1,\theta_2,\theta_3) = \left(\dfrac{\pi}{8}, \dfrac{15\pi}{8}, \dfrac{3\pi}{8}\right)$}
 \end{subfigure}
 \caption{We list nine different selections of $(\theta_1, \theta_2, \theta_3)$. In the corresponding figures, parameterised by $(\psi_1, \psi_2) \in (0, 2\pi)^2$ (with $\psi_1$ on the horizontal axis and $\psi_2$ on the vertical axis), the intersection $I(\theta_1) \cap I(\theta_2)$ is a disk bounded by the provided circle. Although the locus of $Q(X,Y)=0$ varies significantly across configurations, where $X=\cot(\psi_1/2)$ and $Y=\cot(\psi_2/2)$, it always forms a crossing intersection with the disk $W(X,Y) \le 0$. Integer grids are also provided in each figure.}
 \label{figure::crossings}
\end{figure}

\begin{proof}[Proof of Theorem \ref{theorem::intersection_of_three}]
We consider the following cases.

\textbf{Case 1:} $c_{22}'' = 0$.
In this case, either $\theta_3 - (\theta_1 + \theta_2) = 0$ or $\theta_3 - (\theta_1 + \theta_2) = -2\pi$.

Suppose that $\theta_3 = \theta_1 + \theta_2$.
Since $c_{20}'' = \sin^2{(\theta_2/2)} \sin{\left( \theta_1 \right)}$ and
$c_{02}'' = \sin^2{(\theta_1/2)} \sin{\left( \theta_2 \right)}$,
at least one of $c_{20}''$, $c_{02}''$ is non-zero.
If $c_{20}''=0$ but $c_{02}''\neq 0$,
then $Q(X,Y) = Y (c_{02}''Y + 2 c_{11}''X)$ and therefore the locus of $Q=0$ is given by $\{Y = 0\} \cup \left\{X = - \dfrac{c_{02}''}{2c_{11}''} Y\right\}$.
If $c_{02}''=0$ but $c_{20}'' \neq 0$, then
$Q(X,Y) = X (c_{20}''X + 2c_{11}''Y)$
and therefore the locus of $Q=0$ is given by $\{X = 0\} \cup \left\{ X = -\dfrac{2c_{11}''}{c_{20}''} Y \right\}$;
see Figure \ref{figure::crossings} - (a) for an example.
Otherwise, we suppose that $c_{20}'' c_{02}'' \neq 0$.
Then $Q(X,Y) = c_{20}''X^2 + c_{02}''Y^2 + 2c_{11}''XY$,
which implies that
\[
Y = \frac{-c_{11}'' \pm \sqrt{ c_{11}''^2 - c_{20}'' c_{02}''}}{c_{02}''} X;
\]
see Figure \ref{figure::crossings} - (b) for an example.
We conclude that the
locus of $Q=0$ is given by the union of two distinct lines through $(X,Y) = (0,0)$, which is a crossing.
The open domain is then defined to be $\Omega=\R^2$.

The case that $\theta_3 = (\theta_1 + \theta_2) - 2\pi$ has the same conclusion.

\textbf{Case 2:} $c_{22}'' \neq 0$ but $c_{20}'' c_{02}'' = 0$.

Suppose that $c_{20}'' = c_{02}'' = 0$.
We get $\sin{\left(\dfrac{\theta_1 - \theta_2 + \theta_3}{2}\right)} = \sin{\left(\dfrac{\theta_2 - \theta_1 + \theta_3}{2}\right)} = 0$
and thus $\theta_3 = \pi$.
Therefore $Q(X,Y) = XY (c_{22}'' XY + 2 c_{11}'')$
and hence the locus
is given by $\{X=0\} \cup \{Y=0\} \cup \{XY = -2 c_{11}'' / c_{22}''\}$.
As $b = -2 c_{11}'' / c_{22}$ and $a = c = 0$, by Lemma \ref{lemma::a+2b+c=-1} and Lemma \ref{lemma::W_with_XY=1},
the desired domain $\Omega$ is obtained from $\R^2$ by removing regions bounded by the third branch.
See Figure \ref{figure::crossings} - (c) for an example.

Suppose that $c_{20}'' = 0$
but $c_{02}'' \neq 0$.
Then $Q(X,Y) = Y \cdot \left( \left( c_{22}'' X^2 + c_{02}'' \right) Y + 2 c_{11}'' X\right)$.
The locus is therefore given by
$\{Y = 0\} \cup \left\{ Y = \dfrac{-2c_{11}'' X}{c_{22}'' X^2 + c_{02}''} \right\}$
where the second branch is also given by $Y = \dfrac{2b \cdot X}{X^2 - c}$.
The locus of $W=0$ is then given by Lemma \ref{lemma::cneq0_but_a=0}.
See Figure \ref{figure::crossings} - (d) and (e) for example.

The case that $c_{02}''=0$ but $c_{20}'' \neq 0$ has the same argument.

\textbf{Case 3:} $c_{22}'' \neq 0$, $c_{20}'' \neq 0$ and $c_{02}'' \neq 0$.
In this case, we first observe that $X=0$ if and only if $Y=0$.
From now on, we assume that $XY \neq 0$.

Taking $\lambda \coloneqq XY \neq 0$, we rewrite $Q(X,Y) = 0$ as
$a \widehat{x}^2 + (2b\lambda - \lambda^2) \widehat{x} + c\lambda^2 = 0$,
where $\widehat{x} \coloneqq X^2$.
The real solution for $\widehat{x}$, if exists, is therefore given by $\widehat{x}_{\pm}(\lambda)$.
The existence of $\widehat{x}$ depends on if $\Delta \coloneqq \lambda^2 \cdot \left( \left(\lambda - 2b\right)^2 - 4ac \right) \ge 0$
but the existance of $(X,Y)$ also requires that $\widehat{x} > 0$ and $W \le 0$.

\textbf{Case 3.1:} $ac < 0$.

In this case, we have $\Delta > 0$ and
$\left| \lambda \left(\lambda - 2b\right) \right| < |\lambda| \cdot \sqrt{(\lambda - 2b)^2 - 4ac}$.
Therefore,
if $a > 0$, $c < 0$ and $\lambda > 0$, then the only solution for $\widehat{x}$ is $\widehat{x}_+$;
if $a > 0$, $c < 0$ and $\lambda < 0$, then the only solution for $\widehat{x}$ is $\widehat{x}_+$;
if $a < 0$, $c > 0$ and $\lambda > 0$, then the only solution for $\widehat{x}$ is $\widehat{x}_-$;
if $a < 0$, $c > 0$ and $\lambda < 0$, then the only solution for $\widehat{x}$ is $\widehat{x}_-$.
We thus defined $\Omega$ to be $\R^2$.
See Figure \ref{figure::crossings} - (f) for an example.

\textbf{Case 3.2:} $a > 0$, $c > 0$ and $b < 0$.

In this case, the discriminant $\Delta \ge 0$ if and only if either $\lambda \le \lambda_1 \coloneqq 2b - 2\sqrt{ac}$ or $\lambda \ge \lambda_2 \coloneqq 2b + 2\sqrt{ac}$, where $\lambda_1 < \lambda_2 < 0$ as $b^2 > ac$.
As $ac > 0$, when $\Delta \ge 0$,
we get 
$\left| \lambda \left(\lambda - 2b\right) \right| > |\lambda| \cdot \sqrt{(\lambda - 2b)^2 - 4ac}$.
Therefore, there exists a positive solution for $\widehat{x}$ if and only if
$\lambda (\lambda - 2b) > 0$,
i.e., either $\lambda > 0$ or $\lambda < 2b$.

We conclude with the following:
when $\lambda > 0$, there are two available solutions $\widehat{x}_+$ and $\widehat{x}_-$;
when $\lambda_1 < \lambda < 0$, there is no solution for $\widehat{x}$.
It suffices to show that $(X,Y)$ satisfying $\lambda = XY \le \lambda_1$ does not belong to the disk $W \le 0$.
Notice that $\lambda_1 = 2b-2\sqrt{ac} < 2b < 2b+a+c$.
We therefore get that $W \le 0$ by Lemma \ref{lemma::a+2b+c=-1} and Lemma \ref{lemma::W_with_XY=1}.
See Figure \ref{figure::crossings} - (g) for an example.

\textbf{Case 3.3:} $a > 0$, $c > 0$ and $b > 0$. In this case, we find:
\[
\sin{\left( (\theta_3 - (\theta_1 + \theta_2))/2 \right)} > 0, 
\quad
\sin{\left( (\theta_1 - \theta_2 + \theta_3)/2 \right)} < 0,
\quad
\sin{\left( (\theta_2 - \theta_1 + \theta_3)/2 \right)} < 0.
\]
Therefore, the angles $\theta_1$, $\theta_2$ and $\theta_3$ must satisfy
\begin{align}
(\theta_3 - (\theta_1 + \theta_2))/2 &\in (-2\pi, -\pi) \cup (0,\pi), \label{align::1}\\
(\theta_1 - \theta_2 + \theta_3)/2 &\in (-\pi, 0) \cup (\pi, 2\pi), \label{align::2}\\
(\theta_2 - \theta_1 + \theta_3)/2 &\in (\pi, 2\pi). \label{align::3}
\end{align}
Now, observe that $(\ref{align::2}) + (\ref{align::3}) = \theta_3 \in (0, 2\pi)$, which implies that $(\theta_1 - \theta_2 + \theta_3)/2 \in (-\pi, 0)$.
Similarly, since $(\ref{align::3}) - (\ref{align::1}) = \theta_2 \in (0, 2\pi)$, we also get $(\theta_3 - (\theta_1 + \theta_2))/2 \in (0, \pi)$.
However, considering their difference leads to the conclusion that 
$\theta_1 < 0$, which is a contradiction.
We conclude that it is impossible for $a$, $b$ and $c$ to all be positive.

Before considering the last two cases, we first claim that if \(a < 0\) and \(c < 0\), then \(\theta_1 \neq \pi\) and \(\theta_2 \neq \pi\).  
In this case, we observe that 
\(
\sin\left((\theta_3 - (\theta_1+\theta_2))/2\right)
\),
\(
\sin\left((\theta_1 - \theta_2 + \theta_3)/2\right)
\)
and
\(
\sin\left((\theta_2 - \theta_1 + \theta_3)/2\right)
\)  
all share the same sign.  
Therefore, either  
\[
(\theta_3 - (\theta_1 + \theta_2))/2 \in (-2\pi, -\pi) \cup (0, \pi),
\quad
(\theta_1 - \theta_2 + \theta_3)/2 \in (0, \pi), \quad
(\theta_2 - \theta_1 + \theta_3)/2 \in (0, \pi),
\]  
or  
\[
(\theta_3 - (\theta_1 + \theta_2))/2 \in (-\pi, 0), \quad
(\theta_1 - \theta_2 + \theta_3)/2 \in (-\pi, 0) \cup (\pi, 2\pi),
\quad
(\theta_2 - \theta_1 + \theta_3)/2 \in (\pi, 2\pi).
\]  
In both scenarios, it follows that \(\theta_1 \neq \pi\) and \(\theta_2 \neq \pi\).

\textbf{Case 3.4:} $a<0$, $c<0$ and $b>0$.

In this case, the discriminant $\Delta \ge 0$ if and only if
either $\lambda \le \lambda_1$ or $\lambda \ge \lambda_2$,
where $0 < \lambda_1 < \lambda_2$.
Since $ac > 0$, whenever $\Delta \ge 0$, the existance of a positive solution for $\widehat{x}$ requires that $\lambda (\lambda - 2b) < 0$, i.e., $\lambda \in (0,2b)$.
We conclude: when $\lambda \in (0,\lambda_1)$, there are two available solutions for $\widehat{x}$;
when $\lambda < 0$ or $\lambda > \lambda_1$, there is no available solution.

It remains to show that $(X,Y)$ satisfying $\lambda = XY = \lambda_1$ 
lies outside the disk $W \le 0$.
Note that the point $p_\nabla = (X_{\nabla}, Y_\nabla)$ introduced in Lemma \ref{lemma::fixed_point}
satisfies $Q(X_{\nabla}, Y_\nabla) = 0$.
Therefore, we must have $0 \le \lambda_\nabla \le \lambda_1$, where $\lambda_\nabla \coloneqq X_\nabla Y_\nabla$.
However, the point $p_\nabla$ also satisfies $W(X_{\nabla}, Y_\nabla) > 0$, which is the desired result.
See Figure \ref{figure::crossings} - (h) for an example.

\textbf{Case 3.5:} $a<0$, $c<0$ and $b<0$.

In this case, we have $\lambda_1 < \lambda_2 < 0$ and it requires that $\lambda (\lambda - 2b) < 0$.
When $\lambda \in (\lambda_2, 0)$, there are two available solutions, $\widehat{x}_+$ and $\widehat{x}_-$, for $\widehat{x}$.
When $\lambda < \lambda_2$ or $\lambda > 0$, there is no available solution.
As in the previous case,
the point $p_\nabla$ satisfies $Q(X_\nabla, Y_\nabla) = 0$.
Therefore, we must have $\lambda_2 \le \lambda_\nabla \le 0$, but $W(X_\nabla, Y_\nabla) > 0$, as desired.
See Figure \ref{figure::crossings} - (i) for an example.
\end{proof}

\begin{proof}[Proof of Proposition \ref{proposition::Q=0_and_W=0}]
Our proof follows the same case enumeration as in the proof of Theorem \ref{theorem::intersection_of_three}.
When $c_{22}''=0$, or when $c_{22}''\neq 0$ but $c_{20}''c_{02}''=0$, the locus of $Q=0$ in $\Omega$ is a union of two straight lines; these intersect $W=0$ in exactly $4$ points.
When $c_{22}'' \neq 0$, $c_{20}''=0$ and $c_{02}''\neq 0$, it suffices to show that the function
\[
\R_{\ge 0} \ni X \mapsto W\left(X,\dfrac{2bX}{X^2-c}\right)
= c_{22} \dfrac{4b^2 X^4}{(X^2-c)^2} + c_{02} \dfrac{4b^2 X^2}{(X^2-c)^2} + 2c_{11} \dfrac{2bX^2}{X^2-c} + c_{00}
\]
has exactly one zero.
Here $X$ is further restricted to $[0, \sqrt{c})$ if $c > 0$.
We observe that a value $X$ is a zero if and only if $X^2$ is a root of a certain quadratic polynomial in $X^2$.
This polynomial takes a negative value at one endpoint and a positive value at the other.
Consequently, there is exactly one viable zero for $X^2>0$.
The case where $c_{22}''\neq 0$, $c_{20}'' \neq 0$ and $c_{02}''=0$ follows by a similar argument.

From now on, we suppose that $\theta_3$ satisfies $c_{22}'' \neq 0$ and $abc \neq 0$.
By Theorem \ref{theorem::intersection_of_three}, the locus of $\{Q=0\} \cap \{W\le 0\}$ is contained in exactly four specific curves of the form $\gamma_{\epsilon,\tau,\sigma}$.
Moreover, every point $(X,Y)\neq (0,0)$ on $Q=0$ satisfies $XY \neq 0$.
We claim that for a given $\textsc{k} \in (-\infty,+\infty) \setminus {0}$, a point $(X, \textsc{k} X) \neq (0,0)$ on $Q=0$, if it exists, has a uniquely determined value for $X^2$.
Indeed, if $Q(X, \textsc{k} X)=0$, then
\(
\widehat{x} \coloneqq X^2 = - \left(c_{20}'' + c_{02}'' \textsc{k}^2 + 2 c_{11}'' \textsc{k}\right) /  \left(
c_{22}'' \textsc{k}^2 \right).
\)

We define $\textsc{k}_- < \textsc{k}_+$ by
\(
\textsc{k}_{\pm} = 
\left(
-c_{11}'' \pm
\sqrt{c_{11}''^2 - c_{02}'' c_{20}''}
\right) / c_{02}''.
\)
For $\widehat{x} > 0$ to hold:
if $c < 0$, then $\textsc{k} \in (\textsc{k}_-, \textsc{k}_+)$;
if $c > 0$, then $k \in (-\infty, \textsc{k}_-) \cup (\textsc{k}_+,+\infty)$.
We further claim that: if $c<0$, then either $\textsc{k}_- < \textsc{k}_+ < 1$ or $1 < \textsc{k}_- < \textsc{k}_+$;
if $c>0$, then $\textsc{k}_- < 1 < \textsc{k}_+$.
This claim follows from the identity:
\[
\left(\textsc{k}_- - 1\right)
\left(\textsc{k}_+ - 1\right)
=
\dfrac{\sin^2\left(\dfrac{\theta_1-\theta_2}{2}\right)}{\sin^2\left(\dfrac{\theta_1}{2}\right)}
\cdot
\dfrac{\sin\left(\dfrac{\theta_3-\theta_1-\theta_2}{2}\right)}{\sin\left(\dfrac{\theta_2-\theta_1+\theta_3}{2}\right)}
= \dfrac{c_{22}''}{c_{02}''} = -c.
\]
For a point $(X,Y)$ with $Y = \textsc{k} X$, we have
\(
W(X,Y) = c_{22} X^4 \textsc{k}^2
+ c_{20} X^2
+ c_{02} X^2 \textsc{k}^2
+ 2c_{11} X^2 \textsc{k}
+ c_{00}
\).
We define the function $\heartsuit: \textsc{k} \mapsto \textsc{k}^2 \cdot W(X,Y)$,
where $Y=\textsc{k} X$ and $(X,Y)$ satisfies $Q(X,Y) = 0$.
More precisely, we obtain
\[
\heartsuit(\textsc{k}) =
\dfrac{c_{22}}{c_{22}''^2}
\left( c_{20}'' + c_{02}'' \textsc{k}^2 + 2 c_{11}'' \textsc{k} \right)^2
-
\dfrac{1}{c_{22}''}
\left( c_{20} + c_{02} \textsc{k}^2 + 2c_{11} \textsc{k}
\right) \cdot
\left( c_{20}'' + c_{02}'' \textsc{k}^2 + 2 c_{11}'' \textsc{k} \right)
+ c_{00} \textsc{k}^2.
\]
Thus $\heartsuit$ is a quartic polynomial.
Furthermore, along each curve $\gamma_{\epsilon,\tau,\sigma}$ describing $\{Q=0\} \cap \Omega$, the function $\heartsuit$ takes a positive value at the head and a negative value at the tail.
We also claim that $\heartsuit(1) = 0$, 
which follows from Lemma \ref{lemma::a+2b+c=-1} and Lemma \ref{lemma::W_with_XY=1}.

\textbf{Case 1:} $ac < 0$.
In this case, we have $\textsc{k}_- < 0 < \textsc{k}_+$.
If $c<0$, then $\textsc{k}_- < 0 < \textsc{k}_+ < 1$.
If $c>0$, then $\textsc{k}_- < 0 < 1 < \textsc{k}_+$.
Given that $\heartsuit(1)=0$, the intersection of $Q=0$ and $W=0$ always consists of four points.

\textbf{Case 2:} $a>0$, $b<0$ and $c>0$.
In this case, we further have $c_{22}'' < 0$, $c_{20}'' > 0$ and $c_{02}'' > 0$.
Consequently, we obtain $0 < \textsc{k}_- < \textsc{k}_+$.
Since $c>0$, it follows that $0 < \textsc{k}_- < 1 < \textsc{k}_+$.
The desired property of the intersection follows from $\heartsuit(1) = 0$.

\textbf{Case 3:} $a<0$ and $c<0$.
In this case, we have either $0 < \textsc{k}_- < \textsc{k}_+$ or $\textsc{k}_- < \textsc{k}_+ < 0$.
Given that $c<0$, this further implies that either $0 < 1 < \textsc{k}_- < \textsc{k}_+$ or $\textsc{k}_- < \textsc{k}_+ < 1$.
Again, the desired result is implied by $\heartsuit(1) = 0$.
\end{proof}

%
\subsection{Foliation of the disk}
\label{subsection::foliation}

Let \(0 < \theta_1 < \theta_2 < 2\pi\).
If \(\theta_3\) satisfies \(c_{22}'' \neq 0\) and \(abc \neq 0\), then
\[
\sin\left(\dfrac{\theta_3 - (\theta_1 + \theta_2)}{2}\right) \neq 0, \quad
\sin\left(\dfrac{\theta_3}{2}\right) \neq 0, \quad
\sin\left(\dfrac{\theta_1 - \theta_2 + \theta_3}{2}\right) \neq 0, \quad
\sin\left(\dfrac{\theta_2 - \theta_1 + \theta_3}{2}\right) \neq 0.
\]  
Consequently, the singular values of \(\theta_3\) for which \(c_{22}'' \cdot c_{20}'' \cdot c_{02}'' \cdot c_{11}'' = 0\) form the set
\[
S \coloneqq \left\{
0,\ 2\pi,\ \theta_2 - \theta_1,\ 
2\pi - (\theta_2 - \theta_1),\ \theta_1 + \theta_2,\ 
(\theta_1 + \theta_2) - 2\pi
\right\} \cap [0, 2 \pi]
\subset [0, 2\pi].
\]

Fix \(\epsilon, \tau, \sigma \in \{\pm 1\}\).  
For every \(\theta_3 \in [0, 2\pi] \setminus S\), consider the curve \(\gamma_{\epsilon,\tau,\sigma} = \gamma_{\epsilon,\tau,\sigma}(\theta_3)\).  
Let \(U \coloneqq (l, r)\) be a connected component of \([0, 2\pi] \setminus S\) such that, for some \(\theta_3 \in U\), the interval \(J^x_{\epsilon,\tau,\sigma} = J^x_{\epsilon,\tau,\sigma}(\theta_3)\) is not a singleton.
Therefore, by Theorem \ref{theorem::intersection_of_three}, for every \(\theta_3 \in U\), the curve \(\gamma_{\epsilon,\tau,\sigma}(\theta_3)\) joins the origin to an exterior point of the disk \(W \le 0\).  
The curves \(\gamma_{\epsilon,\tau,\sigma}\) deform continuously with respect to \(\theta_3 \in U\).
For simplicity, we define $r_W(\theta_3)$ by
\[
\gamma_{\epsilon,\tau,\sigma}(\theta_3)\left([0, r_W(\theta_3)]\right)
=
\gamma_{\epsilon,\tau,\sigma}(\theta_3)\left(J^x_{\epsilon,\tau,\sigma}(\theta_3)\right) \cap \{W \le 0\}.
\]

\begin{lemma}
\label{lemma::derivative_of_hatx}
Fix $\chi > 0$.
For every $\theta_3 \in U$ such that
$\chi \in (0, r_W(\theta_3)]$,
we define the point
\[
(X, Y) = (X(\theta_3), Y(\theta_3)) \coloneqq \gamma_{\epsilon,\tau,\sigma}(\theta_3)(\chi)\]
and
define $\widehat{x} = \widehat{x}(\theta_3) \coloneqq X^2$.
Then, we get the following:
\begin{enumerate}[label=(\roman*)]
\item $\dfrac{\partial \widehat{x}}{\partial \theta_3} \cdot \left( aX^2 - cY^2 \right)
=
\dfrac{-X^2 
\cdot
(X - Y)
\cdot
\left(
\cos{(\theta_1/2)} 
\sin{(\theta_2/2)} X - 
\sin{(\theta_1/2)} 
\cos{(\theta_2/2)} Y
\right)
}{\sin^2{\left(\dfrac{\theta_1 - \theta_2}{2}\right)} 
\sin^2{\left(\dfrac{\theta_1 + \theta_2 - \theta_3}{2}\right)}}$;

\item $aX^2 \neq cY^2$, $X \neq Y$ and $\cos{(\theta_1/2)} 
\sin{(\theta_2/2)} X \neq 
\sin{(\theta_1/2)} 
\cos{(\theta_2/2)} Y$.
\end{enumerate}
\end{lemma}

\begin{proof}
Notice that $\sigma \chi = XY$ and
\begin{align}
a \cdot \widehat{x}^2 + (2b\sigma \chi - \chi^2) \cdot \widehat{x} + c \chi^2 = 0, 
\label{equation::quadratic_hatx} \\
X^2Y^2 = a\cdot X^2 + 2b\cdot XY + c\cdot Y^2.
\label{equation::quartic}
\end{align}
Therefore, the derivative of the equation (\ref{equation::quadratic_hatx}) implies that
\begin{equation*}
\dfrac{\partial \widehat{x}}{\partial \theta_3} \cdot
\left(
aX^2 - cY^2
\right)
=
- \left(
\dfrac{\partial a}{\partial \theta_3} \cdot X^2
+
2 \cdot
\dfrac{\partial b}{\partial \theta_3} \cdot XY
+
\dfrac{\partial c}{\partial \theta_3} \cdot Y^2
\right).
\end{equation*}
Since
\begin{align*}
\partial a / \partial \theta_3 \cdot \mathcal{D}
=&
\sin{\left(\theta_1\right)}
\cdot \sin^2{(\theta_2 / 2)}, \\
\partial b / \partial \theta_3 \cdot \mathcal{D}
=&
- \sin{(\theta_1/2)}
\cdot \sin{(\theta_2/2)}
\cdot \sin{((\theta_1+\theta_2)/2)},
\\
\partial c / \partial \theta_3 \cdot \mathcal{D}
=&
\sin{\left(\theta_2\right)}
\cdot \sin^2{(\theta_1 / 2)},
\end{align*}
where $\mathcal{D} = 2 \cdot \sin^2{\left(\dfrac{\theta_1-\theta_2}{2}\right)}
\cdot
\sin^2{\left(\dfrac{\theta_1+\theta_2-\theta_3}{2}\right)}$,
we get the desired expression in assertion ($\RNum{1}$).

Assume that $aX^2 = cY^2$.
Then $X^2 / Y^2 = c / a$, hence $\sgn(a) = \sgn(c)$ and
$x / y = \sigma \sqrt{c / a}$.
From equation (\ref{equation::quadratic_hatx}), we also obtain
\(
0 < X^2 = a \cdot \dfrac{X^2}{Y^2} + 2b \cdot \dfrac{X}{Y} + c = \dfrac{2}{a} \sqrt{ac} \cdot \left( \sqrt{ac} + \sigma b \right).
\)
Therefore, we must have $\sgn(a)=\sgn(c)=\sgn(\sigma b)$,
since $b^2 > ac$.
However, this contradicts Theorem \ref{theorem::intersection_of_three}.

Assume that $X=Y$, i.e., $X/Y = 1$.
Then we get $\chi^2 = X^2 = Y^2 = a+2b+c$.
However, by Lemma \ref{lemma::a+2b+c=-1}, we have $\chi = |a+2b+c| = 1$ and, by Lemma \ref{lemma::W_with_XY=1}, this leads to a contradiction.

Assume that
$\cos{\dfrac{\theta_1}{2}} 
\sin{\dfrac{\theta_2}{2}} X = 
\sin{\dfrac{\theta_1}{2}} 
\cos{\dfrac{\theta_2}{2}} Y$.
Since $Y = \sigma \chi / X$
and $\dfrac{\partial \hat{x}}{\partial \theta_3} = 2X \cdot \dfrac{\partial X}{\partial \theta_3}$, we obtain
\begin{equation*}
\dfrac{\partial}{\partial \theta_3}\left(
\cos{\dfrac{\theta_1}{2}} 
\sin{\dfrac{\theta_2}{2}} X - 
\sin{\dfrac{\theta_1}{2}} 
\cos{\dfrac{\theta_2}{2}} Y
\right)
=
\dfrac{1}{2X^2}
\left(
\cos{\dfrac{\theta_1}{2}} 
\sin{\dfrac{\theta_2}{2}} X + 
\sin{\dfrac{\theta_1}{2}} 
\cos{\dfrac{\theta_2}{2}} Y
\right)
\cdot \dfrac{\partial \widehat{x}}{\partial \theta_3}.
\end{equation*}
Now, suppose that there exists a small neighbourhood $U_0$ of $\theta_3$ such that, for every $\theta'_3 \in U_0$, we have
$\cos{\dfrac{\theta_1}{2}} 
\sin{\dfrac{\theta_2}{2}} X(\theta'_3) + 
\sin{\dfrac{\theta_1}{2}} 
\cos{\dfrac{\theta_2}{2}} Y(\theta'_3) \neq 0$.
Then,
\begin{equation*}
\sgn
\left.\left(
\cos{\dfrac{\theta_1}{2}} 
\sin{\dfrac{\theta_2}{2}} X - 
\sin{\dfrac{\theta_1}{2}} 
\cos{\dfrac{\theta_2}{2}} Y
\right)
\right|_{\theta'_3}
= \mu \cdot
\sgn\left(
\dfrac{\partial}{\partial \theta_3}
\left.
\left(
\cos{\dfrac{\theta_1}{2}} 
\sin{\dfrac{\theta_2}{2}} X - 
\sin{\dfrac{\theta_1}{2}} 
\cos{\dfrac{\theta_2}{2}} Y
\right)
\right|_{\theta'_3}
\right)
\end{equation*}
for some $\mu \in \{\pm 1\}$.
However, a function $f$ satisfying $\sgn(f(x)) = \mu\cdot \sgn(f'(x))$ at all points where it is defined cannot vanish at any point, which leads to a contradiction.
Otherwise, we suppose that
$\cos{\dfrac{\theta_1}{2}} 
\sin{\dfrac{\theta_2}{2}} X + 
\sin{\dfrac{\theta_1}{2}} 
\cos{\dfrac{\theta_2}{2}} Y = 0$.
Combining this with the initial assumption, it follows that $X=Y=0$.
\end{proof}

As a consequence, we get the following homeomorphism.

\begin{proposition}
\label{proposition::circular-sector}
The map
\(
\Phi:
\left\{
r e^{i \theta} \in \C
~\middle|~
\begin{array}{c}
0 < \theta < \pi / 2  \\
0 \le r \le 1
\end{array}
\right\}
\rightarrow
\bigcup\limits_{\theta_3 \in U}
\gamma_{\epsilon,\tau,\sigma}(\theta_3)([0,r_W(\theta_3)])
\eqqcolon \mathcal{P}_{\epsilon,\tau,\sigma,U}
\)
defined by
\(
r e^{i \theta}
\mapsto
\gamma_{\epsilon,\tau,\sigma}
\left(l+\theta\dfrac{r-l}{\pi/2}\right)
\left(
r \cdot r_W
\left(l+\theta\dfrac{r-l}{\pi/2}
\right)
\right)
\)
is a homeomorphism.
\end{proposition}

\begin{proof}
By Lemma \ref{lemma::derivative_of_hatx}, the derivative $\partial \widehat{x} / \partial \theta_3$ does not change sign for $\theta_3 \in U$ and $0 < \lambda < r_W(\theta_3)$.
Therefore, Proposition \ref{proposition::Q=0_and_W=0} implies that
the circular sector $\mathcal{P}_{\epsilon,\tau,\sigma,U}$, with the origin and the boundary arc removed,
is foliated by the curves $\gamma_{\epsilon,\tau,\sigma}(\theta_3)(0,r_W(\theta_3))$. 
Furthermore, by the same proposition, this foliation extends to the boundary arc. 
We conclude that the map $\Phi$ is a homeomorphism.
\end{proof}

It is unclear whether the homeomorphism introduced in Proposition \ref{proposition::circular-sector} is orientation preserving. However, by the definition of $\gamma_{\epsilon,\tau,\sigma}$, all curves \(\gamma_{\epsilon,\tau,\sigma}(\theta_3)\) with \(\theta_3 \in U\) lie in a common quadrant. Consequently, the orientation of the circular sector \(\mathcal{P}_{\epsilon,\tau,\sigma,U}\) is determined by the relative position of  
\(
\lim\limits_{\theta_3 \to l_+} \gamma_{\epsilon,\tau,\sigma}(\theta_3)
\)
and
\(\lim\limits_{\theta_3 \to r_-} \gamma_{\epsilon,\tau,\sigma}(\theta_3)
\).

Note that since $Q_{\theta_1,\theta_2,0} = -Q_{\theta_1,\theta_2,2\pi}$, the equation $Q_{\theta_1,\theta_2,0} = 0$ defines the same locus as $Q_{\theta_1,\theta_2,2\pi} = 0$.
We therefore arrive at the following foliation of the disk $W \le 0$.

\begin{theorem}
\label{theorem::foliation}
For $0 < \theta_1 < \theta_2 < 2\pi$, the punctured disk $\{W(X,Y) \leq 0\} \setminus \{(0,0)\}$ is foliated
by the loci of ${Q_{\theta_1,\theta_2,\theta_3}(X,Y) = 0}$ as $\theta_3$ varies over $[0, 2\pi)$.
\end{theorem}

Figure \ref{figure::foliations} displays a foliation of the disk ${W = 0}$ by the loci of $Q_{\theta_1,\theta_2,\theta_3}(X,Y) = 0$, parameterised by $(\psi_1, \psi_2)$.
Here, the angle $\theta_3$ takes values at regular intervals in $[0, 2\pi)$.
The topology of the resulting foliation varies with the choice of the pair $(\theta_1, \theta_2)$.

\begin{figure}[htb]
 \centering
 \begin{subfigure}[b]{0.325\textwidth}
  \centering
  \includegraphics[width=\textwidth]{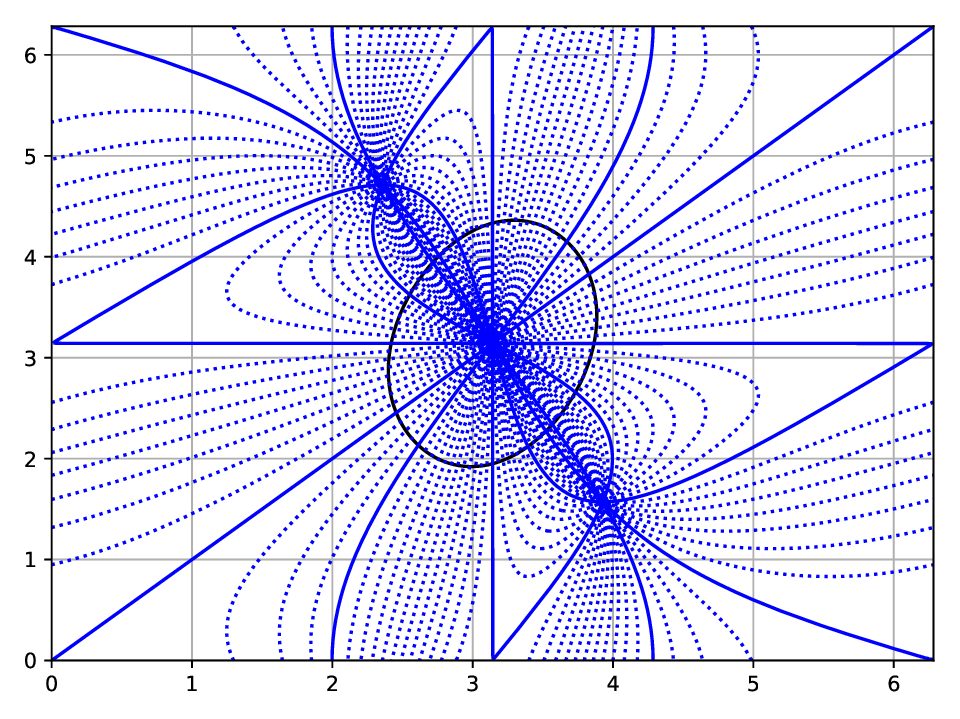}
  \caption{$(\theta_1,\theta_2) = \left(\dfrac{\pi}{4}, \dfrac{3\pi}{2}\right)$}
 \end{subfigure}
 \hfill
 \begin{subfigure}[b]{0.325\textwidth}
  \centering
  \includegraphics[width=\textwidth]{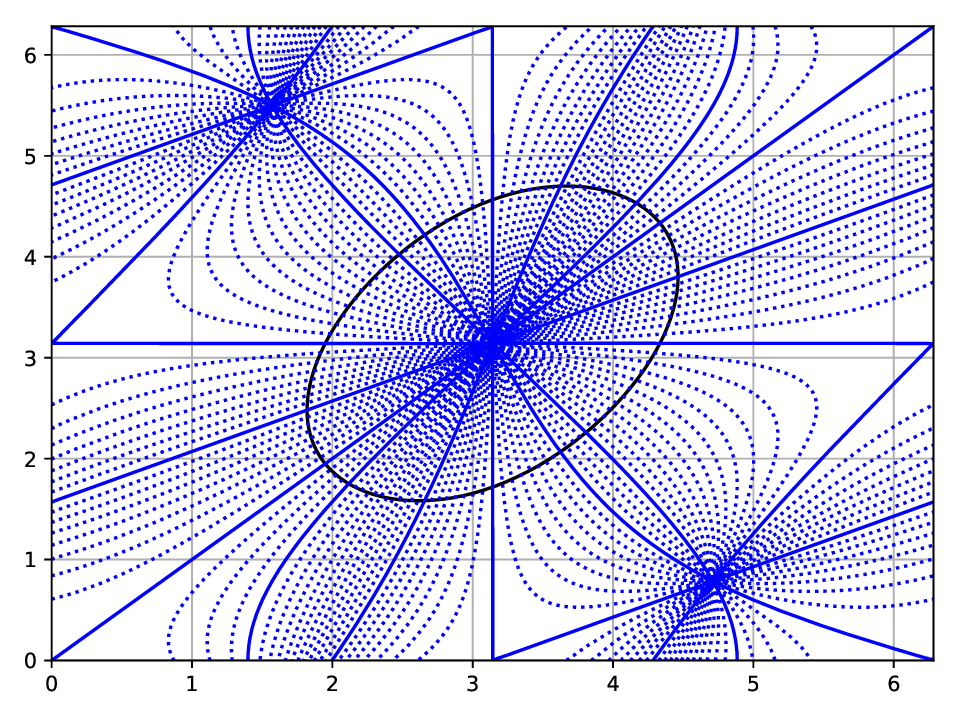}
  \caption{$(\theta_1,\theta_2) = \left(\dfrac{\pi}{2}, \dfrac{5\pi}{4}\right)$}
 \end{subfigure}
 \hfill
 \begin{subfigure}[b]{0.325\textwidth}
  \centering
  \includegraphics[width=\textwidth]{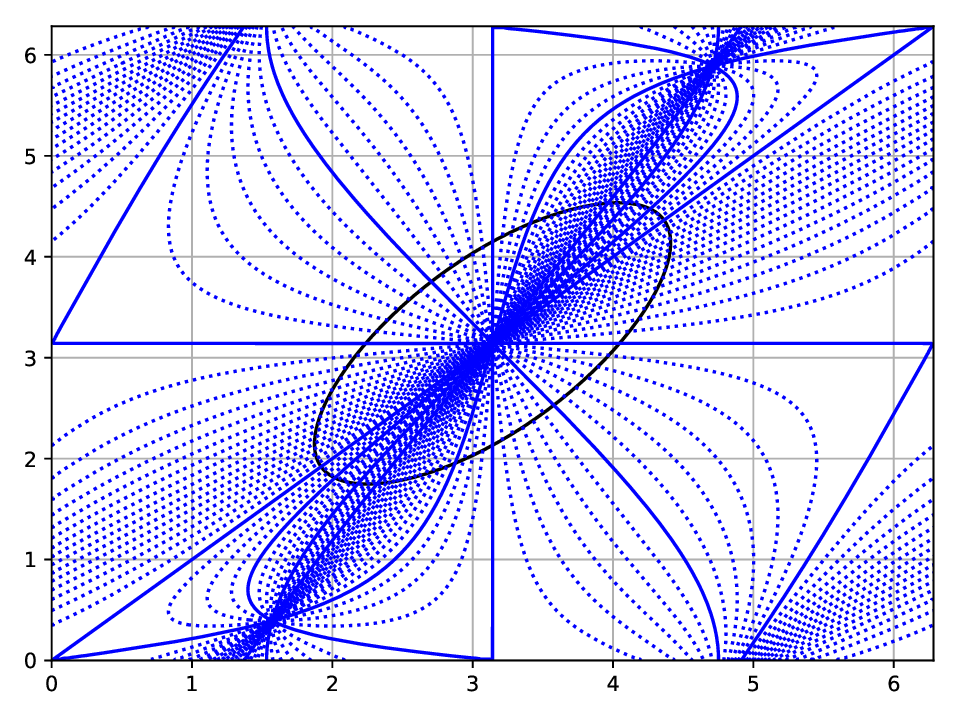}
  \caption{$(\theta_1,\theta_2) = \left(\dfrac{\pi}{2}, \dfrac{7\pi}{8}\right)$}
 \end{subfigure}
 \caption{We present three distinct choices for the pair $(\theta_1, \theta_2)$. For each, the intersection $I(\theta_1) \cap I(\theta_2)$ forms a disk, bounded by the displayed circle and visualized in coordinates parameterised by $(\psi_1, \psi_2) \in (0, 2\pi)^2$ (with $\psi_1$ horizontal and $\psi_2$ vertical). For each pair, we provode the loci of $Q_{\theta_1,\theta_2,\theta_3}(X,Y)=0$ for $\theta_3 \in \{k\pi/16 \mid k=0, 1, ..., 31\}$. Singular loci are plotted as solid curves, while generic loci are dotted. Integer grids are also provided in each figure.}
 \label{figure::foliations}
\end{figure}

Theorem \ref{theorem::foliation} allows us to deduce the following corollaries.

\begin{corollary}
\label{corollary::foliation_of_intersections}
For $0 < \theta_1 < \theta_2 < 2\pi$, the punctured disk $I(\theta_1) \cap I(\theta_2) \setminus \{[0,0,1]^t\}$, after the removal of $12$ disjoint curves connecting the origin to the boundary, is foliated by the intersections $I(\theta_1) \cap I(\theta_2) \cap I(\theta_3)$ as $\theta_3$ ranges over $(0,2\pi) \setminus \{\theta_1, \theta_2\}$.
\end{corollary}

\begin{proof}
By Corollary
\ref{corollary::intersection_of_three}, the 12 disjoint curves that were removed are the crossings defined by $Q_{\theta_1,\theta_2,\theta_3}(X,Y) = 0$ for $\theta_3 \in \{0, \theta_1, \theta_2\}$. This corollary follows immediately from Theorem \ref{theorem::foliation}.
\end{proof}

\begin{corollary}
\label{corollary::intersection_of_four}
Let $\theta_1, \theta_2, \theta_3, \theta_4$ be distinct angles in $(0, 2\pi)$.
Then the intersection of the four isometric spheres $I(\theta_1) \cap I(\theta_2) \cap I(\theta_3) \cap I(\theta_4)$ is the single point $[0,0,1]^t$.
\end{corollary}

\begin{proof}
The foliation of $W \le 0$ implies that the leaves do not intersect with each other.
Thus, the intersection of two crossings $I(\theta_1) \cap I(\theta_2) \cap I(\theta_3)$ and $I(\theta_1) \cap I(\theta_2) \cap I(\theta_4)$ is a singleton.
\end{proof}

\begin{proof}[Proof of Theorem \ref{theorem::foliation}]
The coefficients of the function \( Q_{\theta_1,\theta_2,\theta_3} \) satisfy the following properties: when \( \theta_1 + \theta_2 < 2\pi \), \( c_{22}'' > 0 \) if and only if \( \theta_3 > \theta_1 + \theta_2 \) and when \( \theta_1 + \theta_2 > 2\pi \), \( c_{22}'' > 0 \) if and only if \( \theta_3 < (\theta_1 + \theta_2) - 2\pi \); additionally \( c_{20}'' > 0 \) if and only if \( \theta_3 > \theta_2 - \theta_1 \) and \( c_{02}'' > 0 \) if and only if \( \theta_3 < 2\pi - (\theta_2 - \theta_1) \) with \( c_{11}'' \) always being negative.

The set of singular angles $S \subset [0,2\pi]$ has at most five elements, including $0$ and $2\pi$. For the case $|S| = 5$, the complement consists of four intervals $U_1, U_2, U_3, U_4$, ordered left to right. For each $U_i = (l, r)$, we consider an angle $\theta_3 \in U_i$.

\textbf{Case 1:} $\theta_1 + \theta_2 < 2\pi$. Consequently, we have $\theta_2 - \theta_1 < \theta_1 + \theta_2$ and the set of singular angles is
\[
S = \{0, \theta_2 - \theta_1, \theta_1 + \theta_2, 2\pi, 2\pi - (\theta_2 - \theta_1)\} \subset [0,2\pi].
\]

\textbf{Case 1.1:} $0 < 2\pi - (\theta_2 - \theta_1) < \theta_2 - \theta_1 < \theta_1 + \theta_2 < 2\pi$.

The signs of the coefficients are given in the following table.

\begin{center}
\begin{tabular}{cc|ccc|ccc|l}
\hline
$l$ & $r$ & $c_{22}''$ & $c_{20}''$ & $c_{02}''$ & $a$ & $b$ & $c$ & circular sectors \\
\hline
$0$ & $2\pi-(\theta_2-\theta_1)$ & $<0$ & $<0$ & $>0$ & $<0$ & $<0$ & $>0$ & $\mathcal{P}_{\epsilon,-1,\sigma,U_1}$ \\
$2\pi-(\theta_2-\theta_1)$ & $\theta_2-\theta_1$ & $<0$ & $<0$ & $<0$ & $<0$ & $<0$ & $<0$ & $\mathcal{P}_{\epsilon,\tau,-1,U_2}$ \\
$\theta_2-\theta_1$ & $\theta_1+\theta_2$ & $<0$ & $>0$ & $<0$ & $>0$ & $<0$ & $<0$ & $\mathcal{P}_{\epsilon,+1,\sigma,U_3}$ \\
$\theta_1+\theta_2$ & $2\pi$ & $>0$ & $>0$ & $<0$ & $<0$ & $>0$ & $>0$ & $\mathcal{P}_{\epsilon,-1,\sigma,U_4}$ \\
\hline
\end{tabular}
\end{center}

We claim that the disk $\{W \le 0\}$ consists of the closures of $16$ circular sectors. More precisely, the first quadrant comprises $\mathcal{P}_{+1,+1,+1,U_3}$,
$\mathcal{P}_{+1,-1,+1,U_4}$
and 
$\mathcal{P}_{+1,-1,+1,U_1}$ 
arranged counterclockwise, 
while the second quadrant consists of 
$\mathcal{P}_{-1,+1,-1,U_2}$,
$\mathcal{P}_{-1,+1,-1,U_3}$,
$\mathcal{P}_{-1,-1,-1,U_4}$,
$\mathcal{P}_{-1,-1,-1,U_1}$,
$\mathcal{P}_{-1,-1,-1,U_2}$ 
also in counterclockwise order. 
The lower half-plane contains sectors with analogous subscripts obtained by reversing the sign of $\epsilon$.
See Figure \ref{figure::foliations} - (a) for an example.
To prove this claim, let $\chi > 0$ be sufficiently small
and
revisit that
\begin{align*}
\widehat{x}_\tau(\sigma \chi)
&= \dfrac{\sigma\chi}{2a}
\left(
\left( \sigma\chi - 2b \right)
+ \tau \sigma
\sqrt{\left( \sigma\chi - 2b \right)^2 - 4ac}
\right) \\
&= \dfrac{\sigma\chi}{2 c_{20}''}
\left(
\left(-\sigma\chi \cdot c_{22}'' - 2 c_{22}'' \right)
+ \tau \sigma \cdot \sgn(c_{22}'') \cdot
\sqrt{
\left(-\sigma\chi \cdot c_{22}'' - 2 c_{11}''\right)^2
- 4 c_{20}'' c_{02}''
}
\right)
\end{align*}
for $\theta_3$ satisfying $c_{22}'' \neq 0$ and $abc \neq 0$.
For any two potentially adjacent circular sectors, i.e., 
$\mathcal{P}_{\epsilon,\tau,\sigma,[l,r]}$ and $\mathcal{P}_{\epsilon',\tau',\sigma',[l',r']}$ where $r=l'$,
one can then verify that the following limit holds:
\[
\lim_{\theta_3 \rightarrow r_{-}} \widehat{x}_{\tau}(\sigma \chi) = \lim_{\theta_3 \rightarrow l'_{+}} \widehat{x}_{\tau'}(\sigma' \chi).
\]

Since the curves $\gamma_{\epsilon,\tau,\sigma}(\theta_3)$ deform continuously for $\theta_3 \in [0,2\pi] \setminus S$
and the loci of $Q_{\theta_1,\theta_2,\theta_3}=0$ also vary continuously with $\theta_3 \in [0,2\pi]$,
we obtain a foliation of the open disk ${W < 0}$ that extends continuously to a foliation of the closed disk ${W \le 0}$.

\textbf{Case 1.2:} $0 < \theta_2 - \theta_1 < 2 \pi - (\theta_2 - \theta_1) < \theta_1 + \theta_2 < 2 \pi$.

The signs of the coefficients are given in the following table.

\begin{center}
\begin{tabular}{cc|ccc|ccc|l}
\hline
$l$ & $r$ & $c_{22}''$ & $c_{20}''$ & $c_{02}''$ & $a$ & $b$ & $c$ & circular sectors \\
\hline
$0$ & $\theta_2-\theta_1$ & $<0$ & $<0$ & $>0$ & $<0$ & $<0$ & $>0$ & $\mathcal{P}_{\epsilon,-1,\sigma,U_1}$ \\
$\theta_2-\theta_1$ & $2\pi-(\theta_2-\theta_1)$ & $<0$ & $>0$ & $>0$ & $>0$ & $<0$ & $>0$ & $\mathcal{P}_{\epsilon,\tau,+1,U_2}$ \\
$2\pi-(\theta_2-\theta_1)$ & $\theta_1+\theta_2$ & $<0$ & $>0$ & $<0$ & $>0$ & $<0$ & $<0$ & $\mathcal{P}_{\epsilon,+1,\sigma,U_3}$ \\
$\theta_1+\theta_2$ & $2\pi$ & $>0$ & $>0$ & $<0$ & $<0$ & $>0$ & $>0$ & $\mathcal{P}_{\epsilon,-1,\sigma,U_4}$ \\
\hline
\end{tabular}
\end{center}

We claim that the disk ${W \leq 0}$ is comprised of the closures of $16$ circular sectors; specifically, $3$ sectors lie in the first quadrant and $5$ in the second.
See Figure \ref{figure::foliations} - (b) for an example.
Indeed, an argument analogous to that in Case 1.1 shows that for a sufficiently small $\chi > 0$, one can glue the slices across all sectors to induce the desired foliation of the disk.

\textbf{Case 1.3:} $0 < \theta_2-\theta_1 < \theta_1 + \theta_2 < 2\pi - (\theta_2 - \theta_1) < 2\pi$.

The signs of the coefficients are given in the following table.

\begin{center}
\begin{tabular}{cc|ccc|ccc|l}
\hline
$l$ & $r$ & $c_{22}''$ & $c_{20}''$ & $c_{02}''$ & $a$ & $b$ & $c$ & circular sectors \\
\hline
$0$ & $\theta_2-\theta_1$ & $<0$ & $<0$ & $>0$ & $<0$ & $<0$ & $>0$ & $\mathcal{P}_{\epsilon,-1,\sigma,U_1}$ \\
$\theta_2-\theta_1$ & $\theta_1+\theta_2$ & $<0$ & $>0$ & $>0$ & $>0$ & $<0$ & $>0$ & $\mathcal{P}_{\epsilon,\tau,+1,U_2}$ \\
$\theta_1+\theta_2$ & $2\pi-(\theta_2-\theta_1)$ & $>0$ & $>0$ & $>0$ & $<0$ & $>0$ & $<0$ & $\mathcal{P}_{\epsilon,\tau,+1,U_3}$ \\
$2\pi-(\theta_2-\theta_1)$ & $2\pi$ & $>0$ & $>0$ & $<0$ & $<0$ & $>0$ & $>0$ & $\mathcal{P}_{\epsilon,-1,\sigma,U_4}$ \\
\hline
\end{tabular}
\end{center}

We claim that the disk ${W \leq 0}$ is comprised of the closures of $16$ circular sectors; specifically, $6$ sectors lie in the first quadrant and $2$ in the second.
See Figure \ref{figure::foliations} - (c) for an example.
The proof is similar to the argument in Case 1.1.

\textbf{Case 2:} $\theta_1+\theta_2 > 2\pi$. Consequently, we have $\theta_1 + \theta_2 - 2\pi < 2\pi - (\theta_2 - \theta_1)$ and the set of singular angles is
\(
S = \left\{0,\ \theta_1 + \theta_2 - 2\pi,\ 2\pi - (\theta_2 - \theta_1),\ 2\pi,\ \theta_2 - \theta_1\right\} \subset [0,2\pi]
\).
By enumerating the position of $\theta_2 - \theta_1$, we obtain three analogous cases where in each case the disk can be foliated in a similar manner.

\textbf{Degenerated case:}
A one-codimensional locus of $(\theta_1,\theta_2)$ results in the set of singular angles $S$ having fewer than five elements. Since the disk ${W \le 0}$ and the loci ${Q_{\theta_1,\theta_2,\theta_3} = 0}$ for $\theta_3 \in [0,2\pi]$ deform uniformly continuously, every degenerate configuration also yields a foliation of the disk, as desired.
\end{proof}

\section{Ford domain}
\label{section::ford_domain}

In the Siegel domain model of the complex hyperbolic space $\hc$,
consider a discrete subgroup $\Gamma \le \PU(H_2)$.
The \emph{Ford domain} centred at $q_{\infty}$ is defined to be the closed domain
\[
D_{\Gamma} \coloneqq \bigcap_{g \in \Gamma\setminus \Gamma_\infty} \overline{I_+(g)}
\setminus 
\Proj(V_{H_2}^0)
\subset
\Proj(V_{H_2}^-),
\]
where $\overline{I_+(g)} = I(g) \sqcup I_+(g)$
and
$\Gamma_\infty \le \Gamma$ is the stabilizer subgroup of $\Gamma$ with respect to $q_{\infty}$.
The Ford domain is a polytope of dimension $4$ in $\hc$, which is a special cellular complex,
whose cells of dimension $0$, $1$, $2$ and $3$ are called \emph{vertices}, \emph{edges}, \emph{ridges} and \emph{sides}, respectively.
We let $\overline{D_{\Gamma}}$ be the closure of $D_\Gamma$ in $\overline{\hc}$.
Define the \emph{ideal boundary} $\partial_\infty D_{\Gamma}$ to be the intersection of $\overline{D_{\Gamma}}$ with $\partial \hc$,
which is a $3$-dimensional submanifold of $\partial \hc$ with boundary.

Using the standard model of the intersection $\{I(\theta)\}$ introduced in Subsection~\ref{subsection::standard}, we obtain the following standard model for the Ford domain of a finite cyclic group of real elliptic isometries.

\begin{definition}
For $n \ge 3$, define the polytope
\(
D_n \coloneqq
\bigcap_{k=1}^{n-1}
\overline{I_+\left( 2\pi k/n \right)}
\setminus
\Proj(V_{H_1}^0)
\subset
\Proj(V_{H_1}^-)
\).
\end{definition}

\begin{proposition}
\label{proposition::standard_model_Dn}
Let $f$ be a regular, real elliptic isometry of order $n \ge 3$ such that $q_\infty \in \Torus_f$. Then there exists an isometry sending each $I(f^k)$ to $I(2\pi k/n)$, for $k = 1,\dots,n-1$, which induces an isometry between the Ford domain $D_{\langle f\rangle}$ and the polytope $D_n$.
\end{proposition}

\begin{proof}
This derives from Proposition \ref{proposition::spheres_EEE}.
\end{proof}

We now prove Theorem \ref{thmx::ford} by describing the cellular structure of $D_n$ below.

\begin{theorem}
\label{theorem::ford_domain}
For $n \ge 3$, we have:

\begin{enumerate}[label=(\roman*)]

\item The boundary of the ideal boundary $\partial \partial_\infty D_n$ is a $2$-dimensional cellular complex homeomorphic to the $2$-sphere.
The cellular structure is obtained by gluing the following polgons along their edges:

\begin{enumerate}
\item The polygon $f_1$ is a $2(n-2)$-gon.
Its boundary, listed in counterclockwise order, is given by the sequence of edges:
\(
e_{1,n-1},
e'_{1,2}, 
\ldots,
e'_{1,n-2},
e'_{1,n-1},
e_{1,2},
\ldots,
e_{1,n-2}
\).

\item The polygon $f_{n-1}$ is a $2(n-2)$-gon.
Its boundary, listed in counterclockwise order, is given by the sequence of edges:
\(
e'_{1,n-1},
e'_{n-1,n-2}, 
\ldots,
e'_{n-1,2},
e_{1,n-1},
e_{n-1,n-2},
\ldots,
e_{n-1,2}
\).

\item When $n = 4$, the polygons $f_2$ and $f'_2$ are bi-gons, whose boundaries are given by 
$e_{1,2}$, $e_{n-1,2}$ and $e'_{1,2}$, $e'_{n-1,2}$, respectively.
When $n \ge 5$, the polygons $f_2$ and $f'_2$ are triangles, whose boundaries, listed in counterclockwise order, are given by $e_{1,2}$, $e_{n-1,2}$, $e_{2,3}$ and
$e'_{1,2}$, $e'_{n-1,2}$, $e'_{2,3}$, respectively;
the triangles $f_{n-2}$ and $f'_{n-2}$ have the boundaries, listed in counterclockwise order, given by $e_{n-1,n-2}$, $e_{1,n-2}$, $e_{n-3,n-2}$ and
$e'_{n-1,n-2}$, $e'_{1,n-2}$, $e'_{n-3,n-2}$, respectively.

\item When $n \ge 6$, for $3 \le k \le n-3$, the quadrangle $f_k$ has the boundary, listed in counterclockwise order, given by $e_{1,k}$, $e_{k-1,k}$, $e_{n-1,k}$ and $e_{k,k+1}$;
the quadrangle $f'_k$ has the boundary, listed in counterclockwise order, given by $e'_{1,k}$, $e'_{k-1,k}$, $e'_{n-1,k}$ and $e'_{k,k+1}$.
\end{enumerate}
Here, the subscript $1\le k\le n-1$ of a polygon indicates that the polygon belongs to the the ideal $2$-sphere $\partial_\infty I(2\pi k/n) \coloneqq I(2\pi k/n) \cap \partial \hc$.
The subscripts $j, k$ of an edge indicate that the edge belongs to the ideal circle $I(2\pi j/n) \cap I(2\pi k/n) \cap \partial \hc$.

\item The closure of $\partial D_n$ in $\overline{\hc}$, denoted by $\overline{\partial D_n}$, is a $3$-dimensional cellular complex homeomorphic to the closed $3$-ball.
The cellular structure is obtained by taking the cone of $\partial\partial_\infty D_n$.
\end{enumerate}
\end{theorem}

We postpone the proof of Theorem \ref{theorem::ford_domain} to the following subsections.

%
%
\subsection{Boundary of the ideal boundary of the Ford domain}

Throughout this subsection, we work in the unit ball model of $\hc$.

\begin{proposition}
\label{proposition::homeomorphism-to-3ball}
Let $0 < \theta_1 < \theta_2 < 2\pi$. There exists a homeomorphism $\Psi = \Psi_{\theta_1,\theta_2}: I(\theta_1) \to \Ball_3$, where $\Ball_3 \coloneqq \{(x,y,z) \in \R^3 \mid x^2+y^2+z^2 \le 1\}$, with the following properties:
\begin{enumerate}[label=(\alph*)]
    \item It fixes the origin: $\Psi([0,0,1]^t) = (0,0,0)$.
    \item It maps the intersection $I(\theta_1) \cap I(\theta_2)$ onto the unit disk in the $xy$-plane:
    \[
    \Psi(I(\theta_1) \cap I(\theta_2)) = \Ball_2 \coloneqq \{(x,y,0) \in \R^3 \mid x^2+y^2 \le 1\}.
    \]
    \item It sends $I_-(\theta_2) \cap I(\theta_1)$ to the lower half-ball $\{(x,y,z) \in \Ball_3 \mid z < 0\}$ and $I_+(\theta_2) \cap I(\theta_1)$ to the upper half-ball $\{(x,y,z) \in \Ball_3 \mid z > 0\}$.
    \item For any $\theta_3 \in (0,2\pi) \setminus \{\theta_1, \theta_2\}$, the image of the triple intersection is a union of four rays:
    \[
    \Psi(I(\theta_1) \cap I(\theta_2) \cap I(\theta_3)) = \{(0,0,0)\} \cup \left\{ (x,y,0) \in \Ball_2 \ \middle| \ \arg(x+iy) \equiv \theta_3 /4 \pmod{\pi/2} \right\}.
    \]
    \item For a point $(x,y,0) \in \partial \Ball_2$:
    \begin{itemize}
        \item If $\theta_3 \in (0, \theta_1) \cup (\theta_2, 2\pi)$, then $(x,y,0) \in \Psi(\partial_\infty(I(\theta_1) \cap I(\theta_2) \cap I_-(\theta_3)))$ if and only if
        \[
        \arg(x+iy) \in \left(\frac{\theta_3}{4}, \frac{\theta_3}{4}+\frac{\pi}{2}\right) \cup \left(\frac{\theta_3}{4}+\pi, \frac{\theta_3}{4}+\frac{3\pi}{2}\right).
        \]
        \item If $\theta_3 \in (\theta_1, \theta_2)$, then $(x,y,0) \in \Psi(\partial_\infty(I(\theta_1) \cap I(\theta_2) \cap I_+(\theta_3)))$ if and only if the same condition holds.
    \end{itemize}
\end{enumerate}
\end{proposition}

\begin{proof}
The isometric sphere $I(\theta_1)$ is homeomorphic to the closed $3$-ball and contains the point $[0,0,1]^t$. By Theorem \ref{theorem::intersection_of_two}, the intersection $I(\theta_1) \cap I(\theta_2)$ is homeomorphic to a $2$-disk. Therefore, the homeomorphism $\Psi_1$ exists that satisfies properties (a), (b) and (c).

By Corollary \ref{corollary::intersection_of_three} and Corollary \ref{corollary::foliation_of_intersections}, after potentially applying the involution $(x,y,z) \mapsto (-x,y,z)$, the homeomorphism $\Psi_1$ is isotopic to a homeomorphism $\Psi_2$ that satisfies properties (a), (b), (c) and (d).
We claim that either $\Psi_2$ or the map $\Psi_3 \coloneqq \left((x,y,z)\mapsto (-y,x,z)\right) \circ \Psi_2$ satisfies all the desired properties, including (e).

To prove this claim, we parameterise $I(\theta_1) \cap I(\theta_2)$ by $(\psi_1,\psi_2)$ as in Proposition \ref{proposition::intersection-of-two-isometric-spheres} and define $X=\cot(\psi_1/2)$, $Y=\cot(\psi_2/2)$. Let $(X,0)$ be the unique point with $W(X,0) = 0$ and $X > 0$.
By Proposition \ref{proposition::intersection-of-three-isometric-spheres}, we obtain that $(X,0)$ lies on $I(\theta_1) \cap I(\theta_2) \cap I(\theta_2-\theta_1)$,
hence the image $(x_0,y_0) \coloneqq \Psi_2(X,0)$ satisfies
$\arg(x_0+iy_0) \equiv (\theta_2-\theta_1)/4 \pmod{\pi/2}$.
Moreover, we have $(X,0) \in I(\theta_1) \cap I(\theta_2) \cap I_-(\theta_3)$
if and only if $c'_{20} < 0$, i.e.,
\[
\sin\left(\dfrac{\theta_1 - \theta_3}{2}\right)
\sin\left(\dfrac{\theta_2 - \theta_3}{2}\right)
\sin\left(\dfrac{\theta_1 - \theta_2 + \theta_3}{2}\right) < 0.
\]
The claim follows from this condition.
\end{proof}

\begin{corollary}
\label{corollary::beside-f1andfn-1}
Let $0 < \theta_1 < \theta_2 < 2\pi$. Then:

\begin{enumerate}[label=(\alph*)]
\item For any $\theta_3 \in (\theta_1, \theta_2)$, the ideal boundary $\partial_\infty ( I(\theta_1) \cap I_+(\theta_2) \cap I_+(\theta_3) )$ is connected, while the ideal boundary $\partial_\infty ( I(\theta_1) \cap I_-(\theta_2) \cap I_+(\theta_3) )$ has two connected components.

\item For any $\theta_1 < \theta_3 < \theta_4 < \theta_5 < \theta_2$, we have the inclusion
\(
I_+(\theta_1) \cap I_+(\theta_2) \cap I(\theta_3) \cap I(\theta_5) 
\subset I_-(\theta_4).
\)

\item For any finite sequence of angles $\theta_1 < \theta_3 < \theta_4 < \ldots < \theta_m < \theta_2$ with $m \ge 3$, the ideal boundary
\(
\partial_\infty\left( I(\theta_1) \cap I_+(\theta_2) \cap \bigcap_{j=3}^m I_+(\theta_j) \right)
\)
is an open $2(m-1)$-gon. Its boundary consists of sub-arcs of the following circles, arranged in counterclockwise order:
\begin{align*}
\partial_\infty(I(\theta_1) \cap I(\theta_2)),\ 
\partial_\infty(I(\theta_1) \cap I(\theta_3)),\ \ldots,\ 
\partial_\infty(I(\theta_1) \cap I(\theta_m)), \\
\partial_\infty(I(\theta_1) \cap I(\theta_2)),\ 
\partial_\infty(I(\theta_1) \cap I(\theta_3)),\ \ldots,\ 
\partial_\infty(I(\theta_1) \cap I(\theta_m)).
\end{align*}
\end{enumerate}
\end{corollary}

\begin{proof}
The assertion (a) follows from the fact that $\Psi_{\theta_1,\theta_2}(I(\theta_1)\cap I_+(\theta_3))$ deforms continuously for $\theta_3 \in (\theta_1, \theta_2]$.
Applying the homeomorphism $\Psi_{\theta_3,\theta_5}$, we obtain the assertion (b) from Proposition \ref{proposition::homeomorphism-to-3ball} (e).
For (c), it suffices to show that, for $3 \le i < j < k \le m$, we always have
\(
I(\theta_i) \cap I(\theta_k) \cap I_+(\theta_1) \cap I_+(\theta_2) \subset I_-(\theta_j)
\),
which is given by the assertion (b).
\end{proof}

We therefore arrive at the boundary of the ideal boundary $\partial \partial_{\infty} D_n$.

\begin{proof}[Proof of Theorem \ref{theorem::ford_domain} (\RNum{1})]
Set $\theta_k = 2\pi k / n$ for $k = 1, \ldots, n-1$.
The intersection of the cellular structure of $\partial \partial_\infty D_n$ with $\partial_\infty I(\theta_1)$ is
$\partial_\infty\left( I(\theta_1) \cap \bigcap_{j=2}^{n-1} \overline{I_+(\theta_j)} \right)$, which is a $2(n-2)$-gon by Corollary \ref{corollary::beside-f1andfn-1} (c).
Denote this polygon by $f_1$.
As suggested by the statement of Theorem \ref{theorem::ford_domain} (\RNum{1}) (a),
we label its edges by the subscripts of the circles on which they lie.
Similarly, the only polygon in the cellular structure lying on $\partial_\infty I(\theta_{n-1})$ is also a $2(n-2)$-gon, denoted by $f_{n-1}$.

We glue $f_1$ and $f_{n-1}$ along their two pairs of opposite edges.
The union $f_1 \cup f_{n-1}$ then has $2$ boundary components, each consisting of $2(n-3)$ edges.
One such boundary component, in counterclockwise order, has edges labelled
$e_{1,n-2},\ldots,e_{1,2},e_{n-1,2},\ldots,e_{n-1,n-2}$.
We now complete the cellular structure of $\partial \partial_\infty D_n$ from the boundary of $f_1 \cup f_{n-1}$.
Applying Corollary \ref{corollary::beside-f1andfn-1} (b) to angles
$\theta_1 < \theta_i <\theta_j < \theta_k < \theta_2$, for $2 \le i < j < k \le n-1$,
we find that the missing edges lie only on $\partial_\infty\left( I(\theta_i) \cap I(\theta_{i+1}) \right)$ for some $2\le i\le n-3$.
Therefore, each boundary component of $f_1\cup f_{n-1}$ is completed by
a bi-gon when $n=4$,
two triangles when $n\ge 5$,
two triangles and $n-5$ quadrangles when $n\ge 6$.
\end{proof}

\subsection{Boundary of the Ford domain}

\begin{proof}[Proof of Theorem \ref{theorem::ford_domain} (\RNum{2})]
Set $\theta_k = 2\pi k / n$ for $k = 1, \ldots, n-1$.
Consider the cellular structure of $\partial \partial_\infty D_n$ 
from the assertion (\RNum{1}).
If a polygon $f$ lies in the ideal $2$-sphere $\partial_\infty I(\theta_k)$,
then it is the ideal boundary of a polyhedron (i.e., a side) $s$ contained in $I(\theta_k)$.
Similarly, if an edge $e$ lies in $\partial_\infty(I(\theta_j)\cap I(\theta_k))$,
then it is the ideal boundary of a ridge contained in $I(\theta_j) \cap I(\theta_k)$, since no sub-arc of $\partial_\infty (I(\theta_j) \cap I(\theta_k))$ can lie in any other intersection of two spheres.

Let $\mathcal{F}^3_\infty$ denote the set of sides in $\overline{\partial D_n}$ with non-empty ideal boundaries.
We claim that every $s \in \mathcal{F}^3_\infty$ has at least one additional vertex within $\partial \hc$.
Suppose, for contradiction, that this is not the case.
Then, some side of $\overline{\partial D_n}$ would be a prism whose upper and lower boundaries are $f_k$ and $f'_k$, respectively, for some $2\le k \le n-2$.
However, by Corollary \ref{corollary::intersection_of_three}, the point $[0,0,1]^t$ lies on every vertical edge of this prism, which is a contradiction.

Therefore, by Corollary \ref{corollary::intersection_of_four}, the only vertex in the cellular structure of $\overline{\partial D_n}$ within $\hc$ is $[0,0,1]^t$.
Hence, the cellular structure is a cone over its ideal boundary.
\end{proof}

\subsection{Ford domain centred away from the fixed torus}
\label{subsection::Ford_centred_away}

For comparison, we consider the Ford domain of $\langle E_{2\pi/n,-2\pi/n} \rangle$ for $n \ge 3$.
In this case, the centre point $q_{\infty}$ lies outside the fixed torus $\Torus = \Torus_{E_{2\pi/n, -2\pi/n}}$.

By Remark \ref{remark::not_all_R_circles}, the circles $C_{[0,0]}$ and $C_{[\sqrt{2},0]}$ are $\C$-circles,
and are also boundary orbits for every non-torsion real elliptic isometry $E_{\theta,-\theta}$.
The following proposition shows that, without the condition $q_\infty \in \Torus_f$ as in Proposition \ref{prox::intersection},
the Ford domain $D_{\langle E_{2\pi/n,-2\pi/n} \rangle}$ has a different cellular structure from $D_n$.

\begin{proposition}
\label{proposition::Ford_domain_centred_on_C_circle}
Let $0 < \theta < 2\pi$.
Then $C_{[\sqrt{2},0]} \subset I(E_{\theta,-\theta})$.
\end{proposition}

\begin{proof}
Every point $p \in C_{[\sqrt{2}, 0]}$ has a standard lift $\bf{p} = (-1, \sqrt{2}\, e^{i\phi}, 1)^t$ with $\phi \in [0, 2\pi)$.
This proposition follows from the computation
\(
|\langle \bf{p}, \bf{q}_\infty \rangle_{H_2}| = |\langle \bf{p}, E_{\theta,-\theta}^{-1}(\bf{q}_\infty) \rangle_{H_2}| = 1
\).
\end{proof}

\section{Complex hyperbolic \texorpdfstring{$(n,\infty,\infty)$}{(n,∞,∞)}-triangle groups}
\label{section::triangle}

We consider the complex hyperbolic space in the Siegel domain model.
The \emph{triangle group} of type $(n,\infty,\infty)$ has the presentation
\[
T_{n,\infty,\infty} = \left\langle \sigma_1, \sigma_2, \sigma_3
\middle|
\sigma_1^2 = \sigma_2^2 = \sigma_3^2 =
(\sigma_2\sigma_3)^n = 1
\right\rangle.
\]
A \emph{complex hyperbolic $(n,\infty,\infty)$-triangle group} is a subgroup of $\PU(H_2)$ generated by three complex reflections $I_1$, $I_2$, $I_3$ in complex geodesics $C_1$, $C_2$, $C_3$, respectively, such that $C_2$ and $C_3$ meet at angle $\pi/n$, while $C_3$ and $C_1$, and $C_1$ and $C_2$, both meet at angle $0$.
Equivalently, a complex hyperbolic $(n,\infty,\infty)$-triangle group is the image of a representation $\rho:T_{n,\infty,\infty} \rightarrow \PU(H_2)$ sending each $\sigma_i$ to a complex reflections $I_i$, such that $B \coloneqq I_2 I_3$ has order $n$ and the products $I_3 I_1$, $A \coloneqq I_1 I_2$ are unipotent.
In this case, one can further show that $B = I_2 I_3$ is real elliptic (i.e., conjugate to some $E_{\alpha,-\alpha}$)
and that $I_3 I_1$, $A = I_1 I_2$ are three-step unipotent
(see \cite[Proposition 4]{Paupert-Will-2017-involution}).
We denote the complex hyperbolic triangle group by $\Delta_\rho \coloneqq \rho(T_{n,\infty,\infty})$, or simply $\Delta(n,\infty,\infty)$ if $\rho$ is unspecified.

Two complex hyperbolic $(n,\infty,\infty)$-triangle groups $\Delta_{\rho_1}, \Delta_{\rho_2}$ are \emph{isometric} if there exists $Q \in \PU(H_2)$ such that $\rho_1(g) = Q^{-1} \rho_2(g) Q$ for all $g \in T_{n,\infty,\infty}$.
By \cite[Proposition 1]{Anna2005}, a complex hyperbolic $(n,\infty,\infty)$-triangle groups is determined up to isometry by the \emph{angular invariant}
\begin{equation*}
\mathcal{A} \coloneqq
\arg\left(
\langle n_3,n_2 \rangle
\langle n_1,n_3 \rangle
\langle n_2,n_1 \rangle
\right)
\in [0, 2\pi] \mod 2\pi,
\end{equation*}
where $n_i$ is a polar vector of $C_i$.
The moduli space of complex hyperbolic $(n,\infty,\infty)$-triangle groups up to isometry,
denoted by $\SSS(n,\infty,\infty)$, is homeomorphic to $\mathbb{S}^1$ and parameterised by $\mathcal{A}$.

Let $\Delta_{\rho_{\mathcal{A}}}$ be a representative of the isometry class parameterised by $\mathcal{A} \in [0, 2\pi)$.
We call $\rho_{\mathcal{A}}$ a \emph{discrete embedding} if the representation is discrete and faithful.
The moduli space of discrete embeddings
is denoted by $\MM(n,\infty,\infty) \subset \SSS(n,\infty,\infty)$.
An $\R$-Fuchsian representation must be a discrete embedding, so $\pi \in \MM(n,\infty,\infty)$,
whereas a $\C$-Fuchsian representation cannot be a discrete embedding, so $0, 2\pi \not\in \MM(n,\infty,\infty)$.

We prove Corollary \ref{colx::triangle_group} in this section.
Since complex hyperbolic triangle groups with angular invariant $\mathcal{A}$ and $2\pi - \mathcal{A}$ are conjugate via an
anti-holomorphic isometry,
we restrict to $\mathcal{A} \in (0,\pi]$.

\subsection{Parameterisation}

For $n \ge 3$ and $\mathcal{A} \in (0,\pi]$, we fix a representative $\Delta_{\rho_{\mathcal{A}}}$ by defining a homomorphism $\rho_{\mathcal{A}} : T_{n,\infty,\infty} \to \PU(H_2)$ via $\rho_{\mathcal{A}}(\sigma_i) = I_i$, where
\begin{align*}
I_1 &=
\begin{bmatrix}
-1 & 0 & 0 \\
0 & 1 & 0 \\
0 & 0 & -1
\end{bmatrix},
\quad
I_2 =
\begin{bmatrix}
-1 & -2 & 2 \\
0  &  1 & -2 \\
0  &  0 & -1
\end{bmatrix}, \\
I_3 &=
\begin{bmatrix}
-1 & 0 & 0 \\
2 - 2 \cos\left(\frac{\pi}{n}\right) e^{-i \mathcal{A}} & 1 & 0 \\
2 + 2 \cos\left(\frac{\pi}{n}\right) \Bigl( \cos\left(\frac{\pi}{n}\right) - 2 \cos\mathcal{A} \Bigr) & 2 - 2 \cos\left(\frac{\pi}{n}\right) e^{i \mathcal{A}} & -1
\end{bmatrix}.
\end{align*}
For \(i = 1,2,3\), suppose that \(I_i\) lifts to the matrix \(\bf{I}_i \in \mathrm{SU}(H_2)\) given by the same expression.
The isometry $I_i$ is the complex reflection in a complex geodesic $C_i$.
Then the polar vectors are $n_1 = (0,1,0)^t$, $n_2 = (1,-1,0)^t$ and $n_3 = \bigl(0,\frac{1}{1-\cos(\pi/n) e^{i \mathcal{A}}},1\bigr)^t$.
Set $A \coloneqq I_1 I_2$, which is the Heisenberg translation $T_{[-2,0]}$.
Set $B \coloneqq I_2 I_3$ and take the lift $\bf{B} \in \SU(H_2)$ to be the product of the lifts $\bf{I}_2$ and $\bf{I}_3$. The characteristic polynomial of $\bf{B}$ in $\lambda$ is $(\lambda - 1)(\lambda^2 - 2\mu \lambda + 1)$ with $\mu = \cos(2\pi/n)$, so its eigenvalues are $1$, $e^{i \cdot 2\pi/n}$, and $e^{-i \cdot 2\pi/n}$.
Hence $B$ is real elliptic. We also obtain the recurrence $\bf{B}^{k+2} - 2\mu \bf{B}^{k+1} + \bf{B}^k = v$ for all $k \ge 0$, where $v = \bf{B}^2 - 2\mu \bf{B} + I$. From this it follows that
\begin{equation*}
\bf{B}^k =
\cos(2\pi k / n) I +
\dfrac{\sin(2\pi k / n)}{\sin(2\pi / n)} \bigl( \bf{B} - \cos(2\pi / n) I - \tfrac{v}{2} \bigr) +
\dfrac{1 - \cos(2\pi k / n)}{4 \sin^2(\pi / n)} v.
\end{equation*}
One can therefore verify that $B$ has order $n$ and
\[
\sqrt{\dfrac{2}{|(\bf{B}^k)_{3,1}|}} =
\dfrac{\sqrt{2} \sin(\pi/n)}{\sqrt{3 + \cos(2\pi/n) - 4 \cos(\pi/n) \cos\mathcal{A}}} \cdot
\dfrac{1}{\sin(k\pi/n)} .
\]
We now describe the Ford domain $D_{\langle A^k B A^{-k} \rangle}$ for $k \in \mathbb{Z}$.

\begin{corollary}
\label{corollary::ford_domain_ABA}
For every $k \in \Z$, 
there exists an isometry that maps
the Ford domain $D_{\langle A^k B A^{-k} \rangle}$
into the polytope $D_n$.
Under this map, the intersection
$I(A^k B^{-j} A^{-k}) \cap D_{\langle A^k B A^{-k} \rangle}$
is sent into
$I(2\pi j/n) \cap D_n$,
for every $j \in \{1,\cdots, n-1\}$.
\end{corollary}

\begin{proof}
The fixed point of $B$ is $\fp_B = \bigl[\frac{1}{-1+\cos(\pi/n) e^{-i\mathcal{A}}}, 1, 1\bigr]^t$. Since $\mathbb{A}(\fp_B, q_{\infty}, B(q_{\infty})) = 0$, we have $q_{\infty} \in \Torus_B$ (cf. the proof of Proposition~\ref{proposition::torus_2}). Moreover, the fixed point $\fp_A$ coincides with $q_{\infty}$, which yields $q_{\infty} \in \Torus_{A^k B A^{-k}}$ for every $k$; this is a consequence of Proposition~\ref{proposition::standard_model_Dn}.
\end{proof}

\subsection{Applying the Poincar\'e polyhedron theorem}

Consider the index-$2$ subgroup $\langle A, B \rangle$ of $\Delta_{\rho_\mathcal{A}} = \langle I_1, I_2, I_3 \rangle$.
Set $I_{j,k} \coloneqq I(A^k B^j A^{-k}) = S_{\bf{c}_{j,k}}(\bf{r}_{j,k})$ for $j = 1, \dots, n-1$ and $k \in \Z$.
The pair of spheres \((I_{1,k}, I_{n-1,k+1})\) is isometric to \((I(I_1 I_3), I(I_3 I_1))\). Hence, for every \(k\), the unique common fixed point of \(A^k I_1 A^{-k}\) and \(A^k I_3  A^{-k}\) lies in \(I_{1,k} \cap I_{n-1,k+1}\). We therefore present the following sufficient condition for \(\langle I_1, I_2, I_3\rangle\) to be discrete and faithful.

\begin{proposition}
\label{proposition::apply_poincare}
Let \(n \ge 3\) and \(\mathcal{A} \in (0, \pi]\). Suppose that for all \(j, j' \in \{1,\ldots,n-1\}\) and all integers \(k < k'\), the isometric spheres \(I_{j,k}\) and \(I_{j',k'}\) are either disjoint or tangent at a single ideal point, with tangency occurring precisely when  
\[
(k', j', j) = (k+1, n-1, 1).
\]  
Then the group \(\langle A, B \rangle\) is discrete and has the presentation \(\langle A, B \mid B^n = 1 \rangle\).
\end{proposition}

\begin{proof}
Define
\[
\mathcal{D} = \bigcap\limits_{j=1}^{n-1} \bigcap\limits_{k \in \Z} \overline{I_+(A^k B^j A^{-k})} \setminus \Proj(V_{H_2}^0)
\]
in $\hc$.
The closure of $\partial \mathcal{D}$ in $\overline{\hc}$, denoted by $\overline{\partial \mathcal{D}}$, is contained in the union of the isometric spheres $I_{j,k}$
and admits a natural decomposition
$\overline{\partial D} = \bigcup_{k \in \Z} \mathcal{X}_k$,
where
\[
\mathcal{X}_k \coloneqq \overline{\partial \mathcal{D}} \cap \bigcup_{j} I_{j,k}
=
\overline{\partial D_{\langle A^k B A^{-k} \rangle}}.
\]
By Proposition \ref{proposition::spheres_EEE} and Theorem \ref{theorem::ford_domain},
each $\mathcal{X}_k$ is mapped onto the polytope $\overline{\partial D_n}$ via a homeomorphism $\Phi_k: \mathcal{X}_k \rightarrow \overline{\partial D_n}$, which sends the intersection with $I(A^k B^{-j} A^{-k})$ into the intersection with $I(2\pi j/n)$ for every $j=1,\ldots,n-1$.
The cellular structure of $\mathcal{X}_k$ is inherited from that of
$\overline{\partial D_n}$, 
as given by Corollary \ref{corollary::ford_domain_ABA}.
Moreover, under the hypotheses of this proposition,
for $k \neq k'$,
the components $\mathcal{X}_k$ and $\mathcal{X}_{k'}$ intersect only when $|k-k'| = 1$.
In this case, they are tangent at a point lying on the ideal boundary.
Hence, the closed domain $\mathcal{D}$ is a $4$-dimensional polytope with a known cellular structure.

We apply a version of the Poincar\'e polyhedron theorem for $\mathcal{D}$ with an infinite stabilizer (see \cite[Theorem 5.1]{PW2017}).
For the general principle of this theorem, we refer to Beardon \cite[Section 9.6]{Beardon-1983}; see also \cite{Mostow1980,DPP2016}.
The following statement verifies that all hypotheses of the Poincar\'e polyhedron theorem are satisfied.

\begin{enumerate}[label=(\arabic*)]
\item \textbf{Cellular structure:}
The cells in $\mathcal{D}$ of dimension $0$, $1$, $2$ and $3$ are called vertices, edges, ridges and sides, respectively.
We denote by $C_k(\mathcal{D})$ the set of $k$-dimensional cells in $\mathcal{D}$.
The ideal boundary $\partial_{\infty} \mathcal{D}$ is a $3$-dimensional polytope
and we denote by $IC_k(\mathcal{D})$ the set of $k$-dimensional cells in the ideal boundary.
In particular, elements of $IC_0(\mathcal{D})$ are called \emph{ideal vertices}.
The stabilizer of $\mathcal{D}$ is $\Upsilon \coloneqq \langle A \rangle$.

\item
\textbf{The side pairing:}
For each side $s \in C_3(\mathcal{D})$ contained in the isometric sphere $I_{j,k}$, we assign the isometry $\sigma(s) = A^k B^j A^{-k}$.
The map $\sigma: C_3(\mathcal{D}) \rightarrow \PU(H_2)$ is called the \emph{side pairing}.
The stabilizer $\Upsilon$ is compatible with the side pairing, i.e., $\sigma(A^k(s)) = A^k \sigma(s) A^{-k}$ for all $s \in C_3(\mathcal{D})$ and $k \in \Z$.

\item
\textbf{Ridges and cycle relations:}
Let $r \in C_2(\mathcal{D})$ be a ridge defined as the intersection of two sides $s_{\text{head}}, s_{\text{tail}} \in C_3(\mathcal{D})$.
Since the isometry $A' \coloneqq \sigma(s_{\text{head}})$ preserves the cellular structure of $\mathcal{D}$,
the image $A'(r)$ is another ridge, denoted by $r'$.
Suppose that $A'$ sends $s_{\text{head}}$ to $s'_{\text{tail}}$ and
let $s'_{\text{head}}$ be the other side containing $r'$.
Then $r'$ is the intersection of $s'_{\text{head}}$ and $s'_{\text{tail}}$.
Countinuing this process yields a ridge cycle,
where the first triple and the last triple are required to satisfy
$(r',s'_{\text{head}},s'_{\text{tail}}) = (P r,P s_{\text{head}},P s_{\text{tail}})$
for some (unique) $P \in \Upsilon$.

Every ridge $r \in C_2(\mathcal{D})$ is mapped by some $\Phi_k$ to a ridge of $D_n$.
By Theorem \ref{theorem::ford_domain} (\RNum{2}), there is a one-to-one correspondence between the set of ridges $C_2(D_n)$ and the set of edges $C_1(\partial \partial_\infty D_n)$.
Therefore, by Theorem \ref{theorem::ford_domain} (\RNum{1}), the set $C_2(\mathcal{D})$ consists of
two ridges contained by $I_{-1,k} \cap I_{1,k}$,
two ridges contained by $I_{-1,k} \cap I_{j,k}$ for each $j=2,\ldots,n-2$,
two ridges contained by $I_{1,k} \cap I_{j,k}$ for each $j=2,\ldots,n-2$
and two ridges contained by $I_{j,k} \cap I_{j+1,k}$ for each $j=2,\ldots,n-3$,
for every $k \in \Z$.

In our setting, all cycles of intersections $I(g_1) \cap I(g_2)$ that contain at least one ridge under the action of $g_1$ and $g_2$ are given in the following graph:
\begin{center}
\begin{tikzcd}
I_{-1,k} \cap I_{-j,k} \arrow[rr, "A^kB^{-1}A^{-k}"] & & I_{1,k} \cap I_{-(j-1),k} \arrow[ld, "A^k B^{-(j-1)} A^{-k}"] \\
& I_{j,k} \cap I_{j-1,k} \arrow[lu, "A^kB^kA^{-k}"]
\end{tikzcd}
\end{center}
where $j=2,\ldots,n-2,n-1$ and $k \in \Z$.
Hence, every ridge cycle of $\mathcal{D}$ is a triangle with trivial cycle transformation.

\item
\textbf{Ideal vertices and consistent horoballs:}
For every ideal vertex $v \in IC_0(\mathcal{D})$, we assume that $\sigma(s)(v) \in IC_0(\mathcal{D})$ for each side $s\in C_3(\mathcal{D})$ containing $v$.
The Poincar\'e polyhedron theorem requires a system of consistent horoballs: a sufficiently small horoball $H_v$ at each $v \in IC_0(\mathcal{D})$ such that $\sigma(s)(H_v) = H_{\sigma(s)(v)}$.

Every ideal vertex $v \in IC_0(\mathcal{D})$ either lies in a single component $\mathcal{X}_k \subset \overline{\partial \mathcal{D}}$ or occurs as the ideal point of tangency of two adjacent components.
In the former case, the ideal vertex lies in the intersection of three isometric spheres.
All cycles of intersections $I(g_1) \cap I(g_2) \cap I(g_3)$ that contain at least one ideal vertex under the action of $g_1$, $g_2$ and $g_3$ are given in the following graph:
\begin{center}
\begin{tikzcd}
{I_{1,k} \cap I_{j+1,k} \cap I_{j+2,k}}
\arrow[rr,"A^kB^{j+1}A^{-k}"]
\arrow[rd,"A^kB^{j+2}A^{-k}"]
&
&
{I_{1,k}\cap I_{-j,k}\cap I_{-(j+1),k}}
\\
&
{I_{-1,k}\cap I_{-(j+1),k}\cap I_{-(j+2),k}}
\arrow[ru,"A^kB^{-1}A^{-k}"]
&
\\
&
{I_{-1,k}\cap I_{j,k}\cap I_{j+1,k}}
\arrow[u,"A^kB^{j+1}A^{-k}"]
\arrow[luu,"A^kB^{-1}A^{-k}",bend left]
\arrow[ruu,"A^kB^jA^{-k}"',bend right]
&
\end{tikzcd}
\end{center}
where $j=1,\ldots,n-3$ and $k \in \Z$.
From this, all cycles of ideal vertices at the intersection of three isometric spheres yield trivial cycle transformations.

Every ideal point of tangency $v$ is the fixed point $\fp_{A^k(AB)A^{-k}}$
lying on $I_{1,k} \cap I_{-1,k+1}$ for some $k \in \Z$.
Applying the side pairing isometries $A^kBA^{-k}$ and $A^{k+1}B^{-1}A^{k+1}$ to the point of tangency, we obtain an infinite graph:
\[
\longrightarrow
\fp_{A^{k-1}(AB)A^{-(k-1)}}
\xrightarrow{A^k B^{-1} A^{-k}}
\fp_{A^k(AB)A^{-k}}
\xrightarrow{A^{k+1} B^{-1} A^{-(k+1)}}
\fp_{A^{k+1}(AB)A^{-(k+1)}}
\longrightarrow.
\]
Hence, we can define a system of consistent horoballs.
\end{enumerate}

The group $\Gamma$ generated by $\Upsilon$ and the side pairing isometries is exactly $\Gamma = \langle A, B \rangle$.
By the Poincaré polyhedron theorem, the images of $\mathcal{D}$ under the cosets of $\langle A \rangle$ in $\langle A, B \rangle$ tessellate $\hc$, the group $\langle A, B \rangle$ is discrete and it has the presentation $\langle A, B \mid B^n = 1\rangle$, as desired.
\end{proof}

We thus arrive at Corollary \ref{colx::triangle_group}, which we restate for clarity.

\begin{theorem}
\label{theorem::apply_poincare}
Let \(n \ge 3\) and \(\mathcal{A} \in (0, \pi]\).
Suppose that for all \(j, j' \in \{1,\ldots,n-1\}\) and all \(k \in \mathbb{Z}_{>0}\) with \(k \le 2/\sin(\pi/n)\), the isometric spheres \(I_{j',0}\) and \(I_{j,k}\) are either disjoint or tangent at a single ideal point, and tangency occurs precisely when \((j', j, k) = (1, n-1, 1)\). Then the group \(\langle A, B \rangle\) is discrete and has the presentation \(\langle A, B \mid B^n = 1 \rangle\).
\end{theorem} 

\begin{proof}[Proof of Theorem \ref{theorem::apply_poincare}]
We first claim that \(\mathcal{A} \ge \pi/n\). Consider the element \(g = I_1 I_2 I_3 I_2\), whose trace is
\(
\operatorname{tr}(g) = 3 + 16 \cos(\pi/n)\bigl(\cos(\pi/n) - \cos \mathcal{A}\bigr)
\).
Assume that \(\mathcal{A} < \pi/n\). Then \(g\) is regular elliptic. Therefore, the pair of spheres \((I_{n-1,0}, I_{1,1})\), which is isometric to \((I(g), I(g^{-1}))\), cannot be disjoint, which is a contradiction.

For every $j \in \{1,\dots,n-1\}$, the radius $\bf{r}_j \coloneqq \bf{r}_{j,k}$ is independent of $k \in \mathbb{Z}$ and satisfies
\[
\bf{r}_j = \bf{r}_{n-j} = \sqrt{2/|(\bf{B}^j)_{3,1}|} \le \sqrt{2/|\bf{B}_{3,1}|} \le 1/\sin(\pi/n).
\]
For all distinct $j, j' \in \{1,\dots,n-1\}$, the nonempty intersection $I_{j,0} \cap I_{j',0} \neq \emptyset$ implies
\[
d_{\text{Cyg}}(\bf{c}_{j,0}, \bf{c}_{j',0}) \le \bf{r}_j + \bf{r}_{j'} \le 2 / {\sin(\pi/n)}.
\]
We now examine the Cygan distance between $\bf{c}_{j',0}$ and $\bf{c}_{j,k}$.
For $j',j \in \{1,\dots,n-1\}$ and $k \in \mathbb{Z}$, the triangle inequality gives
\[
d_{\text{Cyg}}(\bf{c}_{j',0}, \bf{c}_{j,k}) - (\bf{r}_{j'} + \bf{r}_j)
\ge d_{\text{Cyg}}(\bf{c}_{j,0}, \bf{c}_{j,k}) - \bigl(d_{\text{Cyg}}(\bf{c}_{j,0}, \bf{c}_{j',0}) + \bf{r}_{j'} + \bf{r}_j\bigr)
\ge 2|k| - 4 / {\sin(\pi/n)},
\]
which is positive when $|k| > 2/\sin(\pi/n)$.
Consequently $I_{j',0} \cap I_{j,k} = \emptyset$.
The theorem follows from Proposition~\ref{proposition::apply_poincare} together with the
isometry between $(I_{j,k}, I_{j',k'})$ and $(I_{j,k-k'}, I_{j',0})$.
\end{proof}

\sloppy
\printbibliography[
    heading=bibintoc,
    title={References}
]

\end{document}